\documentclass[11pt, reqno]{amsart}
\usepackage{txfonts}
\usepackage{amsmath,amssymb,amsthm,mathrsfs,mathtools,bm,xcolor,multirow,pbox}
\usepackage{mathrsfs}
\usepackage{algorithm,algorithmic,float}
\usepackage{graphicx,color,framed,tikz,caption,subcaption}
\usepackage{enumitem}
\usepackage{booktabs}
\setlist{leftmargin=9mm}
\usepackage[colorlinks,linkcolor=black,citecolor=black,urlcolor=black]{hyperref}
\allowdisplaybreaks[4]
\numberwithin{equation}{section}
\newcommand{\N}{\mathbb{N}}
\newcommand{\R}{\mathbb{R}}

\newcommand{\pnorm}[2]{\lVert #1\rVert_{#2}}

\newcommand{\abs}[1]{\lvert#1\rvert}
\newcommand{\bigabs}[1]{\big\lvert#1\big\rvert}
\newcommand{\biggabs}[1]{\bigg\lvert#1\bigg\rvert}
\newcommand{\iprod}[2]{\langle#1,#2\rangle}

\renewcommand{\epsilon}{\varepsilon}

\renewcommand{\d}[1]{\mathrm{d}#1}

\newcommand{\bigop}{\mathcal{O}_{\mathbf{P}}}

\newcommand{\smallo}{\mathfrak{o}}
\newcommand{\bigo}{\mathcal{O}}

\DeclareMathOperator{\E}{\mathbb{E}}
\DeclareMathOperator{\Prob}{\mathbb{P}}

\DeclareMathOperator{\sign}{\texttt{sgn}}

\DeclareMathOperator{\tr}{tr}
\DeclareMathOperator{\var}{Var}
\DeclareMathOperator{\cov}{Cov}

\DeclareMathOperator{\op}{op}
\DeclareMathOperator{\TV}{TV}

\DeclareMathOperator{\err}{\texttt{err}}

\DeclareMathOperator*{\argmax}{arg\,max\,}
\DeclareMathOperator*{\argmin}{arg\,min\,}

\newcommand{\beq}{\begin{equation}}
\newcommand{\eeq}{\end{equation}}
\newcommand{\beqa}{\begin{equation} \begin{aligned}}
\newcommand{\eeqa}{\end{aligned} \end{equation}}
\newcommand{\beqas}{\begin{equation*} \begin{aligned}}
\newcommand{\eeqas}{\end{aligned} \end{equation*}}

\newcommand{\bit}{\begin{itemize}}
	\newcommand{\eit}{\end{itemize}}
\newcommand{\bmat}{\begin{bmatrix}}
	\newcommand{\emat}{\end{bmatrix}}

\theoremstyle{definition}\newtheorem{problem}{Problem}[section]
\theoremstyle{definition}\newtheorem{definition}[problem]{Definition}
\theoremstyle{definition}
\theoremstyle{remark}\newtheorem{assumption}{Assumption}

\theoremstyle{remark}\newtheorem{remark}{Remark}
\theoremstyle{definition}\newtheorem{example}[problem]{Example}
\theoremstyle{plain}\newtheorem{theorem}[problem]{Theorem}
\theoremstyle{plain}
\theoremstyle{plain}\newtheorem{lemma}[problem]{Lemma}
\theoremstyle{plain}\newtheorem{proposition}[problem]{Proposition}
\theoremstyle{plain}\newtheorem{corollary}[problem]{Corollary}
\theoremstyle{plain}

\AtBeginDocument{%
	\def\MR#1{}
}

\begin{document}

\title[Empirical Bayes for correlated Gaussian sequence model]{Empirical Bayes for correlated Gaussian sequence model}

\author[Q. Han]{Qiyang Han}

\address[Q. Han]{
Department of Statistics, Rutgers University, Piscataway, NJ 08854, USA.
}
\email{qh85@stat.rutgers.edu}

\author[C.-H. Zhang]{Cun-Hui Zhang}

\address[C.-H. Zhang]{
	Department of Statistics, Rutgers University, Piscataway, NJ 08854, USA.
}
\email{czhang@stat.rutgers.edu}

\date{\today}
\begin{abstract}
Empirical Bayes methods are among the most widely used statistical methods for large-scale inference. A central paradigm is the nonparametric maximum likelihood estimator, whose theoretical guarantees are by now well understood for the independent Gaussian sequence model.

In this paper, we study empirical Bayes estimation from observations in a correlated Gaussian sequence model with a possibly arbitrary dependence structure. We show that the maximum Composite Marginal Likelihood (CML) estimator, 
which ignores all correlations among the observations in the likelihood, converges in weighted Hellinger distance at the rate $n_\ast^{-1/2}$, modulo logarithmic factors, where $n_\ast=n/\pnorm{\texttt{Cor}_{\Sigma_0}}{\op}$ is the `effective sample size' determined solely by the number of observations $n$ and the spectral radius of the correlation matrix $\texttt{Cor}_{\Sigma_0}$ of the Gaussian observations. A complementary minimax lower bound shows that $n_\ast$ indeed serves as the right complexity measure, and that the CML estimator is nearly rate optimal under general dependence.

As an illustration of the broad applicability of the CML method, we consider two concrete settings for prior estimation. In the first, we consider Bayesian linear regression in high dimensions, where the signal prior is estimated via the CML method applied to the generalized least squares estimator. In the second, we consider the more challenging Bayesian nonlinear single-index model, where prior estimation is made possible by applying the CML method to a one-step debiased gradient descent. In both applications, although the full likelihood landscape can be arbitrarily complicated and intractable in high dimensions, our CML method is facilitated by exploiting the high-dimensional distribution of the auxiliary statistics through an approximate correlated Gaussian sequence model.

The key ingredient in the proof of our results is a sharp local maximal inequality for the log composite marginal likelihood process under arbitrarily dependent Gaussian observations. In contrast to standard empirical process methods, we prove this inequality by leveraging a recent version of the geometric Brascamp-Lieb inequality for Gaussian measures.
\end{abstract}
\maketitle

\setcounter{tocdepth}{1}
\tableofcontents

\sloppy

\section{Introduction}

\subsection{Empirical Bayes and the NPMLE}

Empirical Bayes methods are among the most widely used statistical methods for large-scale inference. In the classical compound decision formulation, one observes a collection of related experiments
\begin{align*}
X_j\mid \theta_j\sim f(\cdot\mid \theta_j),\qquad j\in[n],
\end{align*}
and attempts to use the full ensemble $X_{[n]}$ to learn the population structure of the latent effects $\theta_1,\ldots,\theta_n$. If $G_\ast$ denotes the oracle prior, or, in the non-random compound formulation, the empirical distribution $n^{-1}\sum_i\delta_{\theta_i}$, then the oracle Bayes rule $\delta_{G_\ast}$ under a loss function $\mathsf{L}$ provides the benchmark oracle Bayes risk $R(G_\ast)=\E_{G_\ast} \mathsf{L}(\delta_{G_\ast}(X),\theta)$. The fundamental empirical Bayes insight of Robbins is that, in large compound problems, estimating $G_\ast$ can lead to procedures whose average risk approaches the oracle Bayes risk. This principle underlies much of the modern methodology for shrinkage, sparse normal means, multiple testing, large-scale confidence assessment, and compound decision theory; see, e.g., \cite{robbins1951asymptotically,robbins1956empirical,stein1956inadmissibility,efron1972limiting,jiang2009general,zhang2009generalized,efron2010large,efron2014two,koenker2014convex,koenker2025empirical}.

A particularly popular approach to estimating the prior $G_\ast$ is the nonparametric maximum likelihood estimator (NPMLE) \cite{kiefer1956consistency}. In the canonical homoskedastic Gaussian sequence model
\begin{align}\label{def:model_ind}
X_j=\theta_{j}+\mathsf{Z}_j,
\qquad \mathsf{Z}_j\stackrel{\mathrm{i.i.d}}{\sim} \mathcal{N}(0,1),
\quad \theta_j\stackrel{\mathrm{i.i.d.}}{\sim}G_\ast,
\end{align}
the NPMLE estimates $G_\ast$ by
\begin{align}\label{def:NPMLE_intro}
\hat G_n\in\argmax_{G\in\mathscr G}
\frac1n\sum_{j \in [n]}\log \varphi_{G;1}(X_j),
\quad
\varphi_{G;\sigma}(x)=\int \varphi_\sigma(x-u)\,G(\d u).
\end{align}
Here $\mathscr{G}$ is the class of distribution functions on $\R$, and $\varphi_\sigma(u)\equiv (2\pi \sigma^2)^{-1/2}e^{-u^2/(2\sigma^2)}$ is the Lebesgue density function corresponding to $\mathcal{N}(0,\sigma^2)$. We usually write $\varphi_G\equiv \varphi_{G;1}$ for notational simplicity.

The NPMLE in \eqref{def:NPMLE_intro} has a long history and has received significant renewed interest recently; see, among others, \cite{kiefer1956consistency,lindsay1983geometry,jiang2009general,zhang2009generalized,koenker2014convex,dicker2016high,gu2016problem,feng2018approximate,jiang2020general,kim2020fast,polyanskiy2020self,saha2020nonparametric,deb2022two,ignatiadis2025empirical,ghosh2025stein,soloff2025multivariate,shen2026poisson,chen2026empirical,ignatiadis2026compound,kim2026empirical,chen2026normal}. It is well known \cite{ghosal2001entropies,zhang2009generalized,jiang2009general} that, under the model \eqref{def:model_ind}, the NPMLE \eqref{def:NPMLE_intro} enjoys desirable statistical properties, in that it converges at an almost parametric rate in Hellinger distance $d_H$ for mixture density estimation:
\begin{align}\label{eqn:rate_ind_intro}
d_H\big(\varphi_{\hat{G}_n}, \varphi_{G_\ast}\big) = \bigop(n^{-1/2}\mathrm{polylog}\, n).
\end{align}
Unfortunately, the independent observation scheme \eqref{def:model_ind} underpinning the validity of \eqref{eqn:rate_ind_intro} remains a major bottleneck to the broad practical usefulness of the NPMLE. Indeed, many modern applications call for the use of empirical Bayes methods in settings where the observations are dependent. For instance, regression coefficients estimated by ordinary least squares, not to mention those arising in many more challenging high-dimensional regression settings, already exhibit substantial nontrivial correlation. As another example, \cite{gu2022ranking} fit a Bradley–Terry model and apply the NPMLE to the components of the maximum likelihood estimator that are highly correlated to each other.

\subsection{CML estimator under correlated Gaussian sequence model}

In this paper, we consider a general correlated Gaussian sequence model
\begin{align}\label{eq:intro_corr_sequence}
U=\beta_0+\Sigma_0^{1/2}\mathsf Z_n\in \R^n,
\qquad \beta_{0,j}\stackrel{\mathrm{i.i.d.}}{\sim}G_0,
\quad \mathsf Z_n\sim\mathcal N(0,I_n),
\end{align}
where $\Sigma_0$ is an arbitrary covariance matrix. Similar to the setting above, the main statistical goal is to estimate the prior $G_0$ based on the observations $\{U_j\}$.

We propose to estimate $G_0$ by
\begin{align}\label{eq:intro_mple}
\hat G_n\in \argmax_{G\in\mathscr G}
\frac1n\sum_{j \in [n]}\log \varphi_{G;\sigma_{0,j}}(U_j),
\end{align}
where $\sigma_{0,j}^2=(\Sigma_0)_{jj}$. It is easy to see that the estimator \eqref{eq:intro_mple} deliberately ignores the correlations among the coordinates of $U$ in the full likelihood, and, in fact, can be viewed as applying the NPMLE \eqref{def:NPMLE_intro} by `pretending' that $\{U_j\}$ are independent of one another. In this sense, it is a type of `independence likelihood' \cite{chandler2007inference}  or `composite likelihood' \cite{varin2011overview}. Therefore, we call $\hat{G}_n$ in \eqref{eq:intro_mple} the \emph{maximum composite marginal likelihood (CML) estimator}. Clearly, computation of the CML estimator \eqref{eq:intro_mple} is almost as easy as that of the NPMLE \eqref{def:NPMLE_intro}.

As the dependence structure of the observations $\{U_j\}$ can be fairly arbitrary, there appears to be no a priori reason to believe that the CML estimator $\hat{G}_n$ should estimate the unknown $G_0$ well. Surprisingly, we prove in Theorem \ref{thm:MML_rate} that the CML estimator \eqref{eq:intro_mple} converges for mixture density estimation in a weighted Hellinger distance $\mathfrak d_{H;\sigma_{0,[n]}}$ (formally defined in \eqref{def:averaged_hellinger} below):
\begin{align}\label{eqn:rate_intro}
\mathfrak d_{H;\sigma_{0,[n]}}(\hat G_n,G_0)
=\bigop(n_\ast^{-1/2}\mathrm{polylog}\, n).
\end{align}
Here with $\texttt{Cor}_{\Sigma_0}\equiv \mathfrak{D}_{\Sigma_0}^{-1/2}\Sigma_0\mathfrak{D}_{\Sigma_0}^{-1/2}$, $\mathfrak{D}_{\Sigma_0}=\mathrm{diag}(\Sigma_0)$, denoting the correlation matrix of the standardized Gaussian noise, the `\emph{effective sample size}' $n_\ast$ is defined by
\begin{align}\label{eq:intro_eff_size}
n_\ast
\equiv
\frac{n}{\pnorm{ \texttt{Cor}_{\Sigma_0} }{\op}}.
\end{align}
Comparing \eqref{eqn:rate_ind_intro} and \eqref{eqn:rate_intro}, the CML estimator \eqref{eq:intro_mple} remains valid and converges essentially at the same rate as the NPMLE \eqref{def:NPMLE_intro}, with the sample size replaced by the effective sample size $n_\ast$, despite the possible arbitrary dependence among the Gaussian observations ${U_j}$.

The notion of effective sample size \eqref{eq:intro_eff_size} has a natural theoretical interpretation: in the independent case $\Sigma_0 = I_n$, the effective sample size is $n_\ast = n$, whereas in the perfectly correlated case $\Sigma_0 = \bm{1}\bm{1}^\top$, the effective sample size becomes $n_\ast = 1$. More importantly, a complementary minimax lower bound in Proposition \ref{prop:minimax_lower_effective_sample_size} shows that the effective sample size $n_\ast$ indeed serves as the right complexity measure, and that the rate $n_\ast^{-1/2}$ in \eqref{eqn:rate_intro} is optimal, up to logarithmic factors, in a minimax sense.

In addition to convergence in the weighted Hellinger distance in \eqref{eqn:rate_intro}, we also prove convergence of $\hat{G}_n$ to $G_0$ in Wasserstein distance
\begin{align}\label{eqn:rate_W_intro}
\mathsf{W}_p(\hat{G}_n,G_0)=\bigop((\log n_\ast)^{-1/2}).
\end{align}
{The logarithmic rate (\ref{eqn:rate_W_intro}) is known to be optimal already in the independent case (\ref{def:model_ind}) for general $G_0$'s with unbounded supports, cf. \cite{dedecker2013minimax}.}

As a direct application of our theory in (\ref{eqn:rate_intro}) and (\ref{eqn:rate_W_intro}):
\begin{enumerate}
	\item We construct empirical Bayes credible intervals based on marginal posterior distributions, and prove their validity in an averaged sense.
	\item We show that the marginal empirical Bayes regret, with the marginal Bayes rule as the oracle, converges at an optimal parametric rate in terms of the effective sample size $n_\ast$. 
\end{enumerate}
These applications will be detailed in Sections \ref{subsec:credible_consequence} and \ref{subsec:marginal_regret}.

We note that this paper is parallel to and contemporaneous with \cite{zhang2026empirical}, which focuses on a homoscedastic pairwise Gaussian copula model in the compound decision setting under stronger assumptions on the support of the unknown means. Moreover, the analytical approaches of this paper and \cite{zhang2026empirical} are fundamentally different: \cite{zhang2026empirical} is based on a variance inequality of \cite{guo2022extreme} that leads to second moment estimates, whereas the present work, as will become clear below, is based on a geometric Brascamp-Lieb inequality that leads to the sharp high probability estimate in (\ref{eqn:rate_intro}) with optimal rates.

\subsection{Applications to two concrete settings}

We now illustrate two non-trivial applications of the CML method for prior estimation in more concrete settings.

In the first application, we consider the Bayesian Gaussian linear model
\begin{align}\label{eq:intro_gls_model}
Y=A\mu_\ast+\xi,
\qquad \mu_{\ast,j}\stackrel{\mathrm{i.i.d.}}{\sim}G_\ast,\quad
\xi\mid A\sim\mathcal N(0,\tau_\ast^2\Omega),
\end{align}
where $A\in\R^{m\times n}$ has full column rank and $\Omega$ is known up to scale. The generalized least squares (GLS) estimator satisfies, conditionally on $A$,
\begin{align}\label{eq:intro_gls_sequence}
\hat\mu_{\mathrm{gls}}
=\mu_\ast+\tau_\ast Q_A^{-1/2}\mathsf Z_n,
\qquad
Q_A=A^\top\Omega^{-1}A,\quad \mathsf{Z}_n\sim \mathcal{N}(0,I_n).
\end{align}
Thus the coordinatewise GLS estimates form an exact correlated Gaussian sequence as in our setting \eqref{eq:intro_corr_sequence}, and the CML method may be applied to the `observation' $\hat\mu_{\mathrm{gls}}$ for the purpose of estimating $G_\ast$. The details of this program can be found in Section \ref{sec:gls_application}.

From a broader perspective of the literature, prior estimation in linear regression in high dimensions has received significant recent attention. Representative approaches along this line include the variational mean-field approach \cite{mukherjee2023mean,kim2024flexible,lee2026parametric}, which provides consistency of normalized log-likelihood-ratio under specific conditions and usually in the large sample regime, and the gradient-flow based approach with full likelihood \cite{fan2023gradient,fan2025dynamical,fan2025dynamical}, which requires a fairly complicated algorithm in the proportional regime. Our approach based on the CML method applied to \eqref{eq:intro_gls_sequence} is qualitatively different due to the intrinsic 
misspecification of the likelihood. Nonetheless, our approach enjoys a straightforward non-asymptotic convergence guarantee \eqref{eqn:rate_intro} in both regimes and beyond, in contrast to the regime-specific and/or asymptotic analyses presented in the aforementioned works.

The second application concerns the more challenging setting of Bayesian nonlinear regression. Suppose
\begin{align}\label{eq:intro_nonlin_model}
Y_i=\mathsf{F}(\langle A_i,\mu_\ast\rangle,\xi_i),
\qquad i\in[m],
\end{align}
with Gaussian features $A_i\sim\mathcal N(0,\Sigma/n)$ and an unknown prior $G_\ast$ for the coordinates of $\mu_\ast$. The major difficulty for the nonlinear model \eqref{eq:intro_nonlin_model} is that the full likelihood for $G_\ast$ is typically intractable and can be arbitrarily complicated, especially when the link function $\mathsf{F}$ is highly nonlinear or completely unknown. 

Our approach is inspired by the recent development of the so-called `debiased gradient descent' method in \cite{han2026gradient}, but uses a much simplified one-step debiased gradient descent:
\begin{align}\label{eq:intro_gd_sequence}
\hat\mu_{\mathrm{db}}
=\hat\tau\mu_0-\Sigma^{-1}A^\top\mathsf{L}(A\mu_0,Y),
\end{align}
where $\mathsf{L}$ is a user-chosen `loss derivative' function and $\mu_0$ is an initialization sampled from a user-chosen prior. The main reason for the choice is that \eqref{eq:intro_gd_sequence} behaves, in a distributional sense, as a correlated Gaussian sequence model
\begin{align}\label{eq:intro_gd_oracle}
\hat\mu_{\mathrm{db}} \stackrel{d}{\approx} \bar\delta\mu_\ast+\bar\sigma\Sigma^{-1/2}\mathsf Z_n.
\end{align}
Hence the difficult problem of estimating $G_\ast$ under an intractable likelihood can now be reduced to an approximate correlated Gaussian sequence (\ref{eq:intro_gd_oracle}), to which the CML method can be applied to estimate the prior $G_\ast$ (up to an unavoidable scaling factor). The details of this program can be found in Section \ref{sec:gd_application}.

\subsection{Proof techniques}

The key technical ingredient for proving \eqref{eqn:rate_intro} is to establish a sharp local maximal inequality for the log composite marginal likelihood ratio process
\begin{align}\label{eq:intro_log_process}
\mathfrak{L}_{n;\sigma_{0,[n]}}(G)
=\frac1{\sqrt n}\sum_{j \in [n]}
\log\frac{\varphi_{G;\sigma_{0,j}}(U_j)}{\varphi_{G_0;\sigma_{0,j}}(U_j)}.
\end{align}
When $\{U_j\}$ are independent, a large class of classical empirical process tools can be leveraged to provide a sharp bound for \eqref{eq:intro_log_process}, via, for example, entropy methods; cf. \cite{ghosal2001entropies,zhang2009generalized,jiang2009general}.

The major difficulty in our setting \eqref{eq:intro_corr_sequence} therefore lies in the possible arbitrary dependence among ${U_j}$. To this end, we leverage a version of the geometric Brascamp-Lieb inequality for Gaussian measures proved by \cite{chen2015improved} to establish a pointwise Bernstein-type inequality for $\mathfrak L_n(G)$ for a fixed $G$. Conceptually, this is viable because the Brascamp-Lieb inequality decouples the exponential moments of dependent Gaussian variables into marginal ones at the cost of the spectral radius of the underlying correlation matrix. We then use a normal-mixture discretization argument and a dyadic localization over Hellinger shells to strengthen the pointwise Bernstein inequality to a sharp local maximal inequality of the following form: with high probability, for $r>0$ not too small,

\begin{align}\label{eqn:local_maximal_ineq}
&\sup_{\mathfrak d_{H;\sigma_{0,[n]} }(G,G_0)\le r}
|(\mathrm{id}-\E)\mathfrak{L}_{n;\sigma_{0,[n]} }(G)| \lesssim
\pnorm{ \texttt{Cor}_{\Sigma_0} }{\op}^{1/2}\cdot \, r\,\mathrm{polylog}(n).
\end{align}
Once \eqref{eqn:local_maximal_ineq} is proved, we may use the standard peeling argument from empirical process theory \cite{van1996weak,van2000empirical} to prove \eqref{eqn:rate_intro}. In this sense, our method of proof identifies the effective sample size via the multiplicative factor in (\ref{eqn:local_maximal_ineq}) due to the intrinsic Gaussian fluctuation in the presence of a possibly arbitrary dependence structure within $\{U_j\}$.

\subsection{Organization}

The rest of the paper is organized as follows. Section \ref{section:main_results} presents our main theory \eqref{eqn:rate_intro} and a generic application to marginal posterior credible intervals and marginal empirical Bayes regret. Section \ref{sec:gls_application} details the application of the CML method to the generalized least squares estimator in the linear regression setting, whereas Section \ref{sec:gd_application} presents the application of the CML method to one-step debiased gradient descent in general nonlinear regression. Numerical experiments are presented in Section \ref{sec:numerics}. Proofs are deferred to Sections \ref{section:proof_MML_rate}-\ref{section:proof_eb_gd_hellinger} and the appendices.

\subsection{Notation}

For any two integers $m,n$, let $[m:n]\equiv \{m,m+1,\ldots,n\}$ and $[n]\equiv [1:n]$. When $m>n$, it is understood that $[m:n]=\emptyset$.  

For $a \in \R$, $\mathrm{id}(a)=a$. For $a,b \in \R$, $a\vee b\equiv \max\{a,b\}$ and $a\wedge b\equiv\min\{a,b\}$. For $a \in \R$, let $a_\pm \equiv (\pm a)\vee 0$. For a multi-index $a \in \mathbb{Z}_{\geq 0}^n$, let $\abs{a}\equiv \sum_{i \in [n]}a_i$. For $x \in \R^n$, let $\pnorm{x}{p}$ denote its $p$-norm $(0\leq p\leq \infty)$, and $B_{n;p}(R)\equiv \{x \in \R^n: \pnorm{x}{p}\leq R\}$. We simply write $\pnorm{x}{}\equiv\pnorm{x}{2}$ and $B_n(R)\equiv B_{n;2}(R)$. For $x \in \R^n$, let $\mathrm{diag}(x)\equiv (x_i\bm{1}_{i=j})_{i,j \in [n]} \in \R^{n\times n}$. 

For a matrix $M \in \R^{m\times n}$, let $\pnorm{M}{\op},\pnorm{M}{F}$ denote the spectral and Frobenius norms of $M$, respectively. $I_n$ is reserved for the $n\times n$ identity matrix, written simply as $I$ (in the proofs) if no confusion arises. 

We use $C_{x}$ to denote a generic constant that depends only on $x$, whose numerical value may change from line to line unless otherwise specified. $a\lesssim_{x} b$ and $a\gtrsim_x b$ mean $a\leq C_x b$ and $a\geq C_x b$, abbreviated as $a=\bigo_x(b)$ and $a=\Omega_x(b)$, respectively; $a\asymp_x b$ means both $a\lesssim_{x} b$ and $a\gtrsim_x b$. $\bigo$ and $\smallo$ (resp. $\mathcal{O}_{\mathbf{P}}$ and $\mathfrak{o}_{\mathbf{P}}$) denote the usual big-O and small-o notation (resp. in probability). By convention, sums and products over an empty set are understood as $\Sigma_{\emptyset}(\cdots)=0$ and $\Pi_{\emptyset}(\cdots)=1$. 

For a random variable $X$, we use $\Prob_X,\E_X$ (resp. $\Prob^X,\E^X$) to indicate that the probability and expectation are taken with respect to $X$ (resp. conditional on $X$). 

For $\Lambda>0$ and $\mathfrak{p}\in \N$, a measurable map $f:\R^n \to \R$ is called \emph{$\Lambda$-pseudo-Lipschitz of order $\mathfrak{p}$} iff $
\abs{f(x)-f(y)}\leq \Lambda\cdot  (1+\pnorm{x}{}+\pnorm{y}{})^{\mathfrak{p}-1}\cdot\pnorm{x-y}{}$ holds for all  $x,y \in \R^{n}$.
Moreover, $f$ is called \emph{$\Lambda$-Lipschitz} iff $f$ is $\Lambda$-pseudo-Lipschitz of order $1$, and in this case we often write $\pnorm{f}{\mathrm{Lip}}\leq L$, where $\pnorm{f}{\mathrm{Lip}}\equiv \sup_{x\neq y} \abs{f(x)-f(y)}/\pnorm{x-y}{}$.  

Let $d_H^2(f,g)\equiv 2^{-1}\int (f^{1/2}-g^{1/2})^2$ be the standard Hellinger distance defined for two densities $f,g$ on $\R$ (with respect to some dominating measure). For any $p\geq 1$, the Wasserstein $p$-metric between $G_1,G_2\in \mathscr{G}$ is defined as 
	\begin{align*}
	\mathsf{W}_p(G_1,G_2)&\equiv \inf_{(X,Y):X\sim G_1,Y\sim G_2} \E^{1/p}\abs{X-Y}^p.
	\end{align*}
	Here the infimum is taken over all possible couplings of $(X,Y)$ with the prescribed marginal distributions.

\section{Main results}\label{section:main_results}

\subsection{General setting}
Consider a generic setting: suppose $\beta_0 \in \R^n$ has i.i.d. entries
$\{\beta_{0,j}\}$ distributed as $G_0$. For a generic covariance
$\Sigma_0\in \R^{n\times n}$ and $\mathsf{Z}_n\sim \mathcal{N}(0,I_n)$
independent of all other variables, let
\begin{align}\label{def:MPLE_U}
U\equiv \beta_0+ \Sigma_0^{1/2}\mathsf{Z}_n \in \R^n.
\end{align}
The $\delta_n$-near maximum composite marginal likelihood (CMLE) estimator $\hat{G}_n$ is
defined as any distribution function $\hat{G}_n\in \mathscr{G}$ such that
\begin{align}\label{def:MPLE_generic}
\frac{1}{n}\sum_{j \in [n]} \log \varphi_{\hat{G}_n;\sigma_{0,j}}(U_j)
\geq  \max_{G \in \mathscr{G} } \frac{1}{n}\sum_{j \in [n]} \log
\varphi_{G;\sigma_{0,j}}(U_j)-\delta_n.
\end{align}
Here, as before, $\sigma_{0,j}^2\equiv (\Sigma_0)_{jj}$, $j\in[n]$. Recall the log composite marginal likelihood process $\mathfrak{L}_{n;\sigma_{0,[n]}}(G)$ in (\ref{eq:intro_log_process}).

\subsection{Main results}
We need some further notation to state the results:
\begin{itemize}
    \item For $\alpha\in(0,\infty]$ and $M>1$, let $\mathscr G(M,\alpha)\subset \mathscr{G}$ be defined via
	\begin{align*}
	G\in\mathscr G(M,\alpha)\,\Leftrightarrow\,
	\begin{cases}
	\int \exp\{|u|^\alpha/M^\alpha\}\,G(\d u)\leq M,& \alpha<\infty;\\
	\mathrm{supp}(G)\subset[-M,M], & \alpha = \infty.
	\end{cases}
	\end{align*}
    \item Recall $\sigma_{0,j}^2\equiv (\Sigma_0)_{jj}$ and $\mathfrak{D}_{\Sigma_0} = \mathrm{diag}(\Sigma_0)$. Let
\begin{align}\label{def:kappa_n_ast}
\kappa_0\equiv \pnorm{\texttt{Cor}_{\Sigma_0}}{\op} =  \pnorm{  \mathfrak D_{\Sigma_0}^{-1/2}\Sigma_0\mathfrak D_{\Sigma_0}^{-1/2}}{\op},\quad n_\ast\equiv \frac{n}{\pnorm{\texttt{Cor}_{\Sigma_0}}{\op}} = \frac{n}{\kappa_0}.
\end{align} 
\item For any $G_1,G_2 \in \mathscr{G}$, and a sequence $\{\sigma_j\}_{j \in [n]}\subset \R_{>0}$, we define the averaged Hellinger distance as
	\begin{align}\label{def:averaged_hellinger}
	\mathfrak{d}_{H;\sigma_{[n]}}^2(G_1,G_2)\equiv \frac{1}{n}\sum_{j \in [n]} d_H^2\big(\varphi_{G_1;\sigma_{j}},\varphi_{G_2;\sigma_{j}}\big).
	\end{align} 
Here recall $d_H^2(f,g)\equiv 2^{-1}\int (f^{1/2}-g^{1/2})^2$ is the standard Hellinger distance.
\end{itemize}

The main abstract result for this paper is the following; its proof can be found in Section \ref{section:proof_MML_rate}.
\begin{theorem}
\label{thm:MML_rate}
Consider the above setting. Suppose there exist $M>1$ and $\alpha\in(0,\infty]$ such that $G_0\in\mathscr G(M,\alpha)$ and $\sigma_{0,j}^2 \in [1/M,M]$ for all $j \in [n]$.
Then for any $D>0$, there exists a constant $c_1=c_1(\alpha,M,D)>1$ such that if $n_\ast \geq (\log n)^{c_1}$, with probability at least $1-c_1 n^{-D}$, the following hold.
\begin{enumerate}
\item (Local maximal inequality). Let $L_{n,\alpha}\equiv \sqrt{\log n}+(\log n)^{1/\alpha}$ and $r_n\equiv c_1n_\ast^{-1/2}(\log n)^{c_1}$. For any $\mathscr{G}_{\texttt{test}}\subset \mathscr{G}$ and any $r\in [1\wedge r_n,1]$,
\begin{align}\label{ineq:MML_localized_rate}
\sup_{G\in \mathscr{G}(c_1L_{n,\alpha},\infty)\cap \mathscr{G}_{\texttt{test}}:\mathfrak{d}_{H;\sigma_{0,[n]}}(G,G_0)\leq r}
\abs{(\mathrm{id}-\E)\mathfrak{L}_{n;\sigma_{0,[n]}}(G)}
\leq \sqrt{\kappa_0}\cdot r(\log n)^{c_1}.
\end{align}
\item (Convergence rates). Let $\hat G_n$ be a $\delta_n$-near CML estimator as in \eqref{def:MPLE_generic}. On the further event $\{\mathrm{supp}(\hat G_n)\subset [-c_0L_{n,\alpha},c_0L_{n,\alpha}]\}$, there exists $c_2=c_2(c_0,\alpha,M,D)>1$ such that if $
\Delta_{n,\delta}\equiv (n_\ast^{-1/2}+\delta_n^{1/2}) (\log n)^{c_2}\leq 1/2$,
\begin{align}\label{ineq:MML_Hellinger_rate}
d_H(\varphi_{\hat{G}_n;M},\varphi_{G_0;M})\leq \mathfrak{d}_{H;\sigma_{0,[n]}}(\hat G_n,G_0)
\leq \Delta_{n,\delta}.
\end{align}
Moreover, for any $p\geq 1$, there exists a constant $c_3=c_3(p,c_0,\alpha,M,D)>1$ such that, 
\begin{align}\label{ineq:MML_Wasserstein_rate}
\mathsf{W}_p(\hat G_n,G_0)
\leq c_3 \log^{-1/2} (1/\Delta_{n,\delta}).
\end{align}
\end{enumerate}
\end{theorem}

For an exact minimizer, \cite{lindsay1983geometry} proves its uniqueness and that its support must be contained in $[\min_jU_j,\max_jU_j]$. Using the tail bound in Lemma \ref{lem:normal_mix_basic}, we have:
\begin{corollary}\label{cor:MML_exact}
	Suppose the conditions in Theorem \ref{thm:MML_rate} hold. Let $\hat G_n$ be the exact CML estimator defined in \eqref{def:MPLE_generic} with $\delta_n=0$. Then for any $D>0$, there exists a constant $c_1=c_1(\alpha,M,D)>1$ such that, if $n_\ast\ge (\log n)^{c_1}$, with probability at least $1-c_1 n^{-D}$,
	\begin{align*}
	d_H(\varphi_{\hat{G}_n;M},\varphi_{G_0;M})\leq\mathfrak{d}_{H;\sigma_{0,[n]}}(\hat G_n,G_0)\leq n_\ast^{-1/2}(\log n)^{c_1}.
	\end{align*}
	Moreover, for any $p\ge1$, after possibly increasing $c_1=c_1(p,\alpha,M,D)$,
	\begin{align*}
	\mathsf W_p(\hat G_n,G_0)\le c_1(\log n_\ast)^{-1/2}.
	\end{align*}
\end{corollary}

\begin{remark}
	Some technical remarks on Theorem \ref{thm:MML_rate} and Corollary \ref{cor:MML_exact}.
	\begin{enumerate}
		\item The variance boundedness condition $\sigma_{0,j}^2\in [1/M,M]$ is common in heteroscedastic independent Gaussian sequence model \cite{jiang2020general,soloff2025multivariate,chen2026normal}. Here we work with a very general covariance structure $\Sigma_0$ with the same marginal variance condition. 
		\item The term $\delta_n^{1/2}$ in $\Delta_{n,\delta}$ is intrinsic for a near maximizer. The population likelihood contrast is quadratic in the Hellinger distance, so an optimization error of order $\delta_n$ cannot in general imply better than a $\delta_n^{1/2}$ Hellinger error. 
		\item Part (2) is stated for a near maximizer whose support is logarithmically bounded. For an arbitrary near maximizer, projecting it onto $[\min_j U_j,\max_j U_j]$ gives another near maximizer but does not, by itself, prove a rate for the original unprojected one.
		\item The rate $(\log n_\ast)^{-1/2}$ under $\mathsf{W}_p$ in Corollary \ref{cor:MML_exact} may not be improved when $G_0$ has unbounded support \cite{dedecker2013minimax}. However, for $p=1$ and boundedly supported $G_0$, \cite[Theorem 5]{wu2020optimalestimation} shows that the optimal rate under $\mathsf{W}_1$ is $\log\log n/\log n$. It remains open whether the CML estimator is suboptimal in this case, or whether the analysis can be further sharpened.
	\end{enumerate}
	
\end{remark}

As explained in the Introduction, $n_\ast$ defined in (\ref{def:kappa_n_ast}) can be viewed as the `effective sample size'. For instance, when $\Sigma_0 = I_n$, $n_\ast = n$ recovers the usual sample size; in the extreme case $\Sigma_0 = \bm{1}\bm{1}^\top$, where the noise is perfectly correlated, we have $n_\ast=1$, and estimation of $G_0$ is impossible.

A more fundamental reason that $n_\ast$ serves as the right complexity measure is provided by the following minimax lower bound, which also shows that the rate in \eqref{ineq:MML_Hellinger_rate} is optimal modulo logarithmic factors.

\begin{proposition}
	\label{prop:minimax_lower_effective_sample_size}
	Fix $M>1$. There exists $c_1=c_1(M)>0$ such that for all
	$n_\ast\ge c_1$,
	\begin{align*}
	\inf_{\widetilde G} 
	\sup_{(G_0,\Sigma_0) \in \mathscr{F}_M(\kappa_0) }
	\mathbb \Prob_{G_0,\Sigma_0}\big(
	\mathfrak{d}_{H;\sigma_{0,[n]}}(\widetilde G,G_0)> (n/\kappa_0)^{-1/2}/c_1\big)
	\ge 1/4.
	\end{align*}
	Here $\mathscr{F}_M(\kappa_0)\equiv \big\{(G_0,\Sigma_0): G_0\in\mathscr G(M,\infty), \pnorm{\texttt{Cor}_{\Sigma_0}}{\op}\leq \kappa_0\big\}$.
\end{proposition}
The proof of the above proposition can be found in Section \ref{section:proof_minimax_lower_effective_sample_size}.

\begin{remark}
	During the preparation of this manuscript, we became aware of the recent work \cite{chen2026normal}, which considers the convergence of the CML estimator under a general $\Sigma_0$. In particular, \cite[Corollary 4.1]{chen2026normal} shows that $\mathfrak{d}_{H;\{\sigma_{0,j}\}}(\hat G_n,G_0)\to 0$ and $\hat{G}_n \rightsquigarrow G_0$ in probability under a bounded support condition and the condition $\pnorm{\Sigma_0}{\op} = \smallo(n)$. 
	
	Our Theorem \ref{thm:MML_rate} substantially strengthens their results by proving an optimal rate with the effective sample size $n_\ast$ in \eqref{def:kappa_n_ast}, and our results do not require bounded supports. We also provide a matching lower bound in Proposition \ref{prop:minimax_lower_effective_sample_size}. Moreover, the proof method in \cite{chen2026normal} is intrinsically asymptotic, whereas our method leverages a version of the geometric Brascamp-Lieb inequality for Gaussian measures from \cite{chen2015improved} to obtain optimal, non-asymptotic rates.
\end{remark}

\subsection{Empirical Bayes credible intervals}\label{subsec:credible_consequence}

As a generic application of Theorem \ref{thm:MML_rate}, we now consider empirical Bayes credible intervals for the individual coordinates $U=(U_1,\ldots,U_n)$ in the original correlated Gaussian sequence model \eqref{def:MPLE_U}.  

We need some further notation. For $G\in\mathscr G$, $s>0$, and $u\in\R$, define the univariate posterior distribution of $\theta\sim G, u\mid \theta\sim \mathcal N(\theta,s^2)$ by 
\begin{align*}
\Pi_{G,s}(t \mid u)
\equiv
\frac{1}{\varphi_{G;s}(u)}\int_{-\infty}^t\varphi_s(u-\theta)\,G(\d\theta),\quad t \in \R,
\end{align*}
and its lower $a$-quantile by 
\begin{align*}
Q_{G,s}(a\mid u)
\equiv
\inf\{t\in\R:\Pi_{G,s}(t\mid u)\ge a\},\quad a \in (0,1).
\end{align*}
Let $\hat G_n$ be the exact CML estimator in Corollary \ref{cor:MML_exact}.  For $\gamma\in(0,1)$, define the empirical Bayes credible interval for $\beta_{0,j}$ by
\begin{align}\label{eq:eb_ci_original_U}
        \widehat C_j(U;\gamma)
        \equiv
        \big[
        Q_{\hat G_n,\sigma_{0,j}}(\gamma/2\mid U_j),
        Q_{\hat G_n,\sigma_{0,j}}(1-\gamma/2\mid U_j)
        \big],
        \qquad j\in[n].
\end{align}
The following corollary gives an average frequentist coverage guarantee.  

\begin{proposition}
\label{prop:credible_intervals_clean}
Suppose the conditions of Corollary \ref{cor:MML_exact} hold  with $\alpha=\infty$.  In addition, suppose that $G_0$ has a Lebesgue density $g_0$ supported on $[-M_0,M_0]$ on which $g_0(\cdot) \in [1/M,M]$.
Fix $\gamma\in(0,1)$.  Then there exists $c_1=c_1(\gamma,M,M_0)>1$, such that for $n_\ast\ge (\log n)^{c_1}$,
\begin{align}\label{eq:ci_nstar_final_bound}
\abs{
\Prob\big(\beta_{0,\pi_n}\in \widehat C_{\pi_n}(U;\gamma)\big)
-(1-\gamma)}
\le
{\log^{-1/c_1}(en_\ast)}.
\end{align}
Here $\pi_n\sim \mathrm{Unif}[n]$ is independent of all other variables.
\end{proposition}

The proof of the above proposition can be found in Section \ref{section:proof_eb_credible_interval}.

\subsection{Marginal empirical Bayes regret}\label{subsec:marginal_regret}

We next record a decision-theoretic consequence of Theorem \ref{thm:MML_rate} for the marginal empirical Bayes regret, in a similar flavor to \cite{jiang2009general}. For $G\in\mathscr G$, $s>0$, and $u\in\R$, let
\begin{align}\label{def:posterior_mean_regret}
m_{G,s}(u)
\equiv
\frac{1}{\varphi_{G;s}(u)} \int \theta\cdot \varphi_s(u-\theta)\,G(\d\theta)
\end{align}
be the posterior mean in the marginal Gaussian experiment $\theta\sim G$ and $u\mid\theta\sim\mathcal N(\theta,s^2)$.  For a deterministic prior estimate $G$, define its averaged marginal regret relative to $G_0$ by
\begin{align}\label{eq:regret_identity_text}
\texttt{Reg}_n(G,G_0)
\equiv \frac1n\sum_{j \in [n]}
\Big(\E
\big(m_{G,\sigma_{0,j}}(X_j^\circ)-\theta_j^\circ\big)^2
-
\E \big(m_{G_0,\sigma_{0,j}}(X_j^\circ)-\theta_j^\circ\big)^2\Big).
\end{align}
Here $\theta_j^\circ\sim G_0$ and $X_j^\circ=\theta_j^\circ+\sigma_{0,j}\mathsf{Z}_j^\circ$, $\mathsf{Z}_j^\circ \sim \mathcal{N}(0,1)$ is an independent fresh draw for $j \in [n]$. Note that the benchmark in (\ref{eq:regret_identity_text}) is the marginal normal oracle and is different from the full-vector oracle $\E(\beta_{0,j}\mid U_1,\ldots,U_n)$.

\begin{proposition}\label{prop:marginal_regret}
Suppose the conditions of Corollary \ref{cor:MML_exact} hold with $\alpha=\infty$, and let $\hat G_n$ be the exact CML estimator in Corollary \ref{cor:MML_exact}.  Then for any $D>0$, there exists $c_1=c_1(M,D)>1$ such that if $n_\ast\ge (\log n)^{c_1}$, then with probability at least $1-c_1n^{-D}$,
\begin{align*}
0\leq \texttt{Reg}_n(\hat G_n,G_0)
\le
\, n_\ast^{-1}(\log n)^{c_1}.
\end{align*}
\end{proposition}

The proof of Proposition \ref{prop:marginal_regret} uses a recent sharp result of \cite{chen2026sharp} that relates the regret and the Hellinger distance; details can be found in Section \ref{section:proof_marginal_regret}. 

We emphasize here an important distinction of the regret result in Proposition \ref{prop:marginal_regret} with those presented in \cite{jiang2009general}. The estimator \eqref{def:posterior_mean_regret} is a coordinatewise empirical Bayes rule: after learning the prior, it denoises coordinate $j$ through the scalar marginal experiment $U_j=\beta_{0,j}+\sigma_{0,j}\mathsf{Z}_j$.  Its natural oracle is therefore the marginal Bayes rule $m_{G_0,\sigma_{0,j}}(U_j)$.  In contrast, the full posterior mean $\E(\beta_{0,j}\mid U_{[n]})$ uses the correlation structure of the noise to extract additional information about the error in coordinate $j$.  The following example shows that this extra information can produce a non-vanishing risk gap even when the effective sample size $n_\ast$ is of order $n$.

\begin{example}
\label{ex:marginal_vs_full_oracle}
Let $n$ be even and split the coordinates into independent pairs.  Within each pair, suppose
\begin{align*}
\beta=(\beta_1,\beta_2)^\top\sim \mathcal N(0,I_2),
\qquad
U=\beta+\varepsilon,
\quad
\varepsilon\sim\mathcal N(0,\Sigma_\rho),
\end{align*}
where $\Sigma_\rho=
\begin{psmallmatrix}
1 & \rho \\
\rho & 1
\end{psmallmatrix}$ with $0<\rho<1$, and different pairs are independent.  The diagonal-normalized covariance matrix is block diagonal with blocks $\Sigma_\rho$, and hence $
\kappa_0=\|\texttt{Cor}_{\Sigma_0}\|_{\op}=1+\rho$ and  $n_\ast=n/(1+\rho)\asymp n$. 

Nevertheless, even if the prior $G_0=\mathcal N(0,1)$ is known exactly, the marginal Bayes rule does not attain the full posterior minimum mean squared error (MMSE).  Since $U_j=\beta_j+\varepsilon_j$ with $\beta_j\sim\mathcal N(0,1)$ and $\varepsilon_j\sim\mathcal N(0,1)$ marginally, we have $\beta_j \mid U_j \sim \mathcal{N}(U_j/2,1/2)$, and therefore
\begin{align*}
\E(\beta_j-\E(\beta_j\mid U_j))^2=\E (\beta_j-U_j/2)^2= 1/2.
\end{align*}
On the other hand, the full posterior covariance in one pair is $\operatorname{Cov}(\beta\mid U)=\big(I_2+\Sigma_\rho^{-1}\big)^{-1}$.
Because $\Sigma_\rho^{-1}=\frac1{1-\rho^2}
\begin{psmallmatrix}
1 & -\rho \\
-\rho & 1
\end{psmallmatrix}$, 
a direct calculation gives the full posterior marginal variance
\begin{align*}
\big[\big(I_2+\Sigma_\rho^{-1}\big)^{-1}\big]_{jj}
=\frac{(2-\rho^2)(1-\rho^2)}{(2-\rho^2)^2-\rho^2}
<\frac12,
\qquad 0<\rho<1.
\end{align*}
Therefore the excess risk of the exact marginal oracle relative to the full posterior oracle is a positive constant depending only on $\rho$, whereas $n_\ast\asymp n$.  Consequently no nontrivial bound can hold, in general, for a coordinatewise marginal empirical Bayes rule when the benchmark is the full posterior MMSE.
\end{example}

As such, the main purpose of the marginal regret bound in Proposition \ref{prop:marginal_regret} is not to quantify the price of ignoring the full dependence structure in the denoising rule.  Rather, it isolates the statistical cost of estimating the prior $G_0$ within the marginal empirical Bayes class generated by the composite marginal likelihood.

\section{Linear regression: CML via generalized least squares}\label{sec:gls_application}

\subsection{Model and CML via generalized least squares}
Consider the Bayesian Gaussian linear model
\begin{align}\label{def:gls_model}
Y=A\mu_\ast+\xi \in \R^m,\qquad \xi\mid A\sim \mathcal N(0,\tau_\ast^2\Omega).
\end{align}
Here the entries of the signal $\mu_\ast \in \R^n$ are i.i.d. draws from an unknown prior distribution $G_\ast$. Moreover, the design matrix $A\in\R^{m\times n}$ has full column rank, $\Omega\in\R^{m\times m}$ is a known positive definite working covariance, and $\tau_\ast>0$ is either known or estimated from the data. 

In the simplest possible linear regression case, $\Omega=I_m$. Other applications with known $\Omega$ include, (i) weighted least squares with known inverse-variance (precision) weights, where $\Omega=\mathrm{diag}(w_1^{-1},\ldots,w_m^{-1})$, (ii) generalized least squares that treat the covariance shape as known with an unknown scalar variance \cite{seber2003linear}, (iii) spatial or inverse-problem settings where a covariance kernel or measurement-error covariance is supplied by the design or calibration model \cite{cressie2015statistics}, etc. If a nonparametric or high-dimensional estimate $\hat{\Omega}$ of $\Omega$ is used, an additional perturbation analysis for $A^\top\hat\Omega^{-1}A$ is required and will not be pursued here.

\begin{algorithm}
\caption{CML via generalized least squares}\label{alg:gls_eb}
\begin{algorithmic}[1]
\STATE \textbf{Input}: Data $(A,Y)\in \R^{m\times n}\times \R^{m}$, covariance $\Omega \in \R^{m\times m}$.
\STATE With $Q_A\equiv A^\top\Omega^{-1}A$, compute 
\begin{align*}
\hat\mu_{\mathrm{gls}}\equiv Q_A^{-1}A^\top\Omega^{-1}Y.
\end{align*}
\STATE Set
\begin{align*}
\hat\tau\equiv
\begin{cases}
\tau_\ast, & \hbox{if $\tau_\ast$ is known};\\
\Big\{\frac{(Y-A\hat\mu_{\mathrm{gls}})^\top\Omega^{-1}(Y-A\hat\mu_{\mathrm{gls}})}{m-n}\Big\}^{1/2}, & \hbox{if $\tau_\ast$ is unknown and $m>n$}.
\end{cases} 
\end{align*}
\STATE With $s_{A,j}^2\equiv (Q_A^{-1})_{jj}$, compute the CML
\begin{align*}
\hat G_{\mathrm{gls}}=\argmax_{G\in\mathscr G}
\frac1n\sum_{j \in [n]}\log\varphi_{G;\hat\tau s_{A,j}}(\hat\mu_{\mathrm{gls},j}).
\end{align*}
\end{algorithmic}
\end{algorithm}

We will be interested in estimating $G_\ast$ without imposing strong assumptions on the design matrix $A$. Our proposal is based on the following simple observation: conditionally on $A$, the generalized least squares estimator satisfies
\begin{align}\label{def:gls_sequence}
\hat\mu_{\mathrm{gls}}\equiv Q_A^{-1}A^\top\Omega^{-1}Y=\mu_\ast+\tau_\ast Q_A^{-1/2}\mathsf Z_n,
\quad \mathsf Z_n\sim\mathcal N(0,I_n),
\end{align}
where $Q_A \equiv A^\top \Omega^{-1} A$. Thus, $\hat\mu_{\mathrm{gls}}$ can be viewed as an exact correlated Gaussian sequence, with marginal standard errors $\tau_\ast (Q_A^{-1})_{jj}^{1/2}$. Our proposed Algorithm \ref{alg:gls_eb} then implements the CML method applied to $\hat\mu_{\mathrm{gls}}$ in (\ref{def:gls_sequence}).

\subsection{Theoretical guarantee}

The following theorem provides a formal justification for the output $\hat{G}_{\mathrm{gls}}$ from Algorithm \ref{alg:gls_eb}; its proof can be found in Section \ref{section:proof_gls_eb}.

\begin{theorem}
\label{thm:gls_eb}
Suppose $\mu_{\ast,j}\stackrel{\mathrm{i.i.d.}}{\sim}G_\ast$ independently of $A$ and $\xi$, and there exist $M>1$ and $\alpha\in(0,\infty]$ such that $G_\ast\in\mathscr G(M,\alpha)$ and, conditional on $A$, it holds that $\tau_\ast^2(Q_A^{-1})_{jj} \in [1/M,M]$ for all $j \in [n]$. Let $n_{\ast,A}\equiv n/\pnorm{\texttt{Cor}_{Q_A^{-1}}}{\op}$.
Fix $D>0$ and $p\ge1$. Then there exists $c_1=c_1(\alpha,M,D,p)>1$ such that, conditional on $A$, the following hold with probability at least $1-c_1 n^{-D}$:
\begin{enumerate}
\item If $\tau_\ast$ is known and $n_{\ast,A}\geq (\log n)^{c_1}$, then 
\begin{align*}
\mathfrak d_{H;\tau_\ast s_{A,[n]}}(\hat G_{\mathrm{gls}},G_\ast)
&\leq n_{\ast,A}^{-1/2}(\log n)^{c_1},\quad
\mathsf W_p(\hat G_{\mathrm{gls}},G_\ast)
\le c_1 (\log n_{\ast,A})^{-1/2}.
\end{align*}
\item If $\tau_\ast$ is unknown and $m>n$, and $
\Delta_A\equiv (n_{\ast,A}^{-1/2}+\gamma_A^{1/2})(\log n)^{c_1}\leq 1/2$, further on the event $\gamma_A\equiv \abs{\hat\tau/\tau_\ast-1}\leq 1/2$, we have
\begin{align*}
\mathfrak d_{H;\tau_\ast s_{A,[n]}}(\hat G_{\mathrm{gls}},G_\ast)
&\leq \Delta_A,\quad\mathsf W_p(\hat G_{\mathrm{gls}},G_\ast)
\le c_1 
\log^{-1/2}(1/\Delta_A).
\end{align*}
Moreover, $\Prob\big(\gamma_A>c_1 \{\epsilon_{m,n}^{1/2}+\epsilon_{m,n}\}\,|A\big)\leq c_1n^{-D}$ with $\epsilon_{m,n}\equiv \log n/(m-n)$.
\end{enumerate}
\end{theorem}
Let us compare the results in Theorem \ref{thm:gls_eb} to some recent empirical Bayes proposals in the high-dimensional linear model:
\begin{enumerate}
	\item \cite{mukherjee2023mean} studies nonparametric and naive mean-field variational empirical Bayes for linear regression, and establishes consistency and $1$-Wasserstein posterior approximation under deterministic and random designs, typically in the regime $m/n\to \infty$.
	\item \cite{lee2026parametric} studies parametric empirical Bayes in high-dimensional linear regression, estimating a finite-dimensional prior parameter via a variational empirical Bayes objective, and identifies a phase transition in its asymptotic distribution and efficiency theory within the regime $m/n\to \infty$.
	\item \cite{fan2023gradient} proposes to estimate an i.i.d. prior by the NPMLE for the full likelihood, using a Gibbs variational representation and a coupled gradient-flow/Langevin-MCEM algorithm. Their method enjoys asymptotic mixing and convergence guarantees in high-noise or convex-sublevel regimes under a general class of random designs in the proportional regime $m\asymp n$.
	\item \cite{fan2025dynamical} studies a parametric adaptive Langevin empirical Bayes algorithm for Bayesian linear regression with i.i.d. design in the proportional regime $m\asymp n$, with a focus on the high-dimensional asymptotics of the coupled Langevin/prior-parameter dynamics using recent tools from dynamical mean-field theory \cite{celentano2021high,gerbelot2024rigorous,han2025entrywise}.
\end{enumerate}
Our Algorithm \ref{alg:gls_eb} is fundamentally different from these proposals: it estimates the prior by a marginal composite marginal likelihood that can be computed using standard empirical Bayes methods and enjoys a nearly $n_\ast^{-1/2}$ convergence rate in weighted Hellinger distance, even in the most challenging regime $m\asymp n$. However, it should be noted that Algorithm \ref{alg:gls_eb} targets prior estimation through one-dimensional marginals and therefore does not, by itself, provide posterior inference or sampling under the full regression likelihood, as some of the above works do.

We also note that the Gaussian error assumption on ${\xi_i}$ can be easily relaxed at the cost of unnecessary technical detours; an example of such an analysis can be found in the application in Section \ref{sec:gd_application} below. Moreover, it is also straightforward to consider a ridge-regularized version of the generalized least squares estimator in (\ref{def:gls_sequence}) to avoid potential singularity of $A$; we omit these details for simplicity of presentation.

\section{Nonlinear regression: CML via debiased gradient descent}\label{sec:gd_application}

\subsection{Model and CML via debiased gradient descent}
Consider the Bayesian non-linear regression model
\begin{align}\label{def:model}
Y_i = \mathsf{F}\big(\iprod{A_i}{\mu_\ast},\xi_i\big),\quad i \in [m].
\end{align}
Here the entries of $\mu_\ast \in \R^n$ are i.i.d. draws from an unknown prior distribution $G_\ast$. For each sample $i \in [m]$, the response $Y_i$ is generated from the feature vector $A_i\in \R^{n}$, the signal $\mu_\ast\in \R^{n}$ via a possibly unknown, non-linear mapping $\mathsf{F}:\R^{2}\to \R$, and $\xi_i$'s are unobservable statistical errors. We adopt the normalization with $\pnorm{\mu_\ast}{}/\sqrt{n}=\bigo(1)$ and assume that the Gaussian feature vectors $A_i$'s are i.i.d. $\mathcal{N}(0,\Sigma/n)$ for some covariance $\Sigma$. 

We will be interested in estimating $G_\ast$ in the most challenging regime, where the sample size $m$ is proportional to the problem dimension $n$, and the unknown signal $\mu_\ast$ cannot be consistently estimated even in the simplest possible model under (\ref{def:model}). Moreover, even if the link function $\mathsf{F}$ and the error distributions $\xi_i$'s are known, the likelihood function for $G_\ast$ can be arbitrarily complicated and therefore intractable in general.

Clearly, as the link function $\mathsf{F}$ is typically unknown, the prior $G_\ast$ cannot be identified by scalar transformation within the class $\{G_\ast(t\cdot): t \in \R_{\geq 0}\}$. Our proposal below enables estimation of $G_\ast$ up to a scalar factor.

Fix a user-chosen (loss derivative) function $\mathsf{L}:\R\times \R\to \R$. We propose to estimate the prior $G_\ast$ by a two-step procedure, as detailed in Algorithm \ref{alg:eb_gd}.

\begin{algorithm}
	\caption{CML via debiased gradient descent}\label{alg:eb_gd}
	\begin{algorithmic}[1]
		\STATE \textbf{Input}: Data $\{(A_i,Y_i)\}_{i \in [m]} \in \R^n\times \R$, data covariance $\Sigma \in \R^{n\times n}$, a loss derivative function $\mathsf{L}:\R^2\to \R$, initial guess prior  $G_0 \in \mathscr{G}$.
		\STATE Sample $\mu_0\in \R^n$ whose entries are i.i.d. $G_0$.
		\STATE Compute scale estimates $\hat{\sigma}\geq 0, \hat{\tau} \in \R$ by 
		\begin{align*}
		\hat{\sigma}^2\equiv \frac{1}{n}\sum_{i \in [m]} \mathsf{L}^2\big((A\mu_0)_{i},Y_{i}\big),\quad \hat{\tau}\equiv \frac{1}{n}\sum_{i \in [m]} \partial_1 \mathsf{L}\big((A\mu_0)_{i},Y_{i}\big).
		\end{align*}
		\STATE Compute debiased gradient descent $\hat{\mu}_{\mathrm{db}} \in \R^n$ by
		\begin{align}\label{def:one_step_gd}
		\hat{\mu}_{\mathrm{db}} &\equiv \hat{\tau}\cdot \mu_0-\Sigma^{-1} A^\top \mathsf{L}(A\mu_0,Y).
		\end{align}
			\STATE With $\sigma_j\equiv (\Sigma^{-1})_{jj}^{1/2}$, compute the CML estimator
			\begin{align}\label{def:gd_MLE}
			\hat{G}_{\mathrm{db}}= \argmax_{G \in \mathscr{G}} \frac{1}{n}\sum_{j \in [n]} \log \varphi_{G;\hat{\sigma}\sigma_j}(\hat{\mu}_{\mathrm{db},j} ).
			\end{align}
	\end{algorithmic}
\end{algorithm}

It is useful at this point to explain the rationale of Algorithm \ref{alg:eb_gd}. Indeed, inspired by \cite{han2026gradient}, the debiased gradient descent $\hat{\mu}_{\mathrm{db}}$ in (\ref{def:one_step_gd}) has the following distributional approximation
\begin{align}\label{eqn:one_step_gd_dist}
\hat{\mu}_{\mathrm{db}}\stackrel{d}{\approx} \mathcal{N}(\bar{\delta} \mu_\ast, \bar{\sigma}^2 \Sigma^{-1}). 
\end{align}
Now the CML estimator $\hat{G}_{\mathrm{db}}$ is obtained in (\ref{def:gd_MLE}) by pretending that the correlation in $\hat{\mu}$ can be ignored and performing maximum composite marginal likelihood estimation on the marginals.

\subsection{Theoretical guarantee}

Let $\mathfrak{S}:\R^3\to \R$ be defined by $\mathfrak{S}(u_0,v_0,\xi_0)\equiv  \mathsf{L}(u_0,\mathsf{F}(v_0,\xi_0))$.
\begin{assumption}\label{assump:model}
	Suppose the following hold for some $K,\Lambda\geq 2$:
	\begin{enumerate}
		\item[(A1)] $1/K\leq  \phi \equiv m/n \leq K$.
		\item[(A2)] $\{n^{1/2} A_{i\cdot}:  i \in [m]\}$ are i.i.d. as $\mathcal{N}(0,\Sigma)$, where $\pnorm{\Sigma}{\op}\vee \pnorm{\Sigma^{-1}}{\op}\leq \Lambda$. 
		\item[(A3)] The collection $\{\partial_\alpha \mathfrak{S}(\cdot,\cdot,\xi_i): i \in [m], \alpha\in \mathbb{Z}_{\geq 0}^2, \abs{\alpha} \in [0:1]\}$ is $\Lambda$-Lipschitz and bounded at $(0,0)$ by $\Lambda$. 
	\end{enumerate}
\end{assumption}

\begin{assumption}
\label{assump:prior}
Suppose the following hold for some $M\geq 2$ and $\alpha\in(0,\infty]$:
\begin{enumerate}
\item[(B1)] $\mu_\ast$ has i.i.d. entries distributed as $G_\ast$, where $G_\ast\in\mathscr G(M,\alpha)$.
\item[(B2)] $\mu_0$ has i.i.d. entries distributed as $G_0$, where $G_0\in\mathscr G(M,\alpha)$.
\end{enumerate}
\end{assumption}

\begin{definition}\label{def:U_bar}
		With $a_\Sigma\equiv \tr(\Sigma)/n$ and $b_\Sigma\equiv \bm{1}_n^\top \Sigma \bm{1}_n/n$, let $\bar{\mathfrak{U}}=(\bar{\mathfrak{U}}_{1},\bar{\mathfrak{U}}_{2})$ be a centered, bi-variate Gaussian vector with covariance  
		\begin{align*}
		\begin{pmatrix}
		a_\Sigma \var(G_0)+ b_\Sigma (\E G_0)^2 & a_\Sigma \cov(G_0,G_\ast)+ b_\Sigma (\E G_0 \E G_\ast)\\
		a_\Sigma \cov(G_0,G_\ast)+ b_\Sigma (\E G_0 \E G_\ast) & a_\Sigma \var(G_\ast)+ b_\Sigma (\E G_\ast)^2
		\end{pmatrix}.
		\end{align*}
        We then define 
        \begin{align*}
        \bar{\delta}\equiv -\phi\cdot \E \partial_2\mathfrak{S}\big(\bar{\mathfrak{U}}_{1}, \bar{\mathfrak{U}}_{2},\xi_{\pi_m}\big),\quad \bar{G}_\ast(\cdot)\equiv G_\ast(\cdot/\abs{\bar{\delta}}).
        \end{align*}
		Here $\E$ is taken jointly with respect to $\bar{\mathfrak U}$ and $\pi_m\sim\mathrm{Unif}([m])$.	
\end{definition}

Recall the averaged Hellinger distance $\mathfrak{d}_{H;\bar{\sigma}\sigma_{[n]}}$ as defined in (\ref{def:averaged_hellinger}). The following theorem provides a formal justification for the output $\hat{G}_{\mathrm{db}}$ from Algorithm \ref{alg:eb_gd} for estimating $\bar{G}_\ast$.

\begin{theorem}
\label{thm:eb_gd_hellinger}
Suppose Assumptions \ref{assump:model} and \ref{assump:prior} hold for some $K,\Lambda,M\geq 2$ and $\alpha\in(0,\infty]$, and $K\Lambda M (1\wedge \abs{\bar{\delta}})^{-1}\leq (\log n)^{c_0}$ for some $c_0>1$. Let $n_{\ast,\Sigma}\equiv n/\pnorm{\texttt{Cor}_{\Sigma^{-1}}}{\op}$. 
Fix $D>0$ and $p\geq 1$. Then there exist some constants $c_1=c_1(K,\Lambda,c_0,\alpha,M,D)>0$ and $c_2=c_2(p,K,\Lambda,c_0,\alpha,M,D)>0$ such that if $n_{\ast,\Sigma}\geq (\log n)^{c_1}$, with $\Prob^{\xi}$-probability at least $1-n^{-D}$,
\begin{align*}
\mathfrak{d}_{H;\bar{\sigma}\sigma_{[n]}}\big(\hat{G}_{\mathrm{db}},\bar{G}_\ast\big)
\leq (n^{-1/4}+ n_{\ast,\Sigma}^{-1/2})(\log n)^{c_1},\quad \mathsf{W}_p(\hat{G}_{\mathrm{db}},\bar{G}_\ast)&\leq c_2 (\log n_{\ast,\Sigma})^{-1/2}.
\end{align*}
\end{theorem}
The proof of the above theorem can be found in Section \ref{section:proof_eb_gd_hellinger}. An important technical subtlety in proving Theorem \ref{thm:eb_gd_hellinger} lies in the fact that \eqref{eqn:one_step_gd_dist} holds only in an approximate and averaged sense (cf. Proposition \ref{prop:db_gd_normal}). Consequently, the proof is quantitatively different from that of Theorem \ref{thm:gls_eb} as a direct application of the master Theorem \ref{thm:MML_rate}.

To place Theorem \ref{thm:eb_gd_hellinger} in the literature, prior estimation for a general nonlinear regression model of the form \eqref{def:model} appears substantially less developed than in the linear-model case, mainly because the full likelihood is in general intractable, both theoretically and computationally. In the special case of sparse high-dimensional generalized linear models, \cite{tang2024empirical} proposed an empirical Bayes posterior distribution that achieves the optimal contraction rate and valid posterior inference. Here our proposal in Algorithm \ref{alg:eb_gd} is different: it uses one debiased gradient step to manufacture an approximate correlated Gaussian sequence experiment whose mean is a scalar multiple of $\mu_\ast$, and then applies the composite marginal likelihood theory of Section 2. It therefore avoids the full nonlinear marginal likelihood, but estimates only the prior up to the scalar factor $|\bar\delta|$ and does not immediately address the posterior inference problem in the regime $m\asymp n$.

\begin{remark}
Some technical remarks:
\begin{enumerate}
    \item The rate $n^{-1/4}$ in Theorem \ref{thm:eb_gd_hellinger} can be understood as resulting from taking $\delta_n\asymp n^{-1/2}$ in Theorem \ref{thm:MML_rate}, due to the error incurred by the distributional approximation in \eqref{eqn:one_step_gd_dist}. Interestingly, this rate is indeed observed in the numerical experiments in the right panel of Figure \ref{fig:numerics_regression} in Section \ref{sec:numerics}. We therefore conjecture that the $n^{-1/4}$ rate in Theorem \ref{thm:eb_gd_hellinger} cannot be removed for free.
    \item It is possible that $\bar{\delta}=0$, depending on the model characteristics and the choice of the initial guess prior $G_0 \in \mathscr{G}$. For instance, for the noiseless phase retrieval model with $\mathsf{F}(x,\xi) = x^2$ and loss derivative $\mathsf{L}(x,y)=(\d/\d x) (y-x^2)^2/4 =x(x^2-y)$, we have $\mathfrak{S}(u_0,v_0,\xi_0) = \mathsf{L}(u_0,v_0^2)= u_0(u_0^2-v_0^2)$. This means that, for centered $G_0,G_\ast$, we have $\bar{\delta} = 2 \phi a_\Sigma \cov(G_0,G_\ast)$, which requires an informative initial guess prior $G_0$ for estimating $G_\ast$. 
    
    This phenomenon is well understood in the literature, as gradient descent with uninformative random initialization requires $\Omega(\log n)$ iterations to become correlated with the signal \cite{chen2019gradient,han2025long}. It remains open to extend our one-step debiased gradient descent proposal in Algorithm \ref{alg:eb_gd} to accommodate such scenarios.
    \item While the smoothness assumption in Assumption \ref{assump:model} formally excludes the logistic regression example for (\ref{def:model}), we believe that a smoothing technique similar to that developed in \cite[Section 5]{han2026gradient} may be employed to rigorously extend our results to the logistic regression setting.
\end{enumerate}
\end{remark}

\section{Numerical experiments}\label{sec:numerics}

\subsection{Simulation designs}

This section illustrates the finite-sample behavior of the CML estimator for the empirical Bayes credible intervals and the marginal empirical Bayes regret in Sections \ref{subsec:credible_consequence} and \ref{subsec:marginal_regret}, and for the two main regression applications in Sections \ref{sec:gls_application} and \ref{sec:gd_application}. For the purpose of illustration, we adopt the fixed-grid approximation method to compute the Kiefer-Wolfowitz NPMLE as in \cite{jiang2009general}. This algorithm is summarized in Algorithm \ref{alg:grid_em_mple} in Appendix \ref{section:EM} for the reader's convenience. For more recent implementations, the readers are referred to, e.g., the \texttt{REBayes} method \cite{koenker2017rebayes} and the \texttt{EBNM} method \cite{willwerscheid2021ebnm}.

The concrete simulation settings are as follows:
\begin{enumerate}
	\item The first application studies the empirical Bayes credible interval \eqref{eq:eb_ci_original_U} and the marginal empirical Bayes regret \eqref{eq:regret_identity_text} directly in the correlated Gaussian sequence model. We take
	\begin{align*}
	\beta_{0,j}\stackrel{\mathrm{i.i.d.}}{\sim}\mathrm{Unif}[-1,1],
	\qquad U=\beta_0+\sigma \mathsf{Z}_n,
	\qquad \sigma=0.75,
	\end{align*}
	where $\mathsf{Z}_n \sim \mathcal{N}(0,\Lambda)$. The normalized covariance matrix $\Lambda$ is chosen to be block-equicorrelated with block size $b$ and within-block correlation $\rho$, so that $\kappa=\|\Lambda\|_{\op}=1+(b-1)\rho$ and $n_\ast=n/\kappa$. We estimate $G_0$ by the grid CML estimator, compute the equal-tailed $90\%$ credible intervals via \eqref{eq:eb_ci_original_U}, and compute the marginal empirical Bayes regret in \eqref{eq:regret_identity_text} by numerical integration in the one-dimensional Gaussian experiment.
	\item The second application considers the two regression settings. For linear generalized least squares (GLS), we observe
		\begin{align*}
		U=\mu_\ast+\tau_\ast \mathsf{Z}_n,
		\qquad \tau_\ast=0.8,
		\qquad \mathsf{Z}_n\sim \mathcal{N}(0,\Lambda),
		\end{align*}
		with $\mu_{\ast,j}\stackrel{\mathrm{i.i.d.}}{\sim}0.25\delta_{-1.5}+0.5\delta_0+0.25\delta_{1.5}$. The composite marginal likelihood uses the oracle $\tau_\ast$, in order to isolate the statistical behavior of the CML estimator from residual-variance plug-in effects.
	
	   For the debiased GD in nonlinear regression, we run Algorithm \ref{alg:eb_gd} on the high-dimensional nonlinear regression data
	   \begin{align*}
	   Y_i=\tanh\big(1.25\iprod{A_i}{\mu_\ast}\big)+\xi_i,
	   \qquad A_i\stackrel{\mathrm{i.i.d.}}\sim \mathcal{N}(0,\Sigma/n),
	   \quad i\in[m],
	   \end{align*}
	   where $m=\lceil 1.2 n\rceil$, $\xi_i\stackrel{\mathrm{i.i.d.}}\sim \mathcal{N}(0,1.2^2)$, and the coordinates of $\mu_\ast$ are i.i.d. from $G_\ast=0.25\delta_{-1}+0.5\delta_0+0.25\delta_1$. The initialization is set as $\mu^0=0$ and $\mathsf{L}(u,y)=u-y$. In this experiment, the normalized precision matrix $\Sigma^{-1}$ is chosen from the same block-equicorrelated family, because the relevant effective sample size in Theorem \ref{thm:eb_gd_hellinger} is determined by the correlation matrix of $\Sigma^{-1}$. The target prior is the scaled law $\bar G_\ast(\cdot)=G_\ast(\cdot/|\bar\delta|)$, where with $(V,\xi) \sim \mathcal{N}(0,a_\Sigma\mathrm{Var}(G_\ast))\otimes \mathcal{N}(0,1.2^2)$ and $a_\Sigma=\mathrm{tr}(\Sigma)/n$, we have
	   \begin{align*}
	   \bar\delta=1.2\cdot 1.25\,\E\mathrm{sech}^2(1.25V),
	   \qquad
	   \bar\sigma^2=1.2\,\E\big(\tanh(1.25V)+\xi\big)^2.
	   \end{align*}
\end{enumerate}
\begin{figure}
	\centering
	\IfFileExists{eb_mple_credible_regret.pdf}{%
		\includegraphics[width=.98\textwidth]{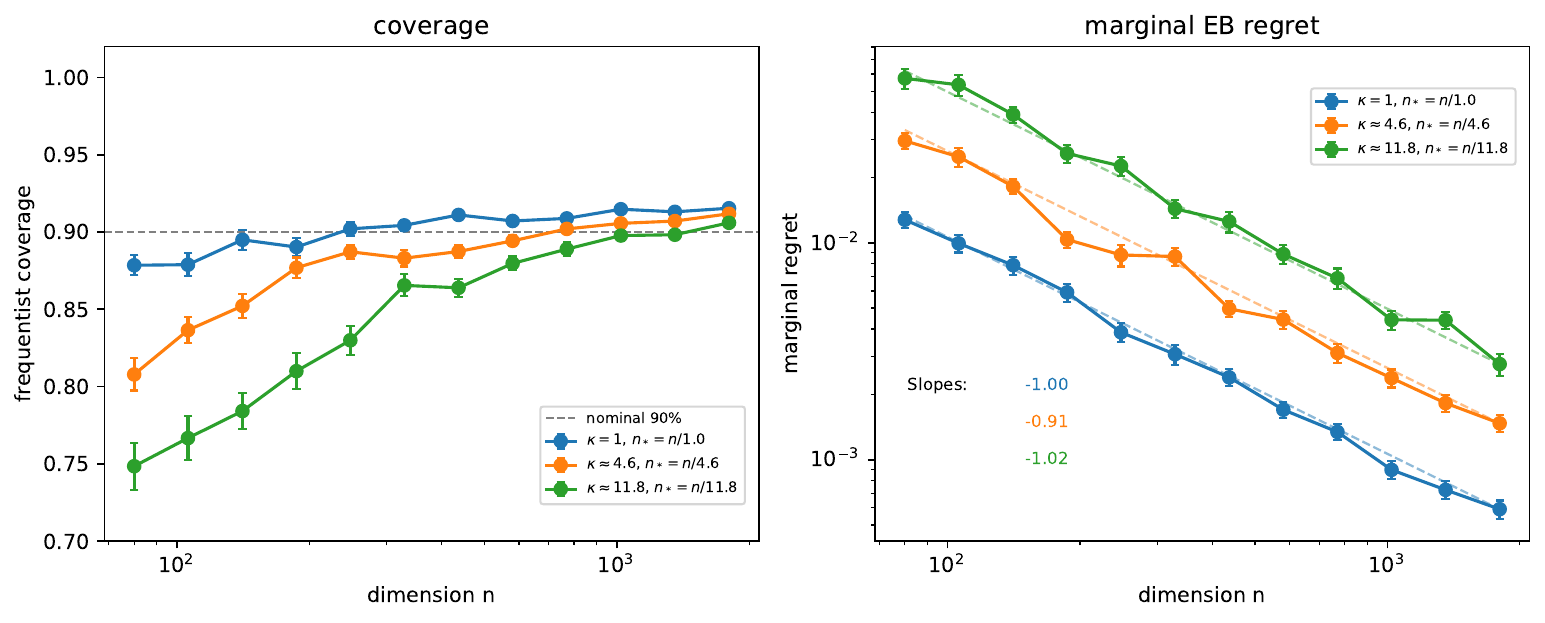}%
	}{%
		\fbox{\parbox{.9\textwidth}{\centering Figure file \texttt{eb\_mple\_credible\_regret.pdf} not found.}}%
	}
	\caption{\emph{Left}: average coverage for the nominal $90\%$ empirical Bayes credible intervals in \eqref{eq:eb_ci_original_U}. \emph{Right}: marginal empirical Bayes regret in \eqref{eq:regret_identity_text}. }
	\label{fig:numerics_credible}
\end{figure}

For the credible-interval and marginal-regret experiments in Figure \ref{fig:numerics_credible}, we use a logarithmic grid of sample sizes between $80$ and $1800$. For the regression experiments in Figure \ref{fig:numerics_regression}, we use the same grid of sample sizes and compare the moderate dependence levels $\kappa=1$, $\kappa\approx 1.9$, and $\kappa\approx 2.8$. The Monte Carlo averages and standard errors are computed over $100$ repetitions.

\subsection{Simulation results}

Figure \ref{fig:numerics_credible} reports the Monte Carlo average of the realized coverage $n^{-1}\sum_j\bm 1_{\{\beta_{0,j}\in\widehat C_j(U)\}}$ and the marginal empirical Bayes regret \eqref{eq:regret_identity_text} for several choices of the normalized covariance matrix $\Lambda$ in application (1). The plotted curves compare three choices: $\kappa=1$, $\kappa\approx4.6$, and $\kappa\approx11.8$. The left panel shows that the frequentist coverage improves as $n_\ast$ increases. Equivalently, for larger $\kappa$, a larger $n$ is needed for the empirical Bayes intervals to approach the nominal coverage. The right panel shows that the marginal empirical Bayes regret decreases approximately at the $n_\ast^{-1}$ rate predicted by Proposition \ref{prop:marginal_regret}; the fitted slopes are close to $-1$ for all three values of $\kappa$.

\begin{figure}
	\centering
	\IfFileExists{eb_mple_regression.pdf}{%
		\includegraphics[width=.98\textwidth]{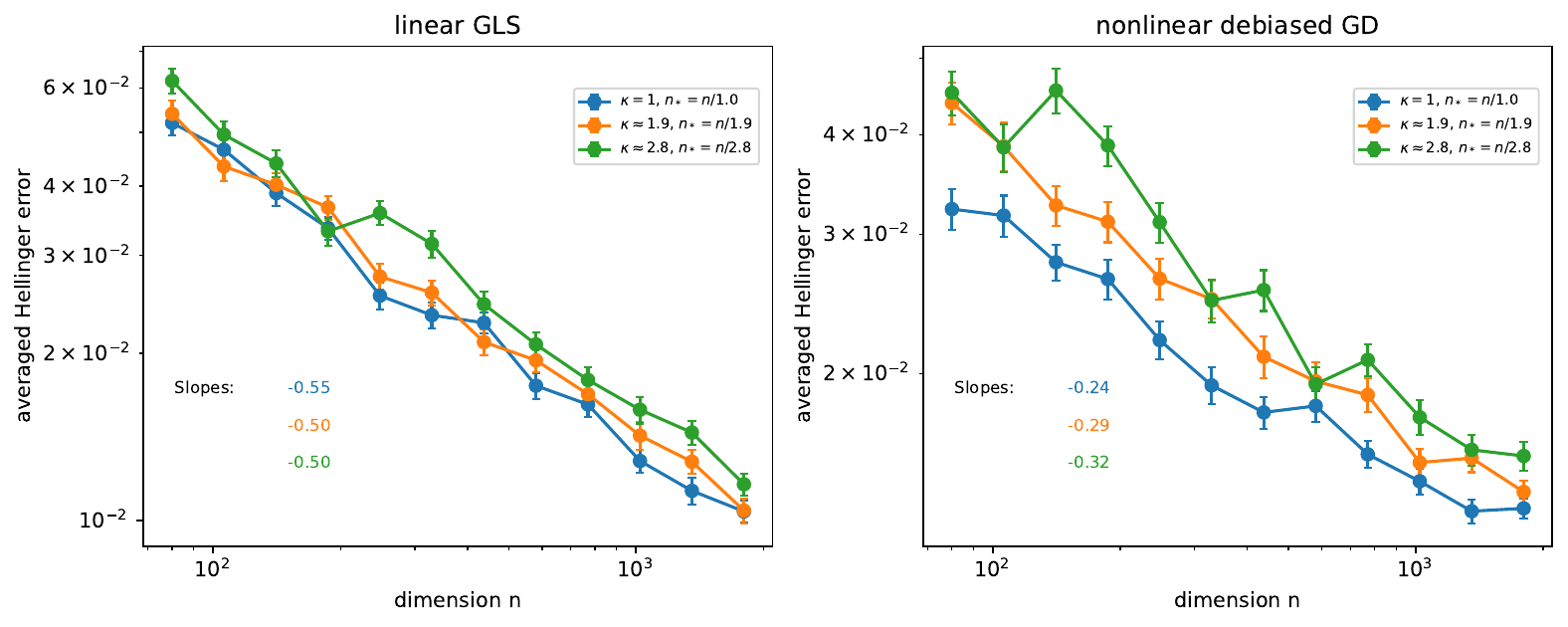}%
	}{%
		\fbox{\parbox{.9\textwidth}{\centering Figure file \texttt{eb\_mple\_regression.pdf} not found.}}%
	}
	\caption{ \emph{Left}: linear GLS with oracle noise scale. \emph{Right}: one-step debiased GD run on the original high-dimensional nonlinear regression model. }
	\label{fig:numerics_regression}
\end{figure}

Figure \ref{fig:numerics_regression} reports the averaged Hellinger error for the two regression applications in (2), with all three values of $\kappa$ displayed in a single two-panel figure. In both panels, the curves move upward as $\kappa$ increases, consistent with the theoretical prediction that the convergence rate is governed by the effective sample size $n_\ast=n/\kappa$. The left panel shows slopes close to $-1/2$ for the GLS experiment, in line with Theorem \ref{thm:gls_eb}. The right panel shows slower slopes, now close to the $-1/4$ side for all three moderate dependence levels, for the one-step debiased GD experiment. This is consistent with the additional distributional approximation error in Theorem \ref{thm:eb_gd_hellinger}.

\section{Proof of Theorem \ref{thm:MML_rate}}\label{section:proof_MML_rate}

\subsection{Notation and preliminary estimates}
Throughout this section we write $\mathfrak{d}_H$ for $\mathfrak{d}_{H;\sigma_{0,[n]}}$ and
$\mathfrak L_n$ for $\mathfrak{L}_{n;\sigma_{0,[n]}}$ whenever no confusion can
arise. For $G\in\mathscr G$ and $j\in[n]$, write
\begin{align*}
 p_{G,j}(x)\equiv \varphi_{G;\sigma_{0,j}}(x),\quad
 p_{0,j}(x)\equiv \varphi_{G_0;\sigma_{0,j}}(x),\quad
 \ell_{G,j}(x)\equiv \log \frac{p_{G,j}(x)}{p_{0,j}(x)}.
\end{align*}
Let $\chi\in C^\infty(\R)$ be a smooth non-decreasing function such that
$\chi(x)=x$ for $|x|\le 1$, $\chi(x)=2\sign(x)$ for $|x|\ge 3$,
$\pnorm{\chi}{\infty}\le 2$, and $\pnorm{\chi'}{\infty}\le c_\chi$ for some
universal constant $c_\chi>0$. For $T>0$, let $\chi_T(x)\equiv T\chi(x/T)$, and
write
\begin{align*}
 \ell^T_{G,j}(x)\equiv \chi_T(\ell_{G,j}(x)),\quad
 \mathfrak L_{n;T}(G)\equiv \frac{1}{\sqrt n}\sum_{j\in[n]}\ell^T_{G,j}(U_j).
\end{align*}

We begin with elementary estimates for normal mixtures.
\begin{lemma}\label{lem:normal_mix_basic}
Suppose $M^{-1}\leq \sigma_{0,j}^2\leq M$ for all $j\in[n]$ and
$G_0\in\mathscr G(M,\alpha)$.  Let $G\in\mathscr G(R,\infty)$ with $R\geq 1$.
Then, for every $x\in\R$,
\begin{align}\label{ineq:normal_mix_lower}
\inf_{G\in\mathscr G(R,\infty)}\inf_{j\in[n]}p_{G,j}(x)
\ge C_M^{-1}\exp\{-C_M(\abs{x}+R)^2\}.
\end{align}
Moreover, for every $B\geq 1$, all $\abs{x}\leq B$, and all $G\in\mathscr G(R,\infty)$,
\begin{align}\label{ineq:normal_mix_log_growth}
\biggabs{\log\frac{p_{G,j}(x)}{p_{0,j}(x)}}
\le C_{\alpha,M}\{1+R^2+B^2\}.
\end{align}
Finally, if $X\sim p_{0,j}$, then uniformly in $j\in[n]$,
\begin{align}\label{ineq:p0_tail_exp}
\Prob(|X|>B)\leq C_{\alpha,M}\exp\big[-(B^2\wedge B^\alpha)/C_{\alpha,M}\big],
\end{align}
with the second term interpreted as $\infty$ when $\alpha=\infty$.
\end{lemma}

\begin{proof}
For \eqref{ineq:normal_mix_lower}, as for every
$u\in[-R,R]$ we have
$\varphi_{\sigma_{0,j}}(x-u)\geq C_M^{-1}\exp\{-C_M(\abs{x}+R)^2\}$, integration with respect to $G$ then proves the lower bound.

For \eqref{ineq:normal_mix_log_growth}, since $\int \exp\{\abs{u}^\alpha/M^\alpha\}\,G_0(\d u)\leq M$, Markov's inequality gives
$G_0([-a_{\alpha,M},a_{\alpha,M}])\geq 1/2$ for a constant $a_{\alpha,M}<\infty$
(and in the case $\alpha=\infty$ one may take $a_M=M$). Hence, for $|x|\leq B$,
\begin{align*}
p_{0,j}(x)\geq \int_{|u|\leq a_{\alpha,M}}\varphi_{\sigma_{0,j}}(x-u)\, G_0(\d u)
\geq C_{\alpha,M}^{-1}\exp\{-C_{\alpha,M}(1+B^2)\}.
\end{align*}
The corresponding upper bounds $p_{G,j}(x),p_{0,j}(x)\leq C_M$ and
\eqref{ineq:normal_mix_lower} for $p_{G,j}$ yield
\begin{align*}
\abs{\log p_{G,j}(x)}+\abs{\log p_{0,j}(x)}
\leq C_{\alpha,M}(1+R^2+B^2),\quad \abs{x}\le B,
\end{align*}
which proves \eqref{ineq:normal_mix_log_growth}. 
	
For \eqref{ineq:p0_tail_exp}, if
$X=\theta+\sigma_{0,j}\mathsf{Z}$ with $\theta\sim G_0$ and $\mathsf{Z}\sim \mathcal{N}(0,1)$ independent, then
\begin{align*}
\Prob(\abs{X}>B)\leq
\Prob(\abs{\theta}>B/2)+\Prob(\sqrt M \abs{\mathsf{Z}}>B/2).
\end{align*}
The first term is controlled by the exponential moment assumption and the second by
the Gaussian tail bound, proving \eqref{ineq:p0_tail_exp}.
\end{proof}

\subsection{Bernstein inequality for truncated composite marginal likelihood ratio}

We prove a Bernstein inequality for the truncated composite marginal likelihood ratio.
\begin{proposition}\label{prop:fixed_G_bernstein}
Suppose the assumptions of Theorem \ref{thm:MML_rate} hold, and let
$\kappa_0,n_\ast$ be as defined in (\ref{def:kappa_n_ast}). There exists
$C_{\alpha,M}>1$ such that for every $T\ge1$, every $G\in\mathscr G$, and every
$x\ge1$,
\begin{align*}
\Prob\bigg(\abs{(\mathrm{id}-\E)\mathfrak L_{n;T}(G)}
> C_{\alpha,M}T\cdot \bigg[\sqrt{\kappa_0}\,
\mathfrak{d}_{H;\sigma_{0,[n]}}(G,G_0)\sqrt{x}
+\frac{\kappa_0 x}{\sqrt n}\bigg]\bigg)
\le e^{-x}.
\end{align*}
\end{proposition}

We need the following Gaussian H\"older-type decoupling inequality.

\begin{lemma}\label{lem:gaussian_holder}
Under the assumptions of Theorem \ref{thm:MML_rate}, the following decoupling
inequality holds: for all nonnegative measurable
$\{f_j\}$'s,
\begin{align}\label{ineq:correlated_U_decouple}
\E\prod_{j \in [n]} f_j(U_j)
\le \prod_{j \in [n]}\big(\E_{X\sim p_{0,j}} f_j(X)^{\kappa_0}\big)^{1/\kappa_0}.
\end{align}
\end{lemma}

\begin{proof}
We use the following version of the geometric Brascamp-Lieb inequality for Gaussian measures, appearing \cite[Theorem 1, Eqn. (1.4)]{chen2015improved}: if
a centered Gaussian vector $X$ with covariance $\Sigma$ satisfies
$\Sigma\preceq \kappa \mathfrak D_\Sigma$, then
\begin{align*}
\E\prod_{j \in [n]} f_j(X_j)\leq \prod_{j \in [n]}\big(\E f_j(X_j)^\kappa\big)^{1/\kappa}.
\end{align*}
By definition of $\kappa_0$ in (\ref{def:kappa_n_ast}), $
\Sigma_0\preceq \kappa_0\mathfrak D_{\Sigma_0}$.
Conditional on $\beta_0$, the vector $U-\beta_0$ is centered Gaussian with covariance
$\Sigma_0$ and coordinate variances $\sigma_{0,j}^2$. Applying the preceding inequality
conditionally with $\kappa=\kappa_0$ gives
\begin{align*}
\E\bigg[\prod_{j \in [n]} f_j(U_j)\,\big|\,\beta_0\bigg]
\leq \prod_{j \in [n]}
\Big(\E_{\mathsf Z\sim N(0,1)}
f_j(\beta_{0,j}+\sigma_{0,j}\mathsf Z)^{\kappa_0}\Big)^{1/\kappa_0}.
\end{align*}
The right-hand side is a product of functions of the independent variables
$\beta_{0,j}$. Taking expectation in $\beta_0$ and applying Jensen's inequality to
the concave map $x\mapsto x^{1/\kappa_0}$ yields \eqref{ineq:correlated_U_decouple}.
\end{proof}

\begin{proof}[Proof of Proposition \ref{prop:fixed_G_bernstein}]
In the proof we shall write $\mathfrak{d}_{H;\sigma_{0,[n]}} = \mathfrak{d}_H$.

\noindent (\textbf{Step 1}). In this step, we prove that with $Y_{G,j}\equiv \ell^T_{G,j}(X_j)$ with $X_j\sim p_{0,j}$, there exists some universal constant $C>0$ such that for every $s$ with $|s|\le (8T)^{-1}$,
\begin{align}\label{ineq:scalar_mgf_bound}
\log \E\exp\{s(Y_{G,j}-\E Y_{G,j})\}
\le C s^2T^2d_H^2(p_{G,j},p_{0,j}).
\end{align}
To this end, first note that for $a>0$ and $T\ge1$,
\begin{align}\label{ineq:chi_hellinger_pointwise}
\{\chi_T(\log a)\}^2\le C T^2(\sqrt a-1)^2.
\end{align}
Indeed, if $a\in[1/4,4]$, then $|\log a|\le C|\sqrt a-1|$; if
$a\notin[1/4,4]$, then the left hand side is bounded by $4T^2$ and
$(\sqrt a-1)^2$ is bounded below by a positive universal constant.
As
$|Y_{G,j}|\le2T$, by \eqref{ineq:chi_hellinger_pointwise},
\begin{align}\label{ineq:fixed_G_second_moment}
\E Y_{G,j}^2
&\le CT^2\int \bigg(\sqrt{\frac{p_{G,j}}{p_{0,j}}}-1\bigg)^2p_{0,j}
=2C\cdot T^2d_H^2(p_{G,j},p_{0,j}).
\end{align}
Now using $|Y_{G,j}-\E Y_{G,j}|\le4T$ and
$e^u-1-u\le C u^2$ for $|u|\le1/2$, we are led to 
\begin{align*}
\E \exp \{s(Y_{G,j}-\E Y_{G,j})\}
&\le 1+Cs^2\E(Y_{G,j}-\E Y_{G,j})^2
\le 1+Cs^2\E Y_{G,j}^2,
\end{align*}
and \eqref{ineq:scalar_mgf_bound} follows from \eqref{ineq:fixed_G_second_moment}.

\noindent (\textbf{Step 2}). 
Let $\lambda\ge0$ satisfy $\kappa_0\lambda/\sqrt n\le(8T)^{-1}$. Applying
\eqref{ineq:correlated_U_decouple} to
\begin{align*}
f_j(u)=\exp\bigg\{\frac{\lambda}{\sqrt n}\big(\ell^T_{G,j}(u)-
\E_{X\sim p_{0,j}}\ell^T_{G,j}(X)\big)\bigg\}
\end{align*}
and using \eqref{ineq:scalar_mgf_bound} with $s=\kappa_0\lambda/\sqrt n$ gives
\begin{align}\label{ineq:fixed_G_process_mgf}
&\log \E \exp\big\{\lambda(\mathfrak L_{n;T}(G)-\E\mathfrak L_{n;T}(G))\big\} \nonumber\\
&= \log \E \prod_{j \in [n]} f_j(U_j)\leq \frac{1}{\kappa_0}\sum_{j \in [n]} \log \E f_j(U_j)^{\kappa_0}\nonumber\\
&\leq \frac{C}{\kappa_0}\sum_{j \in [n]}  \frac{\kappa_0^2\lambda^2}{n} T^2 d_H^2(p_{G,j},p_{0,j})\leq  C\kappa_0\lambda^2T^2\mathfrak d_H^2(G,G_0).
\end{align}
Therefore, for all $0\le\lambda\le c\sqrt n/(\kappa_0T)$,
\begin{align*}
\Prob\{ (\mathrm{id}-\E)\mathfrak L_{n;T}(G)>t\}
\le \exp\{-\lambda t+C\kappa_0\lambda^2T^2\mathfrak d_H^2(G,G_0)\}.
\end{align*}
Optimizing over this interval yields
\begin{align}\label{ineq:fixed_G_tail_intermediate}
\Prob\{(\mathrm{id}-\E)\mathfrak L_{n;T}(G)>t\}
\le
\exp\bigg[-c\min\bigg\{
\frac{t^2}{\kappa_0T^2\mathfrak d_H^2(G,G_0)},
\frac{t\sqrt n}{\kappa_0T}\bigg\}\bigg],
\end{align}
with the usual convention that the first term is $+\infty$ if
$\mathfrak d_H(G,G_0)=0$. Taking
\begin{align*}
t=C_{\alpha,M}T\bigg\{\sqrt{\kappa_0}\,\mathfrak d_H(G,G_0)\sqrt x+
\frac{\kappa_0 x}{\sqrt n}\bigg\}
\end{align*}
and applying the same argument to $-\mathfrak L_{n;T}(G)$ proves the proposition.
\end{proof}

\subsection{Local maximal inequality for truncated composite marginal likelihood ratio process}

From the Bernstein inequality in Proposition \ref{prop:fixed_G_bernstein}, we may prove the following local maximal inequality for the  truncated composite marginal likelihood ratio process.
\begin{proposition}\label{prop:localized_truncated_process}
Suppose the assumptions of Theorem \ref{thm:MML_rate} hold. Fix $D>0$, $A>1$,
and a deterministic $T\in[1,(\log n)^{c_0}]$, where $c_0>0$ is fixed. There
exists a constant $c_1=c_1(\alpha,M,D,A,c_0)>1$ such that, with probability at least
$1-c_1n^{-D}$, the following holds simultaneously for all
$R\in[M,A L_{n,\alpha}]$, all $\mathscr G_{\texttt{test}}\subset\mathscr G$,
and all $r\in[c_1n_\ast^{-1/2}(\log n)^{c_1},1]$:
\begin{align*}
\sup_{\substack{G\in\mathscr G(R,\infty)\cap\mathscr G_{\texttt{test}}: \mathfrak{d}_{H;\sigma_{0,[n]}} (G,G_0)\le r}}
\abs{(\mathrm{id}-\E)\mathfrak L_{n;T}(G)}
\le c_1\cdot \sqrt{\kappa_0}r\cdot (\log n)^{c_1}.
\end{align*}
\end{proposition}

We need a finite sieve and explicit approximation estimates.
\begin{lemma}\label{lem:sieve_approx}
Suppose $M^{-1}\leq \sigma_{0,j}^2\leq M$ for all $j\in[n]$ and
$R\ge M$. For every $\eta\in(0,1)$ there exists a finite set
$\mathcal N_\eta(R)\subset\mathscr G(R,\infty)$ with
\begin{align}\label{ineq:sieve_cardinality}
\log |\mathcal N_\eta(R)|\le C_M(1+R)^{C_M}\log^2(C_M/\eta)
\end{align}
such that for every $G\in\mathscr G(R,\infty)$ there is a $\pi_\eta G\in\mathcal N_\eta(R)$
satisfying
\begin{align}\label{ineq:sieve_supnorm}
\max_{j\in[n]}\pnorm{p_{G,j}-p_{\pi_\eta G,j}}{\infty}\le\eta.
\end{align}
Moreover, if $B\ge R+1$ and $T\ge1$, then for such a pair $(G,\pi_\eta G)$,
\begin{align}
&\max_{j\in[n]}\sup_{|x|\le B}
\abs{\ell^T_{G,j}(x)-\ell^T_{\pi_\eta G,j}(x)}
\le C_M\cdot \eta e^{C_M(B+R)^2}, \label{ineq:sieve_empirical_diff}\\ 
&\max_{j\in[n]}\E_{X\sim p_{0,j}}
\abs{\ell^T_{G,j}(X)-\ell^T_{\pi_\eta G,j}(X)}\le C_M\cdot \big(\eta e^{C_M(B+R)^2}+T e^{-B^2/C_M}\big), \label{ineq:sieve_expect_diff}
\end{align}
and
\begin{align}\label{ineq:sieve_hellinger_diff}
\max_{j\in[n]}d_H^2(p_{G,j},p_{\pi_\eta G,j})
\le C_M\cdot \big(B\eta+ e^{-(B-R)^2/C_M}\big).
\end{align}
\end{lemma}

\begin{proof}
The existence of $\mathcal N_\eta(R)$ satisfying \eqref{ineq:sieve_cardinality}
and \eqref{ineq:sieve_supnorm} is exactly the normal-mixture discretization
lemma of \cite[Lemma 2]{zhang2009generalized}, applied uniformly over
$\sigma_{0,j}\in[M^{-1/2},M^{1/2}]$.
We prove the remaining estimates below. 

For (\ref{ineq:sieve_empirical_diff}), let $H\equiv\pi_\eta G$. By
\eqref{ineq:normal_mix_lower}, for $|x|\le B$, we have 
$\min\{p_{G,j}(x),p_{H,j}(x)\}
\ge C_M^{-1} e^{-C_M(B+R)^2}$.
Together with \eqref{ineq:sieve_supnorm}, this gives
\begin{align*}
\abs{\log p_{G,j}(x)-\log p_{H,j}(x)}
&\le \frac{|p_{G,j}(x)-p_{H,j}(x)|}{\min\{p_{G,j}(x),p_{H,j}(x)\}}\le C_M\cdot \eta e^{C_M(B+R)^2}.
\end{align*}
Since $\chi_T$ is $c_\chi$-Lipschitz for all $T$, \eqref{ineq:sieve_empirical_diff} follows.

For \eqref{ineq:sieve_expect_diff}, we shall split the integral over $\{|X|\le B\}$ and its
complement. The central part is bounded by \eqref{ineq:sieve_empirical_diff}. On
the complement, both truncated log-likelihoods are bounded by $2T$ in absolute
value. Since $X=\theta+\sigma_{0,j}\mathsf{Z}$ with $|\theta|\le M$ and
$\sigma_{0,j}\le M^{1/2}$, we have $
\sup_{j\in[n]}\Prob_{X\sim p_{0,j}}(|X|>B)
\le C_M e^{-B^2/C_M}$
for $B\ge C_M$. This proves \eqref{ineq:sieve_expect_diff}.

For (\ref{ineq:sieve_hellinger_diff}), note that $
2d_H^2(p_{G,j},p_{H,j})\le \int |p_{G,j}-p_{H,j}|$. 
On $[-B,B]$, the integral is at most $2B\eta$. On $[-B,B]^c$, both densities are
normal mixtures with mixing distributions supported on $[-R,R]$ and variances in
$[M^{-1},M]$, so their total tail mass is bounded by
$C_M e^{-(B-R)^2/C_M}$. This proves \eqref{ineq:sieve_hellinger_diff}.
\end{proof}

\begin{proof}[Proof of Proposition \ref{prop:localized_truncated_process}]
In the proof we shall write  $\mathfrak d_H=\mathfrak d_{H;\sigma_{0,[n]}}$.

Fix $R\in[M,A L_{n,\alpha}]$, $T\in[1,(\log n)^{c_0}]$, and $r\in(0,1]$.
Let $B=B_{\alpha,M,D}L_{n,\alpha}$, with $B_{\alpha,M,D}$ large enough that
$B\ge 2R+1$ and the tail probabilities in \eqref{ineq:p0_tail_exp} are at most
$n^{-20}$.  Let $\eta=n^{-A_\eta}\exp(-A_\eta L_{n,\alpha}^2)$, where $A_\eta$ will be chosen large.  For the
net $\mathcal N_\eta(R)$ in Lemma \ref{lem:sieve_approx}, define
\begin{align*}
\mathcal N_\eta(R,r)=\{H\in\mathcal N_\eta(R):H=\pi_\eta G
\text{ for some }G\in\mathscr G(R,\infty),\ \mathfrak d_H(G,G_0)\le r\}.
\end{align*}
By \eqref{ineq:sieve_hellinger_diff}, after increasing $A_\eta$, every
$H\in\mathcal N_\eta(R,r)$ satisfies $\mathfrak d_H(H,G_0)\le r+n^{-10}$.
For a fixed such $H$, Proposition \ref{prop:fixed_G_bernstein} with
\begin{align*}
x=C_{\alpha,M,D}(1+R)^{C_M}\log^2(C_M/\eta)+(D+4)\log n
\end{align*}
gives, with probability at least $1-\exp(-x)$,
\begin{align*}
\abs{(\mathrm{id}-\E)\mathfrak L_{n;T}(H)}
\le C_{\alpha,M}T\bigg\{\sqrt{\kappa_0}(r+n^{-10})\sqrt x+
\frac{\kappa_0 x}{\sqrt n}\bigg\}.
\end{align*}
Because $R\le A L_{n,\alpha}$, $T\le(\log n)^{c_0}$, and
$\eta=n^{-A_\eta}\exp(-A_\eta L_{n,\alpha}^2)$, the number $x$ is bounded by $(\log n)^{C_{\alpha,M,D}}$.
If $r\ge c n_\ast^{-1/2}(\log n)^c$, then
\begin{align*}
T\bigg\{\sqrt{\kappa_0}r\sqrt x+\frac{\kappa_0 x}{\sqrt n}\bigg\}
\le C_{\alpha,M,D}\cdot \sqrt{\kappa_0}r(\log n)^{C_{\alpha,M,D}}.
\end{align*}
A union bound over the net, using \eqref{ineq:sieve_cardinality}, gives the same
bound simultaneously over $\mathcal N_\eta(R,r)$ with failure probability at most
$n^{-D-3}$.  Passing from $G$ to its net representative $H=\pi_\eta G$ is exactly
as in Lemma \ref{lem:sieve_approx}: on the event
$E_B=\{\max_j |U_j|\le B\}$, \eqref{ineq:sieve_empirical_diff} gives
\begin{align*}
|\mathfrak L_{n;T}(G)-\mathfrak L_{n;T}(H)|
\le \sqrt n\, C_M\eta e^{C_M(B+R)^2}\le n^{-9},
\end{align*}
while \eqref{ineq:sieve_expect_diff} gives the
same $n^{-9}$ bound for the expectations.  Combining the above arguments, for all fixed $R,T,r$,
\begin{align*}
\sup_{G \in \mathscr{G}(R,\infty):\mathfrak{d}_H(G,G_0)\leq r} \abs{\mathfrak L_{n;T}(G)-\E\mathfrak L_{n;T}(G)}\leq C_{\alpha,M,D}\cdot \sqrt{\kappa_0}r(\log n)^{C_{\alpha,M,D}}
\end{align*}
holds on the event $E_B$ and outside a set of probability at most $n^{-D-3}$. The complement of $E_B$ has
probability at most $n^{-D-3}$ by \eqref{ineq:p0_tail_exp} and a union bound.

Finally, take a union bound over dyadic $R\in[M,A L_{n,\alpha}]$ and dyadic
$r\in[c_1n_\ast^{-1/2}(\log n)^{c_1},1]$.  Increasing $c_1$ proves the claim.
\end{proof}

\subsection{Proof of Theorem \ref{thm:MML_rate}}
The following lemma removes the truncation effect at logarithmic level.
\begin{lemma}\label{lem:remove_truncation}
Suppose the assumptions of Theorem \ref{thm:MML_rate} hold. Fix $D>0$ and $A>1$, and set
$R_n\equiv A L_{n,\alpha}$. There exists a constant $c_1=c_1(\alpha,M,D,A)>1$ such
that, with $T_n\equiv c_1 L_{n,\alpha}^2$, it holds with probability at least
$1-c_1 n^{-D}$ that
\begin{align}\label{ineq:trunc_emp_inactive}
\mathfrak L_{n;\sigma_{0,[n]}}(G)=\mathfrak L_{n;T_n}(G),\quad
\hbox{for all }G\in\mathscr G(R_n,\infty),
\end{align}
and, deterministically,
\begin{align}\label{ineq:trunc_expect_close}
\sup_{G\in\mathscr G(R_n,\infty)}
\abs{\E\mathfrak L_{n;\sigma_{0,[n]}}(G)-\E\mathfrak L_{n;T_n}(G)}\le n^{-10}.
\end{align}
\end{lemma}

\begin{proof}
Let $B=C_{\alpha,M,D}L_{n,\alpha}$.  By \eqref{ineq:p0_tail_exp} and a union
bound over $j$, $\Prob(\max_j|U_j|>B)\le Cn^{-D}$.  On the event
$\max_j|U_j|\le B$, (\ref{ineq:normal_mix_log_growth}) in Lemma \ref{lem:normal_mix_basic} gives, uniformly over
$G\in\mathscr G(R_n,\infty)$ and $j\in[n]$,
\begin{align*}
|\ell_{G,j}(U_j)|\le C_{\alpha,M}(1+R_n^2+B^2)\le C_{\alpha,M,D,A}L_{n,\alpha}^2.
\end{align*}
Choosing $T_n$ larger than the last display proves \eqref{ineq:trunc_emp_inactive}.

For the expectation bound, the truncation is inactive on $\{\abs{x}\leq B\}$.  On the
complement $\{\abs{x}>B\}$, using \eqref{ineq:normal_mix_log_growth} and the definition of $T_n$,
\begin{align*}
\abs{\ell_{G,j}(x)-\ell^{T_n}_{G,j}(x)}\leq \abs{\ell_{G,j}(x)} + 2T_n\leq C_{\alpha,M,D,A}(1+x^2+L_{n,\alpha}^2). 
\end{align*}
Combining this envelope with the tail
bound \eqref{ineq:p0_tail_exp}, and increasing $C_{\alpha,M,D,A}$ in the definition
of $B$, gives
\begin{align*}
\sup_{G\in\mathscr G(R_n,\infty)}\sup_{j\in[n]}
\E_{X\sim p_{0,j}}|\ell_{G,j}(X)-\ell^{T_n}_{G,j}(X)|\le n^{-12}.
\end{align*}
Multiplying by $n^{-1/2}$ and summing over $j$ proves
\eqref{ineq:trunc_expect_close}.
\end{proof}

We also need the following result relating Wasserstein and Hellinger distances tailored to our setting. 
\begin{lemma}\label{lem:Wasserstein_from_Hellinger}
	Fix $p,L\ge 1$. Then there exists some $c_1=c_1(p)>1$ such that for any
	$G_1,G_2\in\mathscr G$ with $\mathrm{supp}(G_\ell)\subset[-L,L]$ $(\ell=1,2)$,
	and $\sigma_0=(\sigma_{0,j})_{j\in[n]}\in\R^n_{>0}$,
	\begin{align}\label{ineq:Wasserstein_from_Hellinger}
	\mathsf{W}_p(G_1,G_2)\le c_1\cdot \inf_{\delta\in(0,1)}
	\Big\{\delta+e^{c_1\pnorm{\sigma_0}{\infty}^2\delta^{-2}}\cdot
	L^{c_1}\mathfrak{d}_{H;\sigma_{0,[n]}}^{1/c_1}(G_1,G_2)\Big\}.
	\end{align}
\end{lemma}

The proof of the above lemma adapts ideas from \cite[Theorem 2]{nguyen2013convergence}, so will be deferred to Appendix \ref{section:proof_Wasserstein_from_Hellinger}.

\begin{proof}[Proof of Theorem \ref{thm:MML_rate}]
In the proof we write $\mathfrak d_H=\mathfrak d_{H;\sigma_{0,[n]}}$,
$\mathfrak L_n=\mathfrak L_{n;\sigma_{0,[n]}}$.

For part (1), choose $A=A(\alpha,M,D)$ large enough that
$\Prob(\max_j |U_j|>A L_{n,\alpha})\le n^{-D-2}$, which follows from
\eqref{ineq:p0_tail_exp}.  Combining Proposition
\ref{prop:localized_truncated_process} and Lemma \ref{lem:remove_truncation}, with
$R_n=A L_{n,\alpha}$ and $T_n=C L_{n,\alpha}^2$, gives an event
$E_{\mathrm{loc}}$ of probability at least $1-c_1n^{-D}$ on which, simultaneously
for all $\mathscr G_{\texttt{test}}\subset\mathscr G$ and all
$r\in[c_1n_\ast^{-1/2}(\log n)^{c_1},1]$,
\begin{align}\label{ineq:MML_rate_step2_0}
\sup_{\substack{G\in\mathscr G(R_n,\infty)\cap\mathscr G_{\texttt{test}}:
		\mathfrak d_H(G,G_0)\le r}}
|\mathfrak L_n(G)-\E\mathfrak L_n(G)|\le c_1\cdot \sqrt{\kappa_0}r(\log n)^{c_1}.
\end{align}
Enlarging $c_1$ so that $c_1\ge A$ gives \eqref{ineq:MML_localized_rate}.

For part (2), note the population inequality
\begin{align}\label{ineq:MML_rate_step2_1}
\E\mathfrak L_n(G)
=-\frac1{\sqrt n}\sum_{j \in [n]} \mathrm{KL}(p_{0,j},p_{G,j})
\le -\sqrt n\,\mathfrak d_H^2(G,G_0),\quad \forall G\in\mathscr G.
\end{align}
On $E_{\mathrm{loc}}\cap E_{\mathrm{supp}}$, with
$E_{\mathrm{supp}}=\{\mathrm{supp}(\hat G_n)\subset[-c_0L_{n,\alpha},c_0L_{n,\alpha}]\}$, we shall
apply the localized bound (\ref{ineq:MML_rate_step2_0}) with $R_n=(A\vee c_0)L_{n,\alpha}$.  Let
\begin{align*}
\epsilon_n=A_0\{n_\ast^{-1/2}(\log n)^{c_1}+\delta_n^{1/2}\},
\end{align*}
where $A_0$ is sufficiently large.  If
$\mathfrak d_H(\hat G_n,G_0)>\epsilon_n$, choose $k\ge0$ with
$2^k\epsilon_n<\mathfrak d_H(\hat G_n,G_0)\le2^{k+1}\epsilon_n$.  By
\eqref{ineq:MML_rate_step2_1} and the localized bound,
\begin{align*}
\mathfrak L_n(\hat G_n)\le
-\sqrt n\,2^{2k}\epsilon_n^2+
c_1\sqrt{\kappa_0}\cdot 2^{k+1}\epsilon_n(\log n)^{c_1}
\le -\frac12\sqrt n\,2^{2k}\epsilon_n^2,
\end{align*}
where the last inequality follows from
$\epsilon_n\ge A_0 n_\ast^{-1/2}(\log n)^{c_1}$ and
$n_\ast=n/\kappa_0$ after increasing logarithmic powers in $c_1$.  On the other
hand, near-optimality relative to the admissible point $G_0$ gives
$\mathfrak L_n(\hat G_n)\ge -\sqrt n\,\delta_n$, a contradiction for the same
large $A_0$.  Hence
$\mathfrak d_H(\hat G_n,G_0)\le \epsilon_n$, proving the weighted Hellinger bound
in \eqref{ineq:MML_Hellinger_rate}.  The unweighted bound follows from the
monotonicity of $\sigma\mapsto d_H(\varphi_{G;\sigma},\varphi_{H;\sigma})$.

Finally, we prove the Wasserstein bound. Let $
\Delta_{n,\delta}\equiv (n_\ast^{-1/2}+\delta_n^{1/2})(\log n)^{c_2}$.
The Hellinger part already proved gives
$\mathfrak d_H(\hat G_n,G_0)\le C\Delta_{n,\delta}$ on
$E_{\mathrm{loc}}\cap E_{\mathrm{supp}}$.  If $\alpha=\infty$, then $G_0$ is already
compactly supported.  If $\alpha<\infty$, let $\Pi_L(x)=(-L)\vee x\wedge L$ and
$G_0^{(L)}=(\Pi_L)_\#G_0$, with $L=C L_{n,\alpha}$ and $C$ sufficiently large.
The exponential moment assumption implies
\begin{align*}
\mathsf W_p(G_0,G_0^{(L)})\le n^{-10},\qquad
\mathfrak d_H(G_0,G_0^{(L)})\le n^{-10},
\end{align*}
after increasing $C$; in the compactly supported case we simply take
$G_0^{(L)}=G_0$.  Hence, by the triangle inequality,
\begin{align*}
\mathfrak d_H(\hat G_n,G_0^{(L)})\le C\Delta_{n,\delta}+n^{-10}
\le C\Delta_{n,\delta},
\end{align*}
after increasing constants.  Both $\hat G_n$ and $G_0^{(L)}$ are supported in
$[-CL_{n,\alpha},CL_{n,\alpha}]$.  Applying Lemma
\ref{lem:Wasserstein_from_Hellinger} gives, for every $\eta\in(0,1)$,
\begin{align*}
\mathsf W_p(\hat G_n,G_0^{(L)})
\le
C\big\{\eta+
e^{C\eta^{-2}}L_{n,\alpha}^{C}\Delta_{n,\delta}^{1/C}\big\}.
\end{align*}
Adding the truncation error $\mathsf W_p(G_0,G_0^{(L)})\le n^{-10}$ and taking $\eta = A\log^{-1/2}(1/\Delta_{n,\delta})$ with $A$ sufficiently large makes the second term in the preceding display no larger than the first, and yields $\mathsf W_p(\hat G_n,G_0)\le C\log^{-1/2}(1/\Delta_{n,\delta})$, which proves \eqref{ineq:MML_Wasserstein_rate}.
\end{proof}

\section{Remaining proofs for Section \ref{section:main_results}}

\subsection{Proof of Proposition \ref{prop:minimax_lower_effective_sample_size}}\label{section:proof_minimax_lower_effective_sample_size}

	We first specify a special choice of $\Sigma_0$ with $\pnorm{\texttt{Cor}_{\Sigma_0}}{\op}=\kappa_0$. Let $n=Nb$ for integers $N,b\ge 1$, and let
	\begin{align*}
	R_{b,\rho}\equiv (1-\rho)I_b+\rho \mathbf 1_b\mathbf 1_b^\top,
	\qquad 0\le \rho<1.
	\end{align*}
	Set $\Sigma_0\equiv I_N\otimes R_{b,\rho}$. Then $\sigma_{0,j}^2=(\Sigma_0)_{jj}=1$ for all $j$, and
	\begin{align*}
	\kappa_0
	=
	\pnorm{\texttt{Cor}_{\Sigma_0}}{\op}
	= \pnorm{R_{b,\rho}}{\op}=
	1+(b-1)\rho,
	\qquad
	n_\ast={n}/{\kappa_0}.
	\end{align*}
	We shall compute the minimax lower bound based on this choice of $\Sigma_0$ by using a standard two-point testing argument. 
	
	For $t\in\mathbb R$, let $
	G_t=\delta_t$. 
	For $|t|\le M$, we have $G_t\in\mathscr G(M,\infty)$. Under $G_t$, the latent vector is deterministic, $\beta_{0,j}=t$ for all $j$, and the observation vector has law $
	U\sim \mathcal{N}(t\mathbf 1_n,\Sigma_0)$. 
	We write $P_t$ for this law. 
	
	Note that the covariance matrix $\Sigma_0$ is block diagonal with $N$ identical blocks $R_{b,\rho}$. Since $\mathbf 1_b$ is an eigenvector of $R_{b,\rho}$ with eigenvalue $
	\kappa_0=1+(b-1)\rho$,
	we have $
	\mathbf 1_b^\top R_{b,\rho}^{-1}\mathbf 1_b
	=
	{b}/{\kappa_0}$. 
	Therefore, for any $s,t\in\mathbb R$,
	\begin{align}\label{ineq:minimax_lower_effective_sample_size_1}
	\mathrm{KL}(P_t,P_s)
	&=
	\frac12 (t-s)^2 \mathbf 1_n^\top\Sigma_0^{-1}\mathbf 1_n= \frac12 (t-s)^2\, N\mathbf 1_b^\top R_{b,\rho}^{-1}\mathbf 1_b \nonumber\\
	& = \frac12 (t-s)^2\,\frac{n}{\kappa_0}
	= \frac12 (t-s)^2 n_\ast.
	\end{align}
	Now take $
	t_0=0, t_1=a n_\ast^{-1/2}$, 
	where $a>0$ is a sufficiently small constant. For $n_\ast$ large enough,
	$|t_1|\le M$, so both $G_{t_0}$ and $G_{t_1}$ belong to $\mathscr G(M,\infty)$.
	Moreover, using the above display (\ref{ineq:minimax_lower_effective_sample_size_1}), $
	\mathrm{KL}(P_{t_1},P_{t_0})={a^2}/{2}$. 
	Choosing $a>0$ small enough gives, by Pinsker's inequality (cf. \cite[Lemma 2.5]{tsybakov2008introduction}), 
	\begin{align}\label{ineq:minimax_lower_effective_sample_size_2}
	d_{\TV}(P_{t_0},P_{t_1})
	\le\sqrt{\mathrm{KL}(P_{t_1},P_{t_0})/2}={a}/{2} \le 1/2.
	\end{align}
	Next we compute the separation in the target loss. Since all marginal standard errors
	are equal to one, $
	\mathfrak d_{H;\{1\}}(G_{t_0},G_{t_1})
	=
	d_H\big(\varphi_{G_{t_0};1},\varphi_{G_{t_1};1}\big)$. 
	For two unit-variance Gaussian densities with means $t_0,t_1$, an easy calculation shows that $
	d_H^2\big(\varphi_{G_{t_0};1},\varphi_{G_{t_1};1}\big)
	=
	1-e^{-(t_1-t_0)^2/8}$. 
	For $|t_1-t_0|\le 1$, this implies
	\begin{align}\label{ineq:minimax_lower_effective_sample_size_3}
	\mathfrak d_{H;1}(G_{t_0},G_{t_1})=d_H\big(\varphi_{G_{t_0};1},\varphi_{G_{t_1};1}\big)
	\ge c_0 |t_1-t_0| =c_0 a n_\ast^{-1/2}
	\end{align}
	for a universal constant $c_0>0$.
	
	The claimed minimax lower bound now follows from Lemma \ref{lem:two_point_prob_metric} and the estimates in (\ref{ineq:minimax_lower_effective_sample_size_2}) and (\ref{ineq:minimax_lower_effective_sample_size_3}). \qed

\subsection{Proof of Proposition \ref{prop:credible_intervals_clean}}\label{section:proof_eb_credible_interval}

All probabilities below are taken under the original joint model $\beta_{0,j}\stackrel{\mathrm{i.i.d.}}{\sim}G_0,
U=\beta_0+\Sigma_0^{1/2}\mathsf{Z}_n$, where $\mathsf{Z}_n\sim\mathcal{N}(0,I_n)$.
The Gaussian errors may be dependent across coordinates.  We use only the marginal representation
\begin{align}\label{eq:ci_marginal_model}
U_j=\beta_{0,j}+\sigma_{0,j}\mathsf{Z}_j,
\qquad \mathsf{Z}_j\sim\mathcal{N}(0,1),
\qquad \beta_{0,j}\sim G_0,
\end{align}
for posterior calibration. Let 
$S_M\equiv [M^{-1/2},M^{1/2}]$ and $T_\ast\equiv \log(en_\ast)$.

\noindent (\textbf{Step 1}). We first record a consequence of Corollary \ref{cor:MML_exact} and Lemma \ref{lem:Wasserstein_from_Hellinger}.  By Corollary \ref{cor:MML_exact}, for a fixed sufficiently large numerical value of the probability parameter in that corollary, there is an event $E_{H,n}$ such that
\begin{align}\label{eq:ci_hellinger_good_nstar}
\Prob(E_{H,n}^c)\le C n^{-2},
\qquad
\mathfrak d_{H;\sigma_{0,[n]}}(\hat G_n,G_0)
\le C n_\ast^{-1/2}(\log n)^C
\end{align}
on $E_{H,n}$.  On the same event, $
\mathrm{supp}(\hat G_n)\subset [\min_j U_j,\max_j U_j]$. 
A Gaussian tail bound and the compact support of $G_0$ imply
\begin{align}\label{eq:ci_support_event}
\Prob\Big(\max_{j\in[n]} |U_j|>C\sqrt{\log n}\Big)\le Cn^{-2}.
\end{align}
Let $E_n$ be the intersection of $E_{H,n}$ with the event in \eqref{eq:ci_support_event}.  We claim that, after increasing $C$, on $E_n$ where $\Prob(E_n^c)\le C n_\ast^{-2}$,
\begin{align}\label{eq:ci_w1_good_nstar}
\mathsf W_1(\hat G_n,G_0)
\le w_\ast\equiv C T_\ast^{-1/2}.
\end{align}
Indeed, the probability bound follows from $n_\ast\le n$.  To prove the Wasserstein bound, we shall apply Lemma \ref{lem:Wasserstein_from_Hellinger} with support radius $L=C\sqrt{\log n}$.  As $n_\ast\ge(\log n)^{C_0}$ for a sufficiently large constant $C_0$, the Hellinger bound in \eqref{eq:ci_hellinger_good_nstar} is at most $C n_\ast^{-a}$ for some constant $a>0$. Choosing $\delta=A_0T_\ast^{-1/2}$ with $A_0$ sufficiently large in Lemma \ref{lem:Wasserstein_from_Hellinger}, the exponential factor $\exp(C\delta^{-2})$ therein is bounded by $n_\ast^{\varepsilon}$ with $\varepsilon>0$ small enough to be absorbed by the power $n_\ast^{-a}$ from the Hellinger term.  Hence the second term in Lemma \ref{lem:Wasserstein_from_Hellinger} is no larger than $C\delta$, and \eqref{eq:ci_w1_good_nstar} follows.

\noindent (\textbf{Step 2}). We next convert the Wasserstein control (\ref{eq:ci_w1_good_nstar}) into a deterministic posterior quantile perturbation bound.  If $G$ is any distribution with $\mathsf W_1(G,G_0)\le w_\ast$, then the bounded density assumption on $G_0$ implies
\begin{align}\label{eq:ci_K_from_W1_nstar}
d_{\mathrm{Kol}}(G,G_0)
\equiv \sup_{t\in\R}|G(( -\infty,t])-G_0(( -\infty,t])|
\le C w_\ast^{1/2}.
\end{align}
To see this, couple $X\sim G$ and $Y\sim G_0$ so that $\E|X-Y|\le 2\mathsf W_1(G,G_0)$.  For every $\varepsilon>0$ and $t\in\R$,
\begin{align*}
G(( -\infty,t])
&\le G_0(( -\infty,t+\varepsilon])
+\Prob(|X-Y|>\varepsilon)\\
&\le G_0(( -\infty,t])+C\varepsilon+2\varepsilon^{-1}\mathsf W_1(G,G_0).
\end{align*}
The reverse inequality is identical.  Optimizing over $\varepsilon$ gives \eqref{eq:ci_K_from_W1_nstar}.

Fix $s\in S_M$ and $u\in\R$. We shall work on the central region $|u|\le B_\ast$, where
\begin{align}\label{eq:ci_Bstar_def}
B_\ast\equiv B_0+b_0\sqrt{\log T_\ast},
\end{align}
with $B_0$ large enough and $b_0>0$ small enough, both depending only on the fixed model constants.  Since $G_0$ is supported on $[-M_0,M_0]$ and $s\in S_M$,
\begin{align}\label{eq:ci_p0_lower_upper_nstar}
c e^{-CB_\ast^2}\le \varphi_{G_0;s}(u)\le C,
\qquad |u|\le B_\ast .
\end{align}
Moreover, $\theta\mapsto\varphi_s(u-\theta)$ has a uniformly bounded Lipschitz constant for $s\in S_M$, and therefore
\begin{align}\label{eq:ci_denom_diff_nstar}
|\varphi_{G;s}(u)-\varphi_{G_0;s}(u)|\leq \biggabs{\int \varphi_s(u-\theta)(G-G_0)\,\d{\theta}}
\le C\cdot \mathsf W_1(G,G_0)
\le Cw_\ast .
\end{align}
For the numerator of the posterior c.d.f, set
\begin{align*}
h_{t,u,s}(\theta)\equiv \bm 1_{\{\theta\le t\}}\varphi_s(u-\theta).
\end{align*}
The total variation norm of $h_{t,u,s}$ as a function of $\theta$ is bounded by a constant depending only on $M$, uniformly over $(t,u,s)$.  Integration by parts against the signed measure $G-G_0$ gives
\begin{align}\label{eq:ci_num_diff_nstar}
\biggabs{
	\int h_{t,u,s}(\theta)\,(G-G_0)(\d\theta) }
\le C d_{\mathrm{Kol}}(G,G_0)
\le Cw_\ast^{1/2}.
\end{align}
Write $A_{G;s}(t,u)\equiv \int_{-\infty}^t \varphi_s(u-\theta)\,G(\d \theta)\leq \varphi_{G;s}(u)$. Then
\begin{align*}
&\bigabs{\Pi_{G,s}(t\mid u)-\Pi_{G_0,s}(t\mid u)} = \biggabs{ \frac{ A_{G;s}(t,u) }{ \varphi_{G;s}(u) } - \frac{A_{G_0;s}(t,u)}{ \varphi_{G_0;s}(u) }  }\\
&\leq \frac{1}{\varphi_{G;s}(u)}\cdot \abs{ A_{G;s}(t,u)-A_{G_0;s}(t,u) }+ \frac{A_{G_0;s}(t,u)}{\varphi_{G_0;s}(u)}\cdot \frac{ \abs{  \varphi_{G;s}(u)  -  \varphi_{G_0;s}(u) } }{ \varphi_{G;s}(u) }\\
& \leq \frac{1}{\varphi_{G;s}(u)}\cdot\bigg(\biggabs{
	\int h_{t,u,s}(\theta)\,(G-G_0)(\d\theta) }+ \abs{  \varphi_{G;s}(u)  -  \varphi_{G_0;s}(u) } \bigg).
\end{align*}
Combining \eqref{eq:ci_p0_lower_upper_nstar}-\eqref{eq:ci_num_diff_nstar}, if $Cw_\ast\leq c e^{-CB_\ast^2}/2$, we have $\varphi_{G;s}(u)\geq c e^{-CB_\ast^2}/2$, and therefore  for $|u|\le B_\ast$,
\begin{align}\label{eq:ci_posterior_cdf_stability_nstar}
&\sup_{t\in\R}
|\Pi_{G,s}(t\mid u)-\Pi_{G_0,s}(t\mid u)| \le C e^{CB_\ast^2}w_\ast^{1/2}
\equiv \Delta_\ast.
\end{align}
For small $n_\ast$ the desired bound \eqref{eq:ci_nstar_final_bound} is trivial after increasing $C$, so we may assume $\Delta_\ast\le 1/8$.

Now we shall invert (\ref{eq:ci_posterior_cdf_stability_nstar}) to obtain an estimate for the quantiles. Let $q_{a,0}(u,s)=Q_{G_0,s}(a\mid u)$ for $a\in\{\gamma/2,1-\gamma/2\}$.  The oracle posterior density is
\begin{align}\label{eq:ci_oracle_post_density_nstar}
\pi_{0,s}(t\mid u)
=\frac{\varphi_s(u-t)g_0(t)}{\varphi_{G_0;s}(u)},
\qquad t\in[-M_0,M_0].
\end{align}
By the bounded density assumption on $g_0$ and \eqref{eq:ci_p0_lower_upper_nstar}, this density is bounded below near every oracle quantile:
\begin{align}\label{eq:ci_post_density_lower_nstar}
\pi_{0,s}(t\mid u)
\ge c e^{-CB_\ast^2},
\qquad t\in[-M_0,M_0],\ |u|\le B_\ast .
\end{align}
Therefore the usual quantile inversion argument (cf. Lemma \ref{lem:quantile_inversion}) applied to \eqref{eq:ci_posterior_cdf_stability_nstar} gives, for $a\in\{\gamma/2,1-\gamma/2\},\ |u|\le B_\ast$, if $w_\ast\leq  e^{-CB_\ast^2}$,
\begin{align}\label{eq:ci_quantile_stability_nstar}
|Q_{G,s}(a\mid u)-Q_{G_0,s}(a\mid u)|
\le r_\ast\equiv C e^{CB_\ast^2}w_\ast^{1/2}.
\end{align}

\noindent (\textbf{Step 3}).  Consider the oracle marginal credible interval
\begin{align}\label{eq:ci_oracle_interval}
C_j^0(U_j;\gamma)
\equiv
\big[
Q_{G_0,\sigma_{0,j}}(\gamma/2\mid U_j),
Q_{G_0,\sigma_{0,j}}(1-\gamma/2\mid U_j)
\big].
\end{align}
On $E_n$, if $|U_j|\le B_\ast$, then \eqref{eq:ci_quantile_stability_nstar} with $G=\hat G_n$, $s=\sigma_{0,j}$, and $u=U_j$ implies that the two endpoints of $\widehat C_j(U;\gamma)$ are within $r_\ast$ of the corresponding endpoints of $C_j^0(U_j;\gamma)$.  Hence, deterministically on $E_n$,
\begin{align}\label{eq:ci_indicator_compare_nstar}
&\bigabs{
	\bm 1_{\{\beta_{0,j}\in \widehat C_j(U;\gamma)\}}
	-
	\bm 1_{\{\beta_{0,j}\in C_j^0(U_j;\gamma)\}} }\nonumber\\
&\quad\le
\bm 1_{\{|U_j|>B_\ast\}}
+
\sum_{a\in\{\gamma/2,1-\gamma/2\}}
\bm 1_{\{|U_j|\le B_\ast,\ |
	\beta_{0,j}-Q_{G_0,\sigma_{0,j}}(a\mid U_j)|\le r_\ast\}} .
\end{align}
This deterministic implication is valid although $\hat G_n$ was learned from the same data vector $U$.

The oracle interval has exact marginal Bayes coverage.  By \eqref{eq:ci_marginal_model} and the continuity of the posterior distribution under the bounded density assumption on $g_0$,
\begin{align}\label{eq:ci_oracle_exact_nstar}
\Prob(\beta_{0,j}\in C_j^0(U_j;\gamma)\mid U_j)=1-\gamma,
\qquad
\Prob(\beta_{0,j}\in C_j^0(U_j;\gamma))=1-\gamma .
\end{align}
Furthermore, compact support of $G_0$ and $s\in S_M$ imply
\begin{align}\label{eq:ci_U_tail_nstar}
\Prob(|U_j|>B_\ast)
\le C e^{-cB_\ast^2}.
\end{align}
For the boundary terms in \eqref{eq:ci_indicator_compare_nstar}, the oracle posterior density upper bound
\begin{align}\label{eq:ci_post_density_upper_nstar}
\sup_{s\in S_M, |u|\le B_\ast, t\in\R}\pi_{0,s}(t\mid u)
\le C e^{CB_\ast^2}
\end{align}
gives, for $a\in\{\gamma/2,1-\gamma/2\}$,
\begin{align}\label{eq:ci_boundary_prob_nstar}
&\Prob\Big(|U_j|\le B_\ast,
|\beta_{0,j}-Q_{G_0,\sigma_{0,j}}(a\mid U_j)|\le r_\ast\Big)\nonumber\\
&\quad =
\E\Big[
\bm 1_{\{|U_j|\le B_\ast\}}
\Pi_{G_0,\sigma_{0,j}}
\big([Q_{G_0,\sigma_{0,j}}(a\mid U_j)-r_\ast,
Q_{G_0,\sigma_{0,j}}(a\mid U_j)+r_\ast]\mid U_j\big)
\Big]\nonumber\\
&\quad\le C e^{CB_\ast^2}r_\ast
\le C e^{CB_\ast^2}w_\ast^{1/2}.
\end{align}
Taking expectations in \eqref{eq:ci_indicator_compare_nstar}, adding the probability of $E_n^c$, using \eqref{eq:ci_oracle_exact_nstar}, \eqref{eq:ci_U_tail_nstar}, and \eqref{eq:ci_boundary_prob_nstar}, and then averaging over $j\in[n]$, yields
\begin{align}\label{eq:ci_pre_optimized_bound_nstar}
\abs{
	\Prob(\beta_{0,\pi_n}\in \widehat C_{\pi_n}(U;\gamma))-(1-\gamma)}
\le C n_\ast^{-2}+C e^{-cB_\ast^2}+C e^{CB_\ast^2}w_\ast^{1/2}.
\end{align}
Finally substitute $w_\ast=C T_\ast^{-1/2}$ and $B_\ast=B_0+b_0\sqrt{\log T_\ast}$.  Taking $b_0>0$ sufficiently small gives constants $C,c>0$ such that $
e^{-cB_\ast^2}+e^{CB_\ast^2}w_\ast^{1/2}
\le C T_\ast^{-c}$. 
Since $n_\ast^{-2}\le C T_\ast^{-c}$, \eqref{eq:ci_nstar_final_bound} follows. \qed

\subsection{Proof of Proposition \ref{prop:marginal_regret}}\label{section:proof_marginal_regret}

We first record a simple observation.
\begin{lemma}\label{lem:regret_alt_rep}
	$\texttt{Reg}_n(G,G_0)$ defined in (\ref{eq:regret_identity_text}) can be rewritten as
	\begin{align}\label{def:marginal_regret}
	\texttt{Reg}_n(G,G_0)
	\equiv
	\frac1n\sum_{j \in [n]}
	\E_{X\sim\varphi_{G_0;\sigma_{0,j}}}
	\big(m_{G,\sigma_{0,j}}(X)-m_{G_0,\sigma_{0,j}}(X)\big)^2.
	\end{align}
\end{lemma}
\begin{proof}
	As $m_{G_0,\sigma_{0,j}}(X_j^\circ)=\E(\theta_j^\circ\mid X_j^\circ)$, from (\ref{eq:regret_identity_text}) we have 
	\begin{align*}
	\texttt{Reg}_n(G,G_0) &= \frac{1}{n}\sum_{j \in [n]} \Big\{ \E
	\big(m_{G,\sigma_{0,j}}(X_j^\circ)-\theta_j^\circ\big)^2-\E\big(\E(\theta_j^\circ\mid X_j^\circ)-\theta_j^\circ\big)^2 \Big\}\\
	&= \frac1n\sum_{j \in [n]}
	\E
	\big(m_{G,\sigma_{0,j}}(X_j^\circ)-m_{G_0,\sigma_{0,j}}(X_j^\circ)\big)^2 = \hbox{RHS of (\ref{def:marginal_regret})},
	\end{align*}
	completing the proof.
\end{proof}

We now state a regret-Hellinger transfer inequality used in the proof. The proof makes use a recent result of \cite{chen2026sharp}.

\begin{lemma}\label{lem:regret_hellinger_growing_support}
	Fix $M>1$.  There exists a constant $C=C(M)>0$ such that the following
	holds.  Let $L\ge M$, $G\in\mathscr G(L,\infty)$, $H\in\mathscr G(M,\infty)$, and
	$s\in[M^{-1/2},M^{1/2}]$.  Define $
	h_s(G,H)
	\equiv d_H(\varphi_{G;s},\varphi_{H;s})$.
	Then
	\begin{align}
	\label{ineq:regret_hellinger_growing_support_pointwise}
	&\E_{X\sim\varphi_{H;s}}
	(m_{G,s}(X)-m_{H,s}(X))^2 \le
	C(1+L)^4\cdot  h_s^2(G,H)
	\log^C\!\bigg({e\over h_s(G,H)}\bigg),
	\end{align}
	with the convention that the right-hand side is zero when $h_s(G,H)=0$.
	
	Consequently, if $s_1,\ldots,s_n\in[M^{-1/2},M^{1/2}]$, then
	\begin{align}
	\label{ineq:regret_hellinger_growing_support_average}
	&{1\over n}\sum_{j \in [n]}
	\E_{X\sim\varphi_{H;s_j}}
	\{m_{G,s_j}(X)-m_{H,s_j}(X)\}^2 \le
	C(1+L)^4 \cdot \bar h^2
	\log^C (e/\bar{h}),
	\end{align}
	where $
	\bar h^2 \equiv
	n^{-1} \sum_{j \in [n]}
	h_{s_j}(G,H)$.
\end{lemma}

\begin{proof}
	\noindent (\textbf{Step 1}). We first prove the result for $s=1$.  Write
	\begin{align*}
	f_G=\varphi_{G;1},
	\qquad
	f_H=\varphi_{H;1},
	\qquad
	h=d_H(f_G,f_H).
	\end{align*}
	The case $h=0$ is trivial and we therefore assume $h>0$. By \cite[Theorem 5]{chen2026sharp},  if $G,H$ have second moment bounded by $\sigma^2$, then the left hand side of (\ref{ineq:regret_hellinger_growing_support_pointwise}) is bounded by $C e^{\sigma^2/2}(1+L)^4 h^2
	\log^C (e/h)$, possibly after increasing the universal constant $C$.  Thus it remains to verify that the second moment of $G$ is bounded by a
	constant depending only on $M$ in the nontrivial regime.
	
	First we note that if $
	(1+L)^2h^2>c_0$ for a sufficiently small constant $c_0=c_0(M)>0$, then the desired bound is
	trivial.  Indeed, since $m_{G,1}(X)\in[-L,L]$ and
	$m_{H,1}(X)\in[-M,M]$, we necessarily have $
	\E_{X\sim f_H}(m_{G,1}(X)-m_{H,1}(X))^2
	\le C(1+L)^2$, 
	whereas the right-hand side of
	\eqref{ineq:regret_hellinger_growing_support_pointwise} is at least a
	constant multiple of $(1+L)^2$ after increasing $C$.  
	
	Hence we only need to consider the case $
	(1+L)^2h^2\le c_0$, and prove the second-moment bound for $G$ under this condition.  Let $\Theta\sim G$ and
	$X=\Theta+\mathsf{Z}$, where $\mathsf{Z}\sim \mathcal{N}(0,1)$ is independent.  For $t\ge2$,
	on the event $\{|\Theta|>t,\ |\mathsf{Z}|\le t/2\}$ we have $|X|>t/2$.
	Since $\Prob(|\mathsf{Z}|\le t/2)$ is bounded below by an absolute positive constant
	for $t\ge2$, it follows that
	\begin{align*}
	G(|\theta|>t)
	&\le
	C\,\Prob_{X\sim f_G}(|X|>t/2)\\
	&\stackrel{(\ast)}{\leq} C\cdot \big(\Prob_{X\sim f_H}(|X|>t/2)+h^2\big)\leq C_M e^{-c_Mt^2}+C h^2.
	\end{align*}
	Here in $(\ast)$ we used that for any event $A$, $
	\Prob_{f_G}(A) = \int_A f_G = \int_A \big(f_H^{1/2}+\{f_G^{1/2}-f_H^{1/2}\}\big)^2
	\leq
	2\Prob_{f_H}(A)+2h^2$. Since $G$ is supported on $[-L,L]$,
	\begin{align*}
	\int \theta^2\,G(\d\theta)
	&=
	\int_0^\infty 2t\,G(|\theta|>t)\,\d t\\
	&\le
	C+
	C_M\int_2^L t\{e^{-c_Mt^2}+h^2\}\,\d t \le
	C_M+C_M h^2L^2
	\le
	C_M,
	\end{align*}
	where the last inequality follows from the assumption $
	(1+L)^2h^2\le c_0$.
	The prior $H$ has second moment bounded by $M^2$.  Thus the second moments
	of both $G$ and $H$ are bounded by a constant depending only on $M$.
	
	Now applying \cite[Theorem 5]{chen2026sharp} with this second-moment bound gives
	\begin{align*}
	\E_{X\sim f_H}
	\big(m_{G,1}(X)-m_{H,1}(X)\big)^2
	\le
	C_M(1+L)^4 \cdot h^2
	\log^C\!(e/h),
	\end{align*}
	which proves the claim for $s=1$.
	
	\noindent (\textbf{Step 2}). For general $s\in[M^{-1/2},M^{1/2}]$, reduce to the unit-variance case by
	scaling. Let $
	G^{(s)}=G(s\cdot), H^{(s)}=H(s\cdot)$
	Then $\varphi_{G;s}(x)
	= s^{-1}\varphi_{G^{(s)};1}(x/s)$,
	$\varphi_{H;s}(x) = s^{-1}\varphi_{H^{(s)};1}(x/s)$,
	and therefore $
	d_H(\varphi_{G;s},\varphi_{H;s})
	=
	d_H(\varphi_{G^{(s)};1},\varphi_{H^{(s)};1})$. 
	Moreover,
	\begin{align*}
	m_{G,s}(x)=s\,m_{G^{(s)},1}(x/s),
	\qquad
	m_{H,s}(x)=s\,m_{H^{(s)},1}(x/s).
	\end{align*}
	Thus, if $X\sim\varphi_{H;s}$ and $Y=X/s$, then
	$Y\sim\varphi_{H^{(s)};1}$ and
	\begin{align*}
	\begin{aligned}
	\E_{X\sim\varphi_{H;s}}
	\big(m_{G,s}(X)-m_{H,s}(X)\big)^2 =
	s^2
	\E_{Y\sim\varphi_{H^{(s)};1}}
	\big(m_{G^{(s)},1}(Y)-m_{H^{(s)},1}(Y)\big)^2 .
	\end{aligned}
	\end{align*}
	Since $s\in[M^{-1/2},M^{1/2}]$, the rescaled priors satisfy $
	G^{(s)}\in\mathscr G(C_ML,\infty),
	H^{(s)}\in\mathscr G(C_M,\infty)$.
	Applying the already proved unit-variance bound to
	$(G^{(s)},H^{(s)})$, multiplying by $s^2\le M$, and using the Hellinger
	scaling identity gives
	\eqref{ineq:regret_hellinger_growing_support_pointwise}.
	
	\noindent (\textbf{Step 3}). Finally, applying the pointwise bound with $s=s_j$ gives
	\begin{align*}
	{1\over n}\sum_{j \in [n]}
	\E_{X\sim\varphi_{H;s_j}}
	\big(m_{G,s_j}(X)-m_{H,s_j}(X)\big)^2
	\le
	C(1+L)^4\cdot 
	{1\over n}\sum_{j \in [n]}
	h_j^2\log^C(e/h_j),
	\end{align*}
	where $h_j=d_H(\varphi_{G;s_j},\varphi_{H;s_j})$.  Since
	$t\mapsto t\log^C(e/\sqrt t)$ is bounded above on $[0,1]$ by a concave
	multiple of itself after increasing $C$, Jensen's inequality yields $
	n^{-1}\sum_{j \in [n]}
	h_j^2\log^C(e/h_j)
	\le
	C\bar h^2\log^C(e/\bar h)$. 
	This proves \eqref{ineq:regret_hellinger_growing_support_average}.
\end{proof}

\begin{proof}[Proof of Proposition \ref{prop:marginal_regret}]
	Let $L_n\equiv A\sqrt{\log n}$ for a sufficiently large constant $A=A(M,D)$.  Since $G_0\in\mathscr G(M,\infty)$ and $M^{-1}\le \sigma_{0,j}^2\le M$, a union bound and the Gaussian tail inequality give
	\begin{align}\label{eq:regret_support_event}
	\Prob\Big(\max_{j\in[n]} |U_j|>L_n\Big)\le n^{-D-1}
	\end{align}
	after increasing $A$.  On the event in \eqref{eq:regret_support_event}, the CML estimator satisfies $\hat G_n\in\mathscr G(L_n,\infty)$. By Corollary \ref{cor:MML_exact}, with probability at least $1-c n^{-D-1}$,
	\begin{align}\label{eq:regret_hellinger_event}
	\bar h_n^2
	\equiv
	\frac1n\sum_{j \in [n]}
	d_H^2(\varphi_{\hat G_n;\sigma_{0,j}},\varphi_{G_0;\sigma_{0,j}})
	=
	\mathfrak d_{H;\sigma_{0,[n]}}^2(\hat G_n,G_0)
	\le
	C n_\ast^{-1}(\log n)^C.
	\end{align}
	Let $E_n$ be the intersection of the events in \eqref{eq:regret_support_event} and \eqref{eq:regret_hellinger_event}.  Then $\Prob(E_n^c)\le Cn^{-D}$.  On $E_n$, Lemma \ref{lem:regret_hellinger_growing_support} with $G=\hat G_n$, $H=G_0$, $L=L_n$, and $s_j=\sigma_{0,j}$ yields
	\begin{align*}
	\texttt{Reg}_n(\hat G_n,G_0)
	&\le
	C(1+L_n)^C\cdot \bar h_n^2\{\log(e/\bar h_n)\}^C \le C\cdot  n_\ast^{-1}(\log n)^C.
	\end{align*}
	Here we used $L_n=A\sqrt{\log n}$ and the assumption $n_\ast\ge (\log n)^C$; the exponent $C$ may change from line to line.  This proves the claimed estimate.
\end{proof}

\section{Proof of Theorem \ref{thm:gls_eb}}\label{section:proof_gls_eb}

\begin{lemma}\label{lem:log_normal_mix_der}
	For every \(x\in\mathbb R\) and \(\sigma>0\),
	\begin{align*}
	\partial_\sigma \log \varphi_{G;\sigma}(x)
	=
	\frac1\sigma
	\left\{
	\E\left[
	\frac{(X-\theta)^2}{\sigma^2}
	\,\bigg|\, X=x
	\right]
	-1
	\right\},
	\end{align*}
	where \(X=\theta+\sigma \mathsf{Z}\), \(\mathsf{Z}\sim N(0,1)\), \(\theta\sim G\), and
	\(\mathsf{Z}\perp \theta\). Consequently, if \(|x|\le B\) and \(\mathrm{supp}(G)\subset [-R,R]\), we have $
	\abs{
	\partial_\sigma \log \varphi_{G;\sigma}(x)}
	\leq
	\sigma^{-1}
	\{
	1+(B+R)^2/\sigma^2
	\}$.
\end{lemma}

\begin{proof}
	\noindent (1). Applying $\partial_\sigma \varphi_\sigma(t)
	=
	\sigma^{-1}
	\big(t^2/\sigma^2-1
	\big)\varphi_\sigma(t)$ with \(t=x-\theta\), we get
	\begin{align*}
	\partial_\sigma \varphi_{G;\sigma}(x)
	=
	\int
	\partial_\sigma \varphi_\sigma(x-\theta)\,G(\d\theta) =
	\frac1\sigma
	\int
	\bigg(
	\frac{(x-\theta)^2}{\sigma^2}-1
	\bigg)
	\varphi_\sigma(x-\theta)\,G(\d\theta).
	\end{align*}
	Since \(\varphi_{G;\sigma}(x)>0\), dividing by \(\varphi_{G;\sigma}(x)\) gives
	\begin{align*}
	\partial_\sigma \log \varphi_{G;\sigma}(x)
	=
	\frac{\partial_\sigma \varphi_{G;\sigma}(x)}
	{\varphi_{G;\sigma}(x)} =
	\frac1\sigma
	\bigg\{
	\frac{
		\int \frac{(x-\theta)^2}{\sigma^2}
		\varphi_\sigma(x-\theta)\,G(\d\theta)
	}{
		\int \varphi_\sigma(x-\theta)\,G(\d\theta)
	}
	-1
	\bigg\}.
	\end{align*}
	It remains only to identify the ratio as a conditional expectation. Under the
	model \(X=\theta+\sigma \mathsf{Z}\), the joint law of \((\theta,X)\) has conditional
	density $
	\theta\mapsto \varphi_\sigma(x-\theta)\,G(\d\theta)$
	given \(X=x\). Hence the posterior distribution of \(\theta\) given \(X=x\) is $
	\Pi_{\sigma,x}(\d\theta)
	=
	{\varphi_\sigma(x-\theta)\,G(\d\theta)}\big/
	{\int \varphi_\sigma(x-u)\,G(\d u)}$.
	This means
	\begin{align*}
	\frac{
		\int \frac{(x-\theta)^2}{\sigma^2}
		\varphi_\sigma(x-\theta)\,G(\d\theta)
	}{
		\int \varphi_\sigma(x-\theta)\,G(\d\theta)
	}
	=
	\E\bigg[
	\frac{(X-\theta)^2}{\sigma^2}
	\,\bigg|\, X=x
	\bigg],
	\end{align*}
	where on the right-hand side \(X=x\) is fixed inside the conditional expectation.
	This proves the identity.
	
	\noindent(2).	Finally, if \(|x|\le B\) and \(|\theta|\le R\), then $
	|x-\theta|\le B+R$.
	Thus $
	0\le
	\E[
	{(X-\theta)^2}/{\sigma^2}
	\,|\, X=x
	]
	\le
	{(B+R)^2}/{\sigma^2}$.
	The claimed estimate follows.
\end{proof}

\begin{proof}[Proof of Theorem \ref{thm:gls_eb}]
Conditional on $A$,
\begin{align*}
\hat\mu_{\mathrm{gls}}-\mu_\ast=Q_A^{-1}A^\top\Omega^{-1}\xi
\sim \mathcal{N}(0,\tau_\ast^2Q_A^{-1}),
\end{align*}
so the GLS statistic is an exact correlated Gaussian sequence with
$\Sigma_0=\tau_\ast^2Q_A^{-1}$ and marginal standard errors $\tau_\ast s_{A,j}$.
The effective sample size in Theorem \ref{thm:MML_rate} is therefore precisely
$n_{\ast,A}=n/\pnorm{\texttt{Cor}_{Q_A^{-1}}}{\op}$.  The exponential bound in
\eqref{ineq:p0_tail_exp} gives
$\pnorm{\hat\mu_{\mathrm{gls}}}{\infty}\le C_{\alpha,M,D}L_{n,\alpha}$ with conditional
probability at least $1-Cn^{-D}$.  Applying Corollary \ref{cor:MML_exact} conditionally on $A$ proves the Hellinger rate.  The Wasserstein bound in the same display follows from the exact CML case of Theorem \ref{thm:MML_rate}, or equivalently from its Wasserstein bound with $\delta_n=0$.

It remains to compare the plug-in likelihood with the oracle likelihood when
$\tau_\ast$ is unknown.  On the event
$\gamma_A=|\hat\tau/\tau_\ast-1|\le1/2$, and on
$\max_j|\hat\mu_{\mathrm{gls},j}|\le B$ with
$B=C_{\alpha,M,D}L_{n,\alpha}$, Lemma \ref{lem:log_normal_mix_der} implies,
uniformly over $G\in\mathscr G(R,\infty)$ with $R\le B$ and over
$\sigma\in[\tau_\ast s_{A,j}/2,3\tau_\ast s_{A,j}/2]$,
\begin{align*}
\abs{\partial_\sigma\log\varphi_{G;\sigma}(x) }
\le C_{\alpha,M}\{1+(B+R)^2\}.
\end{align*}
Consequently, uniformly over such $G$,
\begin{align*}
\biggabs{\frac1n\sum_{j \in [n]}\log\varphi_{G;\hat\tau s_{A,j}}(\hat\mu_{\mathrm{gls},j})
-\frac1n\sum_{j \in [n]}\log\varphi_{G;\tau_\ast s_{A,j}}(\hat\mu_{\mathrm{gls},j})}
\le C_{\alpha,M,D}(\log n)^C\gamma_A.
\end{align*}
Applying this comparison once to the plug-in maximizer and once to the oracle
maximizer shows that $\hat G_{\mathrm{gls}}$ is a
$\delta_n$-near maximizer of the oracle criterion with
$\delta_n\le C_{\alpha,M,D}(\log n)^C\gamma_A$.  Theorem
\ref{thm:MML_rate}, applied conditionally on $A$, gives
the Hellinger rate after absorbing the power of $\log n$ into $c_1$.  The Wasserstein bound follows from the Wasserstein part of Theorem \ref{thm:MML_rate} with the same near-maximization error $\delta_n\le C_{\alpha,M,D}(\log n)^C\gamma_A$; the resulting quantity $n_{\ast,A}^{-1/2}(\log n)^C+\delta_n^{1/2}$ is bounded by $\Delta_A$ after increasing $c_1$.

Finally, for the probability estimate for $\gamma_A$, let $P_A\equiv AQ_A^{-1}A^\top\Omega^{-1}$.  Algebra gives, conditionally on $A$,
\begin{align*}
(m-n)\cdot \frac{\hat\tau^2}{\tau_\ast^2}
=\frac{\xi^\top\Omega^{-1}(I-P_A)\xi}{\tau_\ast^2}\sim \chi^2_{m-n}.
\end{align*}
The displayed probability bound follows from the standard
chi-square concentration inequalities.
\end{proof}

\section{Proof of Theorem \ref{thm:eb_gd_hellinger}}\label{section:proof_eb_gd_hellinger}

Throughout this section, $\Prob^{(0)}$ and $\E^{(0)}$ denote probability and expectation over the Gaussian design $A$ conditional on $(\mu_0,\mu_\ast,\xi)$, while $\Prob^\xi$ and $\E^\xi$ denote probability and expectation conditional only on $\xi$. For notational simplicity, we usually omit the subscript $\mathrm{db}$ in, e.g., $\hat{\mu}_{\mathrm{db}}, \hat{G}_{\mathrm{db}}$.

Recall $\bar{\mathfrak{U}}=(\bar{\mathfrak{U}}_{1},\bar{\mathfrak{U}}_{2})$ defined in Definition \ref{def:U_bar}. In addition to $\bar{\delta}$ defined therein, we also define
\begin{align*}
\bar{\sigma}\equiv \big\{\phi\cdot \E \mathfrak{S}^2\big(\bar{\mathfrak{U}}_{1}, \bar{\mathfrak{U}}_{2},\xi_{\pi_m}\big)\big\}^{1/2},\quad \bar{\tau} \equiv \phi\cdot \E \partial_1\mathfrak{S}\big(\bar{\mathfrak{U}}_{1}, \bar{\mathfrak{U}}_{2},\xi_{\pi_m}\big). 
\end{align*}

\subsection{Distribution of debiased gradient descent}

We shall define some notation that will be used in the proofs below.
\begin{definition}
	Fix $\mu_0 \in \R^n$.
	\begin{enumerate}
		\item Let $\mathfrak{U}_{\ast}=(\mathfrak{U}_{\ast,1},\mathfrak{U}_{\ast,2})$ be a centered, bi-variate Gaussian vector with covariance 
		$\frac{1}{n}
		\begin{psmallmatrix}
		\pnorm{\mu_0}{\Sigma}^2 & \iprod{\mu_0}{\mu_\ast}_\Sigma\\
		\iprod{\mu_0}{\mu_\ast}_\Sigma & \pnorm{\mu_\ast}{\Sigma}^2
		\end{psmallmatrix}$.
		\item Let
		\begin{itemize}
			\item $\sigma_\ast\equiv \big\{\phi\cdot \E^{(0)} \mathfrak{S}^2\big(\mathfrak{U}_{\ast,1}, \mathfrak{U}_{\ast,2},\xi_{\pi_m}\big)\big\}^{1/2}$, 
			\item  $
			\tau_\ast\equiv \phi\cdot \E^{(0)} \partial_1\mathfrak{S}\big(\mathfrak{U}_{\ast,1}, \mathfrak{U}_{\ast,2},\xi_{\pi_m}\big)$,
			\item $
			\delta_\ast\equiv -\phi\cdot \E^{(0)} \partial_2\mathfrak{S}\big(\mathfrak{U}_{\ast,1}, \mathfrak{U}_{\ast,2},\xi_{\pi_m}\big)$.
		\end{itemize}
		\item Let $L_\ast\equiv 1+\pnorm{\mu_{0}}{\infty}+ \pnorm{\mu_{\ast}}{\infty}$.
	\end{enumerate}
\end{definition}

We show in the following proposition  that $
\hat{\mu}\stackrel{d}{\approx} \mathcal{N}(\delta_\ast \mu_\ast, \sigma_\ast^2 \Sigma^{-1})$
in the empirical distributional sense.
\begin{proposition}\label{prop:db_gd_normal}
	Suppose Assumption \ref{assump:model} holds for some $K,\Lambda\geq 2$. Fix $\{\psi_j\}_{j \in [n]}$ with $\max_{j \in [n]}\pnorm{\psi_j}{\mathrm{Lip}}\leq \Lambda$. Then for any $D>0$, there exists some $c_1\equiv c_1(D)>1$, such that if $n\geq (K\Lambda L_\ast)^{c_1}$, it holds with $\Prob^{(0)}$-probability at least $1-\exp(-\log^D n)$ that, 
	\begin{align*}
	\biggabs{\frac{1}{n}\sum_{j \in [n]} \psi_j (\hat{\mu}_j)- \frac{1}{n} \sum_{j \in [n]} \psi_j (\mu_j^{\mathrm{db}})}\leq \big(K\Lambda L_\ast (1\wedge \sigma_\ast)^{-1}\log n\big)^{c_1} \cdot  n^{-1/2}.
	\end{align*}
	Here with  $\mathsf{Z}_n\sim \mathcal{N}(0,I_n)$  independent of all other variables, let
	\begin{align}\label{def:mu_db}
	\mu^{\mathrm{db}}\equiv \delta_\ast \mu_{\ast}+  \sigma_\ast \Sigma^{-1/2}\mathsf{Z}_n.
	\end{align}
\end{proposition}

We shall first prove the following.

\begin{lemma}\label{lem:gd_onestep}
Suppose Assumption \ref{assump:model} holds for some $K,\Lambda\geq 2$. For any vector $\mu_0 \in \R^n$, let
\begin{align}\label{def:gd_debias_oracle}
\bar{\mu} \equiv \tau_\ast \mu_0 - \Sigma^{-1} A^\top \mathsf{L}(A\mu_0,Y).
\end{align}
Then for any $D>0$, there exists some $c_1\equiv c_1(D)>1$ such that if $n\geq (K\Lambda L_\ast)^{c_1}$, for any $\{\psi_j\}_{j \in [n]}$ with $\max_{j \in [n]}\pnorm{\psi_j}{\mathrm{Lip}}\leq \Lambda$, it holds with $\Prob^{(0)}$-probability at least $1- \exp(-\log^D n)$ that
\begin{align*}
& \biggabs{ \frac{1}{n}\sum_{j \in [n]} \Big( \psi_j(\bar{\mu}_j)- \E^{(0)} \psi_j\big(\mu_j^{\mathrm{db}} \big)\Big)  }\leq \big(K\Lambda L_\ast(1\wedge \sigma_\ast)^{-1}\big)^{c_1}\cdot \bigg(\frac{\log n}{n}\bigg)^{1/2}.
\end{align*}
\end{lemma}

\begin{proof}
Without loss of generality we assume that $\psi_j(0)=0$. For notational convenience, we work with a common test function $\psi$; the proof for non-identical $\{\psi_j\}$ is almost identical. We write $w_0\equiv \Sigma^{1/2}\mu_0$ and $w_\ast\equiv \Sigma^{1/2}\mu_\ast$, and work with the case $\sigma_\ast\in (0,1)$. Let the entries of $\mathsf{Z}_{m\times n} \in \R^{m\times n}$ be i.i.d. $\mathcal{N}(0,1)$.

\noindent (\textbf{Step 1}). In this step, we prove the following concentration estimate: with $\Prob^{(0)}$-probability at least $1- \exp(-\log^D n)$, 
\begin{align}\label{ineq:gd_onestep_step1}
\biggabs{ \frac{1}{n}\sum_{j \in [n]} \Big( \psi(\bar{\mu}_j)- \E^{(0)} \psi(\bar{\mu}_j)\Big)  }\leq (K\Lambda L_\ast)^{c_1}\cdot \bigg(\frac{\log n}{n}\bigg)^{1/2}.
\end{align}
To this end, let $H: \R^{m\times n}\to \R$ be defined as
\begin{align*}
H(Z)&\equiv \frac{1}{n}\sum_{j \in [n]} \psi \,\Big(\tau_\ast \mu_{0,j} - e_j^\top\Sigma^{-1/2} Z^\top \mathfrak{S}\big(Zw_0,Zw_\ast,\xi\big)\Big).
\end{align*}
Then we have the estimate
\begin{align}\label{ineq:gd_onestep_step1_1}
\abs{H(Z)}&\leq \Lambda^c \abs{\tau_\ast}\cdot (n^{-1/2}\pnorm{\mu_0}{})+ \Lambda^c \big(1+n^{-1/2}(\pnorm{w_0}{}+\pnorm{w_\ast}{})\big)^{c_1}\cdot \big(1+\pnorm{Z}{\op}\big)^{c_1}\nonumber\\
&\leq \phi \cdot (\Lambda L_\ast)^{c_1}\cdot \big(1+\pnorm{Z}{\op}\big)^{c_1}.
\end{align}
Moreover, for any $Z_1,Z_2\in \R^{m\times n}$, 
\begin{align}\label{ineq:gd_onestep_step1_2}
&\abs{H(Z_1)-H(Z_2)}\nonumber\\
&\leq \Lambda^c\cdot \big(1+n^{-1/2}(\pnorm{w_0}{}+\pnorm{w_\ast}{})\big)^{c_1}\cdot \big(1+\pnorm{Z_1}{\op}+\pnorm{Z_2}{\op}\big)^{c_1}\cdot \pnorm{Z_1-Z_2}{\op}\nonumber\\
&\leq (\Lambda L_\ast)^{c_1}\cdot \big(1+\pnorm{Z_1}{\op}+\pnorm{Z_2}{\op}\big)^{c_1}\cdot \pnorm{Z_1-Z_2}{\op}.
\end{align}
Now we may apply Lemma \ref{lem:gaussian_conc} with (\ref{ineq:gd_onestep_step1_1})-(\ref{ineq:gd_onestep_step1_2}) to obtain that with $\Prob^{(0)}$-probability at least $1-\exp(-\log^D n)$, 
\begin{align*}
\hbox{LHS of (\ref{ineq:gd_onestep_step1})}=\abs{H(n^{-1/2}\mathsf{Z}_{m\times n})-\E^{(0)} H(n^{-1/2}\mathsf{Z}_{m\times n})}\leq (K\Lambda L_\ast)^{c_1}\cdot (\log n/n)^{1/2},
\end{align*}
proving (\ref{ineq:gd_onestep_step1}). 

\noindent (\textbf{Step 2}). In this step, we prove that for $n\geq (K\Lambda L_\ast)^{c_1'}$, 
\begin{align}\label{ineq:gd_onestep_step2}
\biggabs{ \frac{1}{n}\sum_{j \in [n]} \Big( \E^{(0)}\psi(\bar{\mu}_j)- \E^{(0)} \psi\big(\delta_\ast \mu_{\ast,j}+\sigma_\ast \sigma_j\mathsf{Z} \big)\Big)  }\leq \frac{(K\Lambda L_\ast/\sigma_\ast)^{c_1} }{n^{1/2}}.
\end{align}
For notational simplicity, we write
\begin{align*}
X_i\equiv \Sigma^{-1/2} Z_i \mathfrak{S}\big(n^{-1/2}\iprod{Z_{i}}{w_0},n^{-1/2}\iprod{Z_{i}}{w_\ast},\xi_{i}\big). 
\end{align*}
Note that by Gaussian integration-by-parts, we may compute the bias
\begin{align*}
&\E^{(0)} \bar{\mu}(A)
= \tau_\ast \mu_0 - n^{-1/2}\cdot m \E^{(0)} X_{\pi_m}\\
& = \tau_\ast \mu_0 - \frac{\Sigma^{-1/2}}{n^{1/2}} \sum_{i \in [m]} \E^{(0)} Z_i \mathfrak{S}\big(n^{-1/2}\iprod{Z_{i}}{w_0},n^{-1/2}\iprod{Z_{i}}{w_\ast},\xi_{i}\big)\\
& =\tau_\ast \mu_0 - \Sigma^{-1/2} \big(\tau_\ast w_0-\delta_\ast w_\ast\big) = \delta_\ast \mu_\ast. 
\end{align*}
Therefore, in order to prove (\ref{ineq:gd_onestep_step2}), it suffices to prove that for $n\geq (K\Lambda L_\ast)^{c_1'}$, 
\begin{align}\label{ineq:gd_onestep_step2_1}
&\frac{1}{n}\sum_{j \in [n]} \biggabs{\E^{(0)} \psi \, \bigg(\frac{1}{n^{1/2}}\sum_{i \in [m]} (\mathrm{id}-\E^{(0)}) X_{i,j}\bigg)- \E^{(0)}\psi\big(\sigma_\ast \sigma_j \mathsf{Z}\big)}\leq \frac{(K\Lambda L_\ast/\sigma_\ast)^{c_1}}{n^{1/2}}.
\end{align}
In order to prove (\ref{ineq:gd_onestep_step2_1}), first note that with $b_j\equiv \Sigma^{-1/2}e_j$, by repeated applications of Gaussian integration-by-parts,
\begin{align}\label{ineq:gd_onestep_step2_2}
\E^{(0)} X_{i,j}^2 & =\E^{(0)} \iprod{Z_i}{b_j}^2\cdot  \mathfrak{S}^2\big(n^{-1/2}\iprod{Z_{i}}{w_0},n^{-1/2}\iprod{Z_{i}}{w_\ast},\xi_{i}\big)\nonumber\\
& = \pnorm{b_j}{}^2\cdot \E^{(0)}  \mathfrak{S}^2\big(n^{-1/2}\iprod{Z_{i}}{w_0},n^{-1/2}\iprod{Z_{i}}{w_\ast},\xi_{i}\big)+\mathrm{Rem}_{i,j},
\end{align}
where
\begin{align}\label{ineq:gd_onestep_step2_3}
\abs{\mathrm{Rem}_{i,j}}&\leq c\cdot \max_{\abs{\alpha}+\abs{\beta}=2} \bigabs{\E^{(0)}(\partial_\alpha\mathfrak{S}\cdot \partial_\beta \mathfrak{S}) \big(n^{-1/2}\iprod{Z_{i}}{w_0},n^{-1/2}\iprod{Z_{i}}{w_\ast},\xi_{i}\big)}\nonumber\\
&\qquad \times n^{-1}\cdot  \big(\iprod{b_j}{w_0}^2+\iprod{b_j}{w_\ast}^2+\abs{ \iprod{b_j}{w_0}}\cdot \abs{ \iprod{b_j}{w_\ast} }\big)\nonumber\\
& \leq (\Lambda L_\ast)^{c_1}\cdot n^{-1} (\mu_{0,j}^2+\mu_{\ast,j}^2). 
\end{align}
Combining (\ref{ineq:gd_onestep_step2_2})-(\ref{ineq:gd_onestep_step2_3}), we have
\begin{align}\label{ineq:gd_onestep_step2_4}
\abs{\phi\cdot  \E^{(0)} X_{i,j}^2- \sigma_\ast^2\sigma_j^2 }\leq (\Lambda L_\ast)^{ c_1 }\cdot n^{-1}(\mu_{0,j}^2+\mu_{\ast,j}^2).
\end{align}
Moreover,
\begin{align}\label{ineq:gd_onestep_step2_5}
(\E^{(0)} X_{i,j})^2& = n^{-1} \big(\tau_\ast \iprod{b_j}{w_0}-\delta_\ast\iprod{b_j}{w_\ast}\big)^2\nonumber\\
&\leq (K\Lambda L_\ast)^{c_1}\cdot n^{-1}(\mu_{0,j}^2+\mu_{\ast,j}^2). 
\end{align}
Combining (\ref{ineq:gd_onestep_step2_4})-(\ref{ineq:gd_onestep_step2_5}), with $\check{\sigma}_j\equiv \var^{1/2}\big(\phi^{1/2}X_{1,j}\big)$, we have 
\begin{align}\label{ineq:gd_onestep_step2_5_0}
\abs{\check{\sigma}_j^2-\sigma_\ast^2\sigma_j^2}\leq (K\Lambda L_\ast)^{c_1}\cdot n^{-1} (\mu_{0,j}^2+\mu_{\ast,j}^2). 
\end{align}
This means that 
\begin{align}\label{ineq:gd_onestep_step2_6}
\frac{1}{n}\sum_{j \in [n]} \bigabs{ \E^{(0)}\psi\big(\check{\sigma}_j \mathsf{Z}\big)-  \E^{(0)}\psi\big(\sigma_\ast \sigma_j \mathsf{Z}\big)} \leq \frac{(K\Lambda L_\ast)^{c_1} }{n\sigma_\ast}.
\end{align}
So in order to prove (\ref{ineq:gd_onestep_step2_1}), it now suffices to prove that for $n\geq (K\Lambda L_\ast)^{c_1' }$, 
\begin{align}\label{ineq:gd_onestep_step2_7}
&\frac{1}{n}\sum_{j \in [n]} \biggabs{\E^{(0)} \psi \, \bigg(\frac{1}{n^{1/2}}\sum_{i \in [m]} (\mathrm{id}-\E^{(0)}) X_{i,j}\bigg)- \E^{(0)}\psi(\check{\sigma}_j \mathsf{Z})}\leq  \frac{(K\Lambda L_\ast/\sigma_\ast)^{c_1 }}{n^{1/2}}.
\end{align}
By the one-dimensional Wasserstein Berry-Esseen bound, cf. \cite[Theorem 3.1]{chen2011normal}, applied conditionally on $(\mu_0,\mu_\ast,\xi)$, the $j$-th term in the left hand side of (\ref{ineq:gd_onestep_step2_7}) is bounded by
\begin{align*}
\frac{C\Lambda}{n^{3/2}\check\sigma_j^3}\sum_{i\in[m]}\E^{(0)}\abs{X_{i,j}-\E^{(0)}X_{i,j}}^3.
\end{align*}
Using $\E^{(0)}\abs{X_{i,j}}^3\leq (\Lambda L_\ast)^{c_1}$ and (\ref{ineq:gd_onestep_step2_5_0}), we have, for $n\geq (K\Lambda L_\ast)^{c_1'}$,
\begin{align*}
\hbox{LHS of (\ref{ineq:gd_onestep_step2_7})}\leq \frac{(K\Lambda L_\ast)^{c_1}}{n^{1/2}}\cdot \frac{1}{n}\sum_{j\in[n]}\frac{1}{(\sigma_\ast\sigma_j)^3}\leq \frac{(K\Lambda L_\ast)^{c_1}}{n^{1/2}\sigma_\ast^3},
\end{align*}
where the last inequality uses $\min_j\sigma_j\geq \Lambda^{-1/2}$. This proves (\ref{ineq:gd_onestep_step2_7}) after adjusting constants.

\noindent (\textbf{Step 3}). The claimed estimate now follows by combining (\ref{ineq:gd_onestep_step1}) and (\ref{ineq:gd_onestep_step2}).
\end{proof}

\begin{lemma}\label{lem:tau_sigma_err}
	Suppose Assumption \ref{assump:model} holds for some $K,\Lambda\geq 2$. Fix $D>0$. Then there exists some $c_1= c_1(D)>1$ such that with $\Prob^{(0)}$-probability at least $1-\exp(-\log^D n)$,
	\begin{align*}
	n^{1/2}\cdot \big(\abs{\hat{\tau}-\tau_\ast}+ \abs{ \hat{\sigma}^2-\sigma_\ast^2 }\big)\leq \big(K\Lambda L_\ast\log n\big)^{c_1}.
	\end{align*}
\end{lemma}
\begin{proof}
	Conditional on $(\mu_0,\mu_\ast,\xi)$, the summands defining $\hat\tau$ and $\hat\sigma^2$ are independent. Assumption \ref{assump:model}-(A3) implies the sub-exponential Orlicz bounds $\|\partial_1\mathfrak S((A\mu_0)_i,(A\mu_\ast)_i,\xi_i)\|_{\psi_1}+\|\mathfrak S^2((A\mu_0)_i,(A\mu_\ast)_i,\xi_i)\|_{\psi_1}\leq (K\Lambda L_\ast)^{c}$ uniformly in $i$. The conditional means are respectively $\tau_\ast$ and $\sigma_\ast^2$. Bernstein's inequality for independent non-identically distributed sub-exponential variables, for instance \cite[Theorem 3.1]{kuchibhotla2022moving}, gives the displayed bound after increasing $c_1$.
\end{proof}

\begin{proof}[Proof of Proposition \ref{prop:db_gd_normal}]
Recall $\bar{\mu}$ defined in (\ref{def:gd_debias_oracle}). We shall work under the sample size requirement $n\geq (K\Lambda L_\ast)^{ c_1}$ and $\sigma_\ast \in (0,1)$.

\noindent (\textbf{Step 1}).  In this step, we prove that with $\Prob^{(0)}$-probability at least $1- \exp(-\log^D n)$,
\begin{align}\label{ineq:db_gd_normal_step1}
&\biggabs{\frac{1}{n}\sum_{j \in [n]} \psi_j (\bar{\mu}_j)- \frac{1}{n} \sum_{j \in [n]} \psi_j \big(\mu_j^{\mathrm{db}}\big)}\leq \big(K\Lambda L_\ast/\sigma_\ast\big)^{c_1}\cdot \bigg(\frac{\log n}{n}\bigg)^{1/2}.
\end{align}
As the map $Z\mapsto n^{-1}\sum_{j \in [n]}  \psi_j \big(\delta_\ast \mu_{\ast}+  \sigma_\ast \Sigma^{-1/2}Z\big)$ is $\Lambda\sigma_\ast \pnorm{\Sigma^{-1/2}}{\op}\cdot  n^{-1/2}$-Lipschitz and $\sigma_\ast,\sigma_j\leq (K\Lambda L_\ast)^{c_1}$ by assumption, by Gaussian concentration inequality, with probability at least $1-e^{-x}$,
\begin{align}\label{ineq:db_gd_normal_step1_2}
\biggabs{\frac{1}{n} \sum_{j \in [n]} \big(\mathrm{id}-\E^{(0)}\big) \psi_j \big(\mu_j^{\mathrm{db}}\big)}\leq (K\Lambda L_\ast)^{c_1} \cdot \sqrt{x}\cdot n^{-1/2}.
\end{align} 
The claim (\ref{ineq:db_gd_normal_step1}) follows from (\ref{ineq:db_gd_normal_step1_2}) and an application of Lemma \ref{lem:gd_onestep}.

\noindent (\textbf{Step 2}). In this step, we prove that  with $\Prob^{(0)}$-probability at least $1-\exp(-\log^D n)$, 
\begin{align}\label{ineq:db_gd_normal_step2}
\biggabs{\frac{1}{n}\sum_{j \in [n]} \psi_j (\bar{\mu}_j)- \frac{1}{n}\sum_{j \in [n]} \psi_j (\hat{\mu}_j)}\leq \big(K\Lambda L_\ast\log n\big)^{c_1}\cdot n^{-1/2}.
\end{align}
To prove (\ref{ineq:db_gd_normal_step2}), first note that 
\begin{align}\label{ineq:db_gd_normal_step2_1}
\hbox{LHS of (\ref{ineq:db_gd_normal_step2})}&\leq \Lambda\cdot n^{-1/2}\pnorm{\bar{\mu}-\hat{\mu}}{}\leq \Lambda L_\ast\cdot  \abs{\tau_\ast-\hat{\tau}}.
\end{align}
By Lemma \ref{lem:tau_sigma_err}, it holds with $\Prob^{(0)}$-probability at least $1-\exp(-\log^D n)$ that,
\begin{align*}
\hbox{RHS of (\ref{ineq:db_gd_normal_step2_1})}&\leq \big(K\Lambda L_\ast\log n\big)^{c_1}\cdot n^{-1/2}.
\end{align*}
This proves the claim (\ref{ineq:db_gd_normal_step2}).

\noindent (\textbf{Step 3}). The claim of the lemma follows from (\ref{ineq:db_gd_normal_step1}) and (\ref{ineq:db_gd_normal_step2}) and adjusting constants. 
\end{proof}

\subsection{CML over debiased gradient descent $\hat{\mu}$}

With $\mathsf{Z}_n\sim \mathcal{N}(0,I_n)$  independent of all other variables, let
\begin{align}\label{def:mu_db_bar}
\bar{\mu}^{\mathrm{db}}\equiv \bar{\delta} \mu_{\ast}+  \bar{\sigma} \Sigma^{-1/2}\mathsf{Z}_n.
\end{align}
The following proposition shows that $\hat{G}$ defined in (\ref{def:gd_MLE}) is a near maximizer for $G\mapsto n^{-1}\sum_{j \in [n]} \log \varphi_{G;\bar{\sigma}\sigma_j}(\bar{\mu}^{\mathrm{db}}_j)$ in a suitable sense.

\begin{proposition}\label{prop:mmle_max_seq}
	Suppose Assumptions \ref{assump:model} and \ref{assump:prior} hold for some $K,\Lambda,M\geq 2$, and $K\Lambda M (1\wedge \bar{\sigma})^{-1}\leq (\log n)^{c_0}$ for some $c_0>1$. Fix $D>0$. Then there exists some $c_1=c_1(c_0,D)>0$ and an event $E_0$ with $\Prob^{\xi}(E_0^c)\leq n^{-D}$, such that for any $\mathscr{G}_{\texttt{test}}\subset \mathscr{G}$, it holds on the event $E_0\cap \{\hat{G} \in \mathscr{G}_{\texttt{test}} \}$ that
	\begin{align*}
	\biggabs{\frac{1}{n}\sum_{j \in [n]} \log \varphi_{\hat{G};\hat{\sigma}\sigma_j}(\hat{\mu}_j) - \max_{G \in \mathscr{G}(\pnorm{\mu^{\mathrm{db}}}{\infty},\infty)\cap \mathscr{G}_{\texttt{test}}} \frac{1}{n}\sum_{j \in [n]} \log \varphi_{G;\bar{\sigma}\sigma_j}(\bar{\mu}_j^{\mathrm{db}}) }\leq \frac{(\log n)^{c_1}}{n^{1/2}}.
	\end{align*}
	Moreover, the constraint $\mathscr{G}(\pnorm{\mu^{\mathrm{db}}}{\infty},\infty)$ in the maximum above may be replaced by $\mathscr{G}(L,\infty)$ for any $L\geq \pnorm{\mu^{\mathrm{db}}}{\infty}\vee\pnorm{\bar\mu^{\mathrm{db}}}{\infty}$.
\end{proposition}

We first prove the following.
\begin{lemma}\label{lem:sigma_ast_comp}
Suppose Assumption \ref{assump:model} holds for some $K,\Lambda\geq 2$. On the event $\big\{ \hat{\sigma} \in [\bar{\sigma}/2,2\bar{\sigma}]\big\} \cap \{\hat{G}\in \mathscr{G}_{\texttt{test}}\}$, it holds for some universal constant $c>0$ that
\begin{align*}
&\biggabs{\frac{1}{n}\sum_{j \in [n]} \log \varphi_{\hat{G};\bar{\sigma}\sigma_j}(\hat{\mu}_j) - \max_{G \in \mathscr{G}(\pnorm{\hat{\mu}}{\infty},\infty)\cap \mathscr{G}_{\texttt{test}}} \frac{1}{n}\sum_{j \in [n]} \log \varphi_{G;\bar{\sigma}\sigma_j}(\hat{\mu}_j)}\\
&\leq \frac{c\Lambda^3}{\bar{\sigma}}\bigg(1+\frac{ \pnorm{\hat{\mu}}{\infty} }{\bar{\sigma}}\bigg)^2\cdot \abs{\hat{\sigma}-\bar{\sigma}}
\end{align*}
Moreover, the constraint $\mathscr{G}(\pnorm{\hat{\mu}}{\infty},\infty)$ in the maximum above may be replaced by $\mathscr{G}(L,\infty)$ for any $L\geq \pnorm{\hat{\mu}}{\infty}$.
\end{lemma}
\begin{proof}
Note that $\hat{G}$ is supported on $[-\pnorm{\hat{\mu}}{\infty},\pnorm{\hat{\mu}}{\infty}]$. Hence, on $\{\hat G\in\mathscr G_{\texttt{test}}\}$, we may restrict the maximum in (\ref{def:gd_MLE}) to $\mathscr{G}(\pnorm{\hat{\mu}}{\infty},\infty)\cap  \mathscr{G}_{\texttt{test}}$:
\begin{align}\label{ineq:sigma_ast_comp_0}
\frac{1}{n}\sum_{j \in [n]} \log \varphi_{\hat{G};\hat{\sigma}\sigma_j}(\hat{\mu}_j) =\max_{G \in \mathscr{G}(\pnorm{\hat{\mu}}{\infty},\infty)\cap \mathscr{G}_{\texttt{test}}} \frac{1}{n}\sum_{j \in [n]} \log \varphi_{G;\hat{\sigma}\sigma_j}(\hat{\mu}_j).
\end{align}
We shall perform an error analysis on the both sides of (\ref{ineq:sigma_ast_comp_0}) when $\hat{\sigma}$ is replaced by $\bar{\sigma}$. Using Lemma \ref{lem:log_normal_mix_der}, on the event $E_0\equiv \big\{\hat{\sigma} \in [\bar{\sigma}/2,2\bar{\sigma}]\big\}$,
\begin{align}\label{ineq:sigma_ast_comp_1}
&\sup_{G \in \mathscr{G}(\pnorm{\hat{\mu}}{\infty},\infty)} \biggabs{\frac{1}{n}\sum_{j \in [n]} \big( \log \varphi_{G;\hat{\sigma}\sigma_j}(\hat{\mu}_j)- \log \varphi_{G;\bar{\sigma}\sigma_j}(\hat{\mu}_j)\big)}\nonumber\\
& \leq \frac{c\Lambda^3}{\bar{\sigma}}\bigg(1+\frac{ \pnorm{\hat{\mu}}{\infty} }{\bar{\sigma}}\bigg)^2\cdot \abs{\hat{\sigma}-\bar{\sigma}}.
\end{align}
Now the claim follows by applying the above estimate on both sides of (\ref{ineq:sigma_ast_comp_0}). 
\end{proof}

\begin{lemma}\label{lem:one_step_gd_deloc}
Suppose Assumption \ref{assump:model} holds for some $K,\Lambda\geq 2$. Fix $D>0$. Then there exists some $c_1\equiv c_1(D)>1$ such that with $\Prob^{(0)}$-probability at least $1-\exp(-\log^D n)$,
\begin{align*}
\pnorm{\hat{\mu}}{\infty}\leq (K\Lambda L_\ast \log n)^{c_1}.
\end{align*}
\end{lemma}
\begin{proof}
Note that
\begin{align}\label{ineq:one_step_gd_deloc_1}
\pnorm{\hat{\mu}}{\infty}\leq \abs{\hat{\tau}}\cdot L_\ast+\pnorm{\Sigma^{-1} A^\top \mathsf{L}(A\mu_0,Y)}{\infty}.
\end{align}
First we consider the first term on the right hand side of (\ref{ineq:one_step_gd_deloc_1}). Under the polynomial growth condition, using the concentration inequality for sum of i.n.i.d. sub-Weibull variables (cf. \cite[Theorem 3.1]{kuchibhotla2022moving}), it holds with $\Prob^{(0)}$-probability at least $1-\exp(-\log^D n)$ that
\begin{align}\label{ineq:one_step_gd_deloc_2}
\abs{\hat{\tau}}&\leq (\Lambda L_\ast \log n)^{c_1}.
\end{align} 
Next we consider the second term on the right hand side of (\ref{ineq:one_step_gd_deloc_1}). With $b_j\equiv \Sigma^{-1/2}e_j$, we may compute
\begin{align*}
&\pnorm{\Sigma^{-1} A^\top \mathsf{L}(A\mu_0,Y)}{\infty}\\
&=n^{-1/2} m\cdot  \max_{j \in [n]}  \bigabs{ \E_{\pi_m} \iprod{Z_{\pi_m}}{b_j} \mathfrak{S}\big(n^{-1/2}\iprod{Z_{\pi_m}}{w_0},n^{-1/2}\iprod{Z_{\pi_m}}{w_\ast},\xi_{\pi_m}\big)}.
\end{align*}
Using again the concentration inequality for sum of i.n.i.d. sub-Weibull variables and a union bound, it holds with $\Prob^{(0)}$-probability at least $1-\exp(-\log^D n)$ that
\begin{align}\label{ineq:one_step_gd_deloc_3_1}
&\pnorm{\Sigma^{-1} A^\top \mathsf{L}(A\mu_0,Y)}{\infty}\nonumber\\
&\leq n^{-1/2} m\cdot  \max_{j \in [n]}  \bigabs{ \E^{(0)} \iprod{Z_1}{b_j} \mathfrak{S}\big(n^{-1/2}\iprod{Z_1}{w_0},n^{-1/2}\iprod{Z_1}{w_\ast},\xi_{\pi_m}\big)}+ (K\Lambda L_\ast \log n)^{c_1}\nonumber\\
&\equiv (\phi m)^{1/2} \cdot \max_{j \in [n]} \abs{I_j}+ (K\Lambda L_\ast \log n)^{c_1}.
\end{align}
Now using Gaussian integration-by-parts for the first term above, 
\begin{align}\label{ineq:one_step_gd_deloc_3_2}
\abs{I_j}&\leq c\cdot \max_{\alpha=1,2} \bigabs{\E^{(0)}\partial_\alpha\mathfrak{S} \big(n^{-1/2}\iprod{Z_{i}}{w_0},n^{-1/2}\iprod{Z_{i}}{w_\ast},\xi_{i}\big)}\nonumber\\
&\qquad \times n^{-1/2}\cdot  \big(\abs{\iprod{b_j}{w_0}}+\abs{\iprod{b_j}{w_\ast}}\big)\leq (\Lambda L_\ast)^{c_1}\cdot n^{-1/2}.
\end{align}
Combining (\ref{ineq:one_step_gd_deloc_3_1})-(\ref{ineq:one_step_gd_deloc_3_2}), it holds with $\Prob^{(0)}$-probability at least $1-\exp(-\log^D n)$ that
\begin{align}\label{ineq:one_step_gd_deloc_3}
\pnorm{\Sigma^{-1} A^\top \mathsf{L}(A\mu_0,Y)}{\infty}\leq (K\Lambda L_\ast \log n)^{c_1}.
\end{align}
The claim now follows by combining (\ref{ineq:one_step_gd_deloc_1}), (\ref{ineq:one_step_gd_deloc_2}) and (\ref{ineq:one_step_gd_deloc_3}).
\end{proof}

\begin{lemma}\label{lem:diff_ast_bar_se}
Suppose (A1), (A3) and Assumption \ref{assump:prior} hold for some $K,\Lambda,M\geq 2$. Fix $D>0$. Then there exists some $c_1=c_1(D)>1$ such that with $\Prob^\xi$-probability at least $1-\exp(-\log^D n)$,
\begin{align*}
n^{1/2}\cdot \big(\abs{\tau_\ast-\bar{\tau}}+ \abs{\sigma_\ast^2-\bar{\sigma}^2 }+ \abs{\delta_\ast-\bar{\delta}}\big)\leq \big(K\Lambda M\log n\big)^{c_1}.
\end{align*}
\end{lemma}
\begin{proof}
It suffices to provide a high probability bound for $\pnorm{\Delta \Sigma_{\mathfrak{U}}}{\op}\equiv 
\pnorm{\cov(\mathfrak{U}_\ast,\mathfrak{U}_\ast)-\cov(\bar{\mathfrak{U}},\bar{\mathfrak{U}})}{\op}$ over the randomness induced by $\mu_0,\mu_\ast$. 

The diagonal elements of $\Delta \Sigma_{\mathfrak{U}}$ can be handled immediately by Hanson-Wright inequality. For instance, consider the $(1,1)$ element of $\Delta \Sigma_{\mathfrak{U}}$. As $\pnorm{\mu_0}{\Sigma}^2= \mu_0^\top\Sigma \mu_0$, by \cite[Theorem 1.1]{rudelson2013hanson}, for $x\geq 1$ it hold with $\Prob^\xi$-probability at least $1-2e^{-x}$ that
\begin{align*}
\bigabs{\pnorm{\mu_0}{\Sigma}^2-\E^\xi \mu_0^\top\Sigma \mu_0}\leq c M^2\cdot \big(\pnorm{\Sigma}{F} x^{1/2}+\pnorm{\Sigma}{\op} x\big)\leq c \Lambda M^2 \sqrt{n}\cdot x.
\end{align*}
The claim now follows by computing the normalized mean
\begin{align*}
\frac1n\E^\xi \mu_0^\top\Sigma \mu_0
=\frac1n\sum_{i,j}\Sigma_{ij}\E^\xi\mu_{0,i}\mu_{0,j}
= a_\Sigma \var(G_0)+ b_\Sigma(\E G_0)^2.
\end{align*}
The $(2,2)$ element of $\Delta \Sigma_{\mathfrak{U}}$ can be handled in the same manner. 

The off-diagonal elements of $\Delta \Sigma_{\mathfrak{U}}$ can be handled in a similar way by applying Hanson-Wright inequality to the quadratic form
\begin{align*}
\mu_0^\top \Sigma \mu_\ast = \frac{1}{2}\binom{\mu_0}{\mu_\ast}^\top 
\begin{pmatrix}
0_{n\times n} & \Sigma\\
\Sigma^\top & 0_{n\times n}
\end{pmatrix}
\binom{\mu_0}{\mu_\ast},
\end{align*}
and computing the normalized mean 
$n^{-1}\E^\xi \mu_0^\top\Sigma \mu_\ast
= a_\Sigma \cov(G_0,G_\ast)+b_\Sigma (\E G_0\E G_\ast)$. 
The proof is complete.
\end{proof}

\begin{proof}[Proof of Proposition \ref{prop:mmle_max_seq}]
Recall $\mu^{\mathrm{db}}$ defined in (\ref{def:mu_db}). By Lemmas \ref{lem:tau_sigma_err}, \ref{lem:one_step_gd_deloc} and \ref{lem:diff_ast_bar_se}, we shall work on an event $E_0$ with $\Prob^\xi(E_0^c)\leq n^{-D}$ such that
\begin{align}\label{ineq:mmle_max_seq_0}
&\pnorm{\hat{\mu}}{\infty}+\pnorm{\mu^{\mathrm{db}}}{\infty}+\pnorm{\bar{\mu}^{\mathrm{db}}}{\infty}+n^{1/2}\abs{\hat{\sigma}-\bar{\sigma}}\leq (\log n)^{c_1},
\end{align}
where $c_1=c_1(c_0,D)>2$. 

Before proceeding, we note a simple relation: for any $G \in \mathscr{G}$ and $\sigma>0$, with $G_\sigma(\cdot)\equiv G(\sigma\cdot)$, we have 
$\varphi_{G;\sigma}(x)=\sigma^{-1}\varphi_{G_\sigma;1}(x/\sigma)$. Consequently, for any $\epsilon \in (0,1/2)$, using \cite[Lemma 2]{zhang2009generalized}, there exists some $\mathscr{G}_\epsilon\subset \mathscr{G}\big((\log n)^{c_1}\big)$ with cardinality
\begin{align*} 
\log \abs{\mathscr{G}_\epsilon}\leq c_0\cdot \log^2(1/\epsilon)\cdot \big[1+(\log n)^{c_1}\big/{\sqrt{\log(1/\epsilon)}}\big],
\end{align*}
such that for any $j \in [n]$,
\begin{align*}
\max_{G \in \mathscr{G}(\pnorm{\hat{\mu}}{\infty},\infty)} \min_{G' \in \mathscr{G}_\epsilon} \pnorm{\varphi_{G';\bar{\sigma}\sigma_j}-\varphi_{G;\bar{\sigma}\sigma_j}   }{\infty}\leq \epsilon/(\bar{\sigma}\sigma_j)\leq \Lambda\epsilon/\bar{\sigma}.
\end{align*}
This means with $\epsilon\equiv \epsilon_n\equiv  n^{-1/2}e^{-2\Lambda^2 (\log n)^{2c_1}/\bar{\sigma}^2}$, we have $\log \abs{\mathscr{G}_\epsilon}\leq (\log n)^{6c_1}$, and it holds that 
\begin{align*}
&\max_{G \in \mathscr{G}(\pnorm{\hat{\mu}}{\infty},\infty)} \min_{G' \in \mathscr{G}_\epsilon} \max_{j \in [n]}\, \bigabs{\log \varphi_{G';\bar{\sigma}\sigma_j}(\hat{\mu}_j)-\log \varphi_{G;\bar{\sigma}\sigma_j}(\hat{\mu}_j)   }\\
&\leq c_0 \epsilon\cdot e^{2\Lambda^2 \pnorm{\hat{\mu}}{\infty}^2/\bar{\sigma}^2}= c_0\epsilon\cdot e^{2\Lambda^2 (\log n)^{2c_1}/\bar{\sigma}^2}\leq c_0 \cdot n^{-1/2}.
\end{align*}
Consequently, on $E_0$, 
\begin{align*}
\biggabs{\max_{G \in \mathscr{G}(\pnorm{\hat{\mu}}{\infty},\infty)\cap \mathscr{G}_{\texttt{test}}} \frac{1}{n}\sum_{j \in [n]} \log \varphi_{G;\bar{\sigma}\sigma_j}(\hat{\mu}_j)-\max_{G' \in \mathscr{G}_\epsilon\cap \mathscr{G}_{\texttt{test}}} \frac{1}{n}\sum_{j \in [n]} \log \varphi_{G';\bar{\sigma}\sigma_j}(\hat{\mu}_j)}\leq c_0\cdot n^{-1/2}.
\end{align*}
Similarly, the above display holds by replacing $\{\hat{\mu}_j\}$ with $\big\{\mu^{\mathrm{db}}_j\big\}$.

Now by applying Proposition \ref{prop:db_gd_normal} with $D\equiv 6c_1+100$, on the event $E_0\cap E_1\cap \{\hat{G}\in \mathscr{G}_{\texttt{test}}\}$ where $\Prob^\xi(E_1^c)\leq \exp(-\log^{100} n)$, with $\err_n(c)\equiv n^{-1/2}(\log n)^{c}$ and a constant $c_1'=c_1'(c_0)>0 $ whose numerical value may vary from line to line,
\begin{align}\label{ineq:mmle_max_seq_1}
&\frac{1}{n}\sum_{j \in [n]} \log \varphi_{\hat{G};\bar{\sigma}\sigma_j}(\hat{\mu}_j) \nonumber\\
&= \max_{G \in \mathscr{G}(\pnorm{\hat{\mu}}{\infty},\infty)\cap \mathscr{G}_{\texttt{test}}} \frac{1}{n}\sum_{j \in [n]} \log \varphi_{G;\bar{\sigma}\sigma_j}(\hat{\mu}_j)+\bigo \big(\err_n(c_1')\big)\nonumber\\
&= \max_{G' \in \mathscr{G}_\epsilon\cap \mathscr{G}_{\texttt{test}}} \frac{1}{n}\sum_{j \in [n]} \log \varphi_{G';\bar{\sigma}\sigma_j}(\hat{\mu}_j)+\bigo \big(\err_n(c_1')\big)\nonumber\\
&= \max_{G' \in \mathscr{G}_\epsilon\cap \mathscr{G}_{\texttt{test}}} \frac{1}{n}\sum_{j \in [n]} \log \varphi_{G';\bar{\sigma}\sigma_j}(\mu_j^{\mathrm{db}})+\bigo \big(\err_n(c_1')\big)\nonumber\\
&= \max_{G \in \mathscr{G}(\pnorm{\mu^{\mathrm{db}}}{\infty},\infty)\cap \mathscr{G}_{\texttt{test}} } \frac{1}{n}\sum_{j \in [n]} \log \varphi_{G;\bar{\sigma}\sigma_j}(\mu_j^{\mathrm{db}})+\bigo \big(\err_n(c_1')\big).
\end{align}
Finally, by Lemma \ref{lem:der_log_mix_est}, for any $G \in \mathscr{G}(\pnorm{\mu^{\mathrm{db}}}{\infty},\infty)$,
\begin{align}\label{ineq:mmle_max_seq_2}
&\biggabs{\frac{1}{n}\sum_{j \in [n]} \log \varphi_{G;\bar{\sigma}\sigma_j}(\mu_j^{\mathrm{db}})- \frac{1}{n}\sum_{j \in [n]} \log \varphi_{G;\bar{\sigma}\sigma_j}(\bar{\mu}_j^{\mathrm{db}})}\nonumber\\
&\leq c\cdot \big(\bar{\sigma}^2\cdot \min_j\sigma_j^2\big)^{-1}\cdot \big(\pnorm{ \mu^{\mathrm{db}} }{\infty}+\pnorm{ \bar{\mu}^{\mathrm{db}} }{\infty}\big)\cdot \pnorm{\mu^{\mathrm{db}}-\bar\mu^{\mathrm{db}}}{\infty}\nonumber\\
&\leq n^{-1/2}(\log n)^{c_2}.
\end{align}
The claim follows by combining the above two displays (\ref{ineq:mmle_max_seq_1})-(\ref{ineq:mmle_max_seq_2}) and an application of Lemma \ref{lem:sigma_ast_comp}.
\end{proof}

\subsection{Proof of Theorem \ref{thm:eb_gd_hellinger}}

We write, for any $G\in\mathscr G$,
\begin{align}\label{ineq:eb_gd_revised_loglr}
\hat{\mathfrak L}_{n}(G)
&\equiv
\frac{1}{\sqrt n}\sum_{j \in [n]}
\log\frac{\varphi_{G;\hat\sigma\sigma_j}(\hat\mu_j)}
{\varphi_{\bar G_\ast;\hat\sigma\sigma_j}(\hat\mu_j)},\quad \bar{\mathfrak L}_{n}(G)\equiv
\frac{1}{\sqrt n}\sum_{j \in [n]}
\log\frac{\varphi_{G;\bar\sigma\sigma_j}(\bar\mu^{\mathrm{db}}_j)}
{\varphi_{\bar G_\ast;\bar\sigma\sigma_j}(\bar\mu^{\mathrm{db}}_j)} .
\end{align}
Here $\bar\mu^{\mathrm{db}}$ is the oracle Gaussian sequence defined in
\eqref{def:mu_db_bar}.  Fix a sufficiently large constant $A_0>1$, to be chosen
at the end of the proof, and set
\begin{align}\label{ineq:eb_gd_revised_epsilon}
\epsilon_n\equiv A_0\big(n^{-1/4}+n_{\ast,\Sigma}^{-1/2}\big)(\log n)^{A_0}.
\end{align}
If $\epsilon_n\ge 1$, the asserted Hellinger bound is trivial after enlarging the
constant in the theorem.  We therefore assume $\epsilon_n<1$ and define
\begin{align}\label{ineq:eb_gd_revised_test_class}
\mathscr G_{\texttt{test}}
\equiv
\big\{G\in\mathscr G:
\mathfrak d_{H;\bar\sigma\sigma_{[n]}}(G,\bar G_\ast)>\epsilon_n
\big\}.
\end{align}
\noindent (\textbf{Step 1}). We first specify the high probability event on which the rest of the proof is
purely deterministic.  By the exponential-tail assumption in Assumption
\ref{assump:prior}, the Gaussian maximal inequality for the coordinates of
$\Sigma^{-1/2}\mathsf Z_n$, Lemmas \ref{lem:tau_sigma_err},
\ref{lem:one_step_gd_deloc}, \ref{lem:diff_ast_bar_se}, and Proposition
\ref{prop:mmle_max_seq}, there exist constants $C>1$ and an event
$E_{\mathrm{good}}$ with $
\Prob^\xi(E_{\mathrm{good}}^c)\le n^{-D-4}$
such that the following facts hold simultaneously on $E_{\mathrm{good}}$:
\begin{align}
&\|\mu_0\|_\infty+\|\mu_\ast\|_\infty
+\|\mu^{\mathrm{db}}\|_\infty+\|\bar\mu^{\mathrm{db}}\|_\infty
+ \|\hat\mu\|_\infty\le (\log n)^C,
\label{ineq:eb_gd_revised_good_sup_1}\\
&
n^{1/2}|\hat\sigma-\bar\sigma|
+n^{1/2}|\sigma_\ast-\bar\sigma|
+n^{1/2}|\delta_\ast-\bar\delta|
\le (\log n)^C,
\label{ineq:eb_gd_revised_good_sup_2}\\
&\bigg|
\frac1n\sum_{j \in [n]}\log\varphi_{\hat G;\hat\sigma\sigma_j}(\hat\mu_j)
-
\max_{G\in\mathscr G(R_n,\infty)\cap\mathscr G_{\texttt{test}}}
\frac1n\sum_{j \in [n]}\log\varphi_{G;\bar\sigma\sigma_j}(\bar\mu^{\mathrm{db}}_j)
\bigg|
\le \frac{(\log n)^C}{n^{1/2}}
\label{ineq:eb_gd_revised_near_max}
\end{align}
on the further event $\{\hat G\in\mathscr G_{\texttt{test}}\}$, where
$R_n\equiv(\log n)^C$.  Here the last display is Proposition
\ref{prop:mmle_max_seq}, using its final sentence to replace the support constraint
by $\mathscr G(R_n,\infty)$; the radius $R_n$ dominates
$\|\mu^{\mathrm{db}}\|_\infty\vee\|\bar\mu^{\mathrm{db}}\|_\infty$ on
$E_{\mathrm{good}}$.  Increasing $C$ if necessary, the same event also contains
$\hat\sigma\in[\bar\sigma/2,2\bar\sigma]$, which follows from the lower bound on
$\bar\sigma$ in the assumption
$K\Lambda(1\wedge\bar\sigma\wedge|\bar\delta|)^{-1}\le(\log n)^{c_0}$.

\noindent (\textbf{Step 2}). We next compare the denominator in the plug-in likelihood ratio with its oracle
counterpart.  We claim that, after enlarging $C$ and reducing the event by another
set of $\Prob^\xi$-probability at most $n^{-D-4}$ if necessary,
\begin{align}\label{ineq:eb_gd_revised_denom_compare}
\bigg|
\frac1n\sum_{j \in [n]}
\log\varphi_{\bar G_\ast;\hat\sigma\sigma_j}(\hat\mu_j)
-
\frac1n\sum_{j \in [n]}
\log\varphi_{\bar G_\ast;\bar\sigma\sigma_j}(\bar\mu^{\mathrm{db}}_j)
\bigg|
\le n^{-1/2}(\log n)^C .
\end{align}
Indeed, decompose the left side into the sum of
\begin{align*}
I_1&\equiv
\bigg|
\frac1n\sum_{j \in [n]}
\Big\{
\log\varphi_{\bar G_\ast;\hat\sigma\sigma_j}(\hat\mu_j)
-
\log\varphi_{\bar G_\ast;\bar\sigma\sigma_j}(\hat\mu_j)
\Big\}
\bigg|,\\
I_2&\equiv
\bigg|
\frac1n\sum_{j \in [n]}
\Big\{
\log\varphi_{\bar G_\ast;\bar\sigma\sigma_j}(\hat\mu_j)
-
\log\varphi_{\bar G_\ast;\bar\sigma\sigma_j}(\mu^{\mathrm{db}}_j)
\Big\}
\bigg|,\\
I_3&\equiv
\bigg|
\frac1n\sum_{j \in [n]}
\Big\{
\log\varphi_{\bar G_\ast;\bar\sigma\sigma_j}(\mu^{\mathrm{db}}_j)
-
\log\varphi_{\bar G_\ast;\bar\sigma\sigma_j}(\bar\mu^{\mathrm{db}}_j)
\Big\}
\bigg|.
\end{align*}
For $I_1$, note first that $\bar G_\ast$ inherits an exponential tail from
$G_\ast$ and the scalar $\bar\delta$, and $|\bar\delta|^{-1}\le(\log n)^{c_0}$.
Hence $\bar G_\ast([-R_n,R_n])\ge 1-n^{-D-10}$ after increasing the polylogarithmic
radius $R_n$.  Together with the lower bound on $\bar\sigma\min_j\sigma_j$, this
implies, uniformly over $|x|\le(\log n)^C$ and
$s\in[\bar\sigma\sigma_j/2,2\bar\sigma\sigma_j]$, $
\varphi_{\bar G_\ast;s}(x)\ge \exp\{-(\log n)^C\}$. 
Using Lemma \ref{lem:der_log_mix_est} and the preceding lower bound gives
\begin{align*}
\sup_{j\le n}\sup_{|x|\le(\log n)^C}
\sup_{s\in[\bar\sigma\sigma_j/2,2\bar\sigma\sigma_j]}
|\partial_s\log\varphi_{\bar G_\ast;s}(x)|
\le (\log n)^C.
\end{align*}
Therefore, by \eqref{ineq:eb_gd_revised_good_sup_2},
$I_1\le n^{-1/2}(\log n)^C$.

For $I_2$, apply Proposition \ref{prop:db_gd_normal}, conditionally on
$(\mu_0,\mu_\ast,\xi)$, to smooth truncations of the coordinate functions
$x\mapsto\log\varphi_{\bar G_\ast;\bar\sigma\sigma_j}(x)$.  The truncation is
inactive on the event \eqref{ineq:eb_gd_revised_good_sup_1}-\eqref{ineq:eb_gd_revised_good_sup_2}, and the preceding derivative bound supplies a common
polylogarithmic Lipschitz constant.  This yields
$I_2\le n^{-1/2}(\log n)^C$ with conditional probability at least
$1-n^{-D-4}$, after increasing $C$.

For $I_3$, use the same derivative bound and the identity
\begin{align*}
\mu^{\mathrm{db}}-\bar\mu^{\mathrm{db}}
=(\delta_\ast-\bar\delta)\mu_\ast
+(\sigma_\ast-\bar\sigma)\Sigma^{-1/2}\mathsf Z_n.
\end{align*}
The bounds in \eqref{ineq:eb_gd_revised_good_sup_1}-\eqref{ineq:eb_gd_revised_good_sup_2} and the Gaussian maximal inequality imply
$\|\mu^{\mathrm{db}}-\bar\mu^{\mathrm{db}}\|_\infty\le n^{-1/2}(\log n)^C$.
Consequently $I_3\le n^{-1/2}(\log n)^C$.  This proves
\eqref{ineq:eb_gd_revised_denom_compare}.

\noindent (\textbf{Step 3}).  Since
$\bar G_\ast$ is an admissible competitor,
\begin{align}\label{ineq:eb_gd_revised_hat_nonnegative}
0\le \hat{\mathfrak L}_n(\hat G).
\end{align}
Combining \eqref{ineq:eb_gd_revised_near_max} and
\eqref{ineq:eb_gd_revised_denom_compare}, and multiplying the normalized
average error by $\sqrt n$, gives the following deterministic implication on
$E_{\mathrm{good}}\cap\{\hat G\in\mathscr G_{\texttt{test}}\}$:
\begin{align}\label{ineq:eb_gd_revised_oracle_nearmax}
0
\le \hat{\mathfrak L}_n(\hat G)
\le
\sup_{G\in\mathscr G(R_n,\infty)\cap\mathscr G_{\texttt{test}}}
\bar{\mathfrak L}_n(G)
+(\log n)^{C_0}.
\end{align}

It remains to show that the supremum on the right is strictly negative.  Under the
oracle sequence
\begin{align*}
\bar\mu^{\mathrm{db}}=\bar\delta\mu_\ast+\bar\sigma\Sigma^{-1/2}\mathsf Z_n,
\end{align*}
the corresponding Gaussian sequence covariance is
$\Sigma_0=\bar\sigma^2\Sigma^{-1}$, the diagonal standard errors are
$\{\bar\sigma\sigma_j\}_{j\le n}$, and the diagonal-normalized covariance has
spectral radius $\pnorm{\texttt{Cor}_{\Sigma}}{\op}$.  Thus the effective sample size is
$n_{\ast,\Sigma}=n/\pnorm{\texttt{Cor}_{\Sigma}}{\op}$.  Applying the localized empirical-process bound
from the proof of Theorem \ref{thm:MML_rate}, equivalently
\eqref{ineq:MML_rate_step2_0} with this oracle sequence, we get an event
$E_{\mathrm{or}}$ with probability at least $1-n^{-D-4}$ such that, simultaneously
for all radii
$r\ge Cn_{\ast,\Sigma}^{-1/2}(\log n)^C$,
\begin{align}\label{ineq:eb_gd_revised_oracle_emp}
\sup_{\substack{G\in\mathscr G(R_n,\infty):
\mathfrak d_{H;\bar\sigma\sigma_{[n]}}(G,\bar G_\ast)\le r}}
\big|\bar{\mathfrak L}_n(G)-\E\bar{\mathfrak L}_n(G)\big|
\le C\cdot \pnorm{\texttt{Cor}_{\Sigma}}{\op}^{1/2}\,r(\log n)^C .
\end{align}
[Note that although \eqref{ineq:MML_rate_step2_0} is stated there for the canonical
logarithmic radius, the proof of Proposition \ref{prop:localized_truncated_process}
applies verbatim to the polylogarithmic radius $R_n=(\log n)^C$, as the
normal-mixture entropy remains a power of $\log n$, and the exponential-tail
truncation is absorbed by increasing the power $C$.]

For every $G\in\mathscr G$ the population likelihood ratio satisfies
\begin{align}\label{ineq:eb_gd_revised_oracle_pop}
\E\bar{\mathfrak L}_n(G)
&=-\frac1{\sqrt n}\sum_{j \in [n]}
\mathrm{KL}\big(
\varphi_{\bar G_\ast;\bar\sigma\sigma_j},
\varphi_{G;\bar\sigma\sigma_j}
\big)\le
-\sqrt n\,
\mathfrak d_{H;\bar\sigma\sigma_{[n]}}^2(G,\bar G_\ast).
\end{align}
We now perform the standard peeling argument.  For integers $k\ge0$, let
\begin{align*}
\mathscr A_k
\equiv
\big\{G\in\mathscr G(R_n,\infty):
2^k\epsilon_n<
\mathfrak d_{H;\bar\sigma\sigma_{[n]}}(G,\bar G_\ast)
\le 2^{k+1}\epsilon_n\big\}.
\end{align*}
Since the Hellinger distance is bounded by one, only $k\le C\log n$ need be
considered.  On the event $E_{\mathrm{or}}$, for all $G\in\mathscr A_k$,
\begin{align*}
\bar{\mathfrak L}_n(G)
&\le
\big\{\bar{\mathfrak L}_n(G)-\E\bar{\mathfrak L}_n(G)\big\}
+
\E\bar{\mathfrak L}_n(G)\\
& \le
C\cdot\pnorm{\texttt{Cor}_{\Sigma^{-1}}}{\op}^{1/2}\,2^{k+1}\epsilon_n(\log n)^C
-
\sqrt n\,2^{2k}\epsilon_n^2.
\end{align*}
By the definition \eqref{ineq:eb_gd_revised_epsilon}, after choosing $A_0$ large
and increasing logarithmic powers if necessary, $
1\ge \epsilon_n\ge A_0 n_{\ast,\Sigma}^{-1/2}(\log n)^{A_0}
=A_0\sqrt{\pnorm{\texttt{Cor}_{\Sigma^{-1}}}{\op}/n}\,(\log n)^{A_0}$, 
and hence, the first term in the right hand side of above can be bounded by $
\frac14\sqrt n\,2^{2k}\epsilon_n^2$ uniformly in $k\geq 0$. 
Therefore
\begin{align}\label{ineq:eb_gd_revised_shell_bound}
\sup_{G\in\mathscr A_k}\bar{\mathfrak L}_n(G)
\le
-\frac14\sqrt n\,2^{2k}\epsilon_n^2
\le
-\frac14\sqrt n\,\epsilon_n^2.
\end{align}
Taking the maximum over the shells, and using the other part of \eqref{ineq:eb_gd_revised_epsilon} that gives
$\epsilon_n\ge A_0n^{-1/4}(\log n)^{A_0}$, by taking $A_0$ still larger, we then arrive at 
\begin{align}\label{ineq:eb_gd_revised_test_sup_bound}
\sup_{G\in\mathscr G(R_n,\infty)\cap\mathscr G_{\texttt{test}}}
\bar{\mathfrak L}_n(G)
\le
-\frac14\sqrt n\,\epsilon_n^2\leq -2(\log n)^{C_0}.
\end{align}
Thus \eqref{ineq:eb_gd_revised_oracle_nearmax} and
\eqref{ineq:eb_gd_revised_test_sup_bound} are incompatible on
$E_{\mathrm{good}}\cap E_{\mathrm{or}}\cap\{\hat G\in\mathscr G_{\texttt{test}}\}$.
Consequently this event is empty, and with $\Prob^\xi$-probability at least
$1-n^{-D}$, $
\hat G\notin\mathscr G_{\texttt{test}}$, which proves the Hellinger assertion by the definition of $\mathscr G_{\texttt{test}}$.

\noindent (\textbf{Step 4}). Finally we prove the Wasserstein bound.  On $E_{\mathrm{good}}$, $
\mathrm{supp}(\hat G)\subset[-\|\hat\mu\|_\infty,\|\hat\mu\|_\infty]
\subset[-R_n,R_n]$ where $R_n=(\log n)^C$. 
Let $\Pi_R(x)=(-R)\vee x\wedge R$ and
$\bar G_\ast^{(R)}\equiv(\Pi_R)_\#\bar G_\ast$.  Since $\bar G_\ast$ inherits an
exponential tail from $G_\ast$ and $|\bar\delta|^{-1}\le(\log n)^{c_0}$, we may take
$R=(\log n)^C$ so large that, for any fixed $B>0$,
\begin{align}\label{ineq:eb_gd_revised_truncation_errors}
\mathsf W_p(\bar G_\ast,\bar G_\ast^{(R)})
\le n^{-B},
\qquad
\mathfrak d_{H;\bar\sigma\sigma_{[n]}}(\bar G_\ast,\bar G_\ast^{(R)})
\le n^{-B}.
\end{align}
Combining the proven Hellinger rate with
\eqref{ineq:eb_gd_revised_truncation_errors} gives
\begin{align*}
\mathfrak d_{H;\bar\sigma\sigma_{[n]}}(\hat G,\bar G_\ast^{(R)})
\le
\big(n^{-1/4}+n_{\ast,\Sigma}^{-1/2}\big)(\log n)^C+n^{-B}.
\end{align*}
Under Assumption \ref{assump:model}-(A2),
$\pnorm{\texttt{Cor}_{\Sigma^{-1}}}{\op}\le \Lambda^2$.  Moreover the standing condition
$K\Lambda M(1\wedge\bar\sigma\wedge|\bar\delta|)^{-1}\le(\log n)^{c_0}$ implies
$\Lambda\le(\log n)^{c_0}$.  Hence
$n_{\ast,\Sigma}^{-1/2}\le n^{-1/2}(\log n)^{c_0}$, and the last display is bounded
by $n^{-1/4}(\log n)^C$.  Both $\hat G$ and $\bar G_\ast^{(R)}$ are supported in
$[-R_n,R_n]$ for a possibly larger polylogarithmic $R_n$.  Lemma
\ref{lem:Wasserstein_from_Hellinger} therefore yields, for a constant $C_p>1$,
\begin{align*}
\mathsf W_p(\hat G,\bar G_\ast^{(R)})
\le
C_p\inf_{\delta\in(0,1)}
\Big\{
\delta+
e^{C_p\delta^{-2}}(\log n)^{C_p}\cdot 
\big(n^{-1/4}(\log n)^C+n^{-B}\big)^{1/C_p}
\Big\}.
\end{align*}
Choose $\delta=A_1(\log n)^{-1/2}$ with $A_1$ large enough.  Then
$\exp(C_p\delta^{-2})=n^{C_p/A_1^2}$; taking $A_1$ sufficiently large and then
$B$ sufficiently large makes the second term in the braces
$\mathfrak{o}((\log n)^{-1/2})$.  Hence $
\mathsf W_p(\hat G,\bar G_\ast^{(R)})\le C(\log n)^{-1/2}$.
Together with \eqref{ineq:eb_gd_revised_truncation_errors}, this proves the Wasserstein claim. 

\appendix

\section{Details of the simulation algorithm}\label{section:EM}

We present below the details for the fixed-grid EM approximation algorithm \cite{jiang2009general} to compute the maximum composite marginal likelihood (CML) estimator in Section \ref{sec:numerics}.

\begin{algorithm}[H]
	\caption{Fixed-grid EM approximation to the CML}\label{alg:grid_em_mple}
	\begin{algorithmic}[1]
		\STATE \textbf{Input}: pseudo-observations $U_1,\ldots,U_n$, standard errors $s_1,\ldots,s_n$, grid $\theta_1<\cdots<\theta_K$, tolerance \texttt{tol}, and maximum iteration number $T_{\max}$.
		\STATE Initialize $w^{(0)}=(1/K,\ldots,1/K)\in\Delta_K$.
		\FOR{$t=0,1,\ldots,T_{\max}-1$}
		\STATE Compute $
		r_{jk}^{(t)}
		={w_k^{(t)}\varphi_{s_j}(U_j-\theta_k)}/
		{\sum_{\ell=1}^K w_\ell^{(t)}\varphi_{s_j}(U_j-\theta_\ell)}$ for 
		$j\in[n],\ k\in[K]$.
		\STATE Update $
		w_k^{(t+1)}=n^{-1}\sum_{j \in [n]} r_{jk}^{(t)}$ for
		$k\in[K]$.
		\STATE Stop if the increase in $
		\ell_n(w)=\sum_{j \in [n]}\log\big\{\sum_{k \in [K]} w_k\varphi_{s_j}(U_j-\theta_k)\big\}$
		is smaller than \texttt{tol}.
		\ENDFOR
		\STATE \textbf{Output}: $\hat G_{\mathrm{grid}}=\sum_{k \in [K]} w_k^{(t)}\delta_{\theta_k}$.
	\end{algorithmic}
\end{algorithm}

\section{Proof of Lemma \ref{lem:Wasserstein_from_Hellinger}}\label{section:proof_Wasserstein_from_Hellinger}

	The proof is inspired by the method used in that of \cite[Theorem 2]{nguyen2013convergence}. Fix $K:\R\to (0,\infty)$ to be a symmetric density function such that (i) all moments of $K$ are finite, and (ii) its Fourier transform $\check{K}$ is supported in $[-1,1]$. Such a kernel can be constructed by starting from a nonnegative $C^\infty$ function $\check K$ supported on $[-1,1]$ and taking its inverse Fourier transform, which has tails decaying faster than any polynomial. 
	
	For any $\delta>0$, let $K_\delta(\cdot)\equiv \delta^{-1}K(\cdot/\delta)$. Note that for any $G \in \mathscr{G}$, by definition of Wasserstein distances, 
	\begin{align*}
	\mathsf{W}_p^p(G,G*K_\delta)\leq \E\abs{X-(X+\delta Z)}^p\leq c_K\delta^p,\qquad X\sim G,\ Z\sim K,\ X\perp Z.
	\end{align*}
	This means there exists some $c_1=c_1(p,K)>0$ such that for any $\delta>0$,
	\begin{align}\label{ineq:w_hellinger_1}
	\mathsf{W}_p(G_1,G_2)&\leq \mathsf{W}_p(G_1,G_1*K_\delta)+\mathsf{W}_p\big(G_1*K_\delta,G_2*K_\delta\big)+\mathsf{W}_p(G_2,G_2*K_\delta)\nonumber\\
	&\leq c_1\cdot \delta + \mathsf{W}_p\big(G_1*K_\delta,G_2*K_\delta\big).
	\end{align}
	Now we shall estimate the term $\mathsf{W}_p\big(G_1*K_\delta,G_2*K_\delta\big)$. By \cite[Theorem 6.15]{villani2009optimal}, 
	\begin{align}\label{ineq:w_hellinger_kernel_1}
	&\mathsf{W}_p^p\big(G_1*K_\delta,G_2*K_\delta\big) \leq c_p\cdot \int \abs{x}^p\,\d{ \abs{G_1*K_\delta-G_2*K_\delta} (x) }\nonumber\\
	&\leq c_p\cdot \max_{\ell \in [2]} \bigg(\int \abs{x}^{2p}\, \d{ \abs{G_\ell*K_\delta} (x) }\bigg)^{1/2}\cdot \pnorm{G_1*K_\delta-G_2*K_\delta}{L_1(\R)}^{1/2}.
	\end{align}
	As $\mathrm{supp}(G_\ell)\subset [\pm L]$ for $\ell \in [2]$, for any $\delta \in (0,1)$,
	\begin{align}\label{ineq:w_hellinger_kernel_2}
	\int \abs{x}^{2p}\, \d{ \abs{G_\ell*K_\delta} (x) } &= \E_{X\sim G_\ell,\;Y\sim K:\,X\perp Y} \abs{X+\delta Y}^{2p}\leq c_1\cdot L^{2p}.
	\end{align}
	Let $\varphi_\sigma$ be the p.d.f. corresponding to $\mathcal{N}(0,\sigma^2)$. For any $j \in [n]$, with $g_{\delta;j}$ denoting the inverse Fourier transform of $\check{K}_\delta/\check{\varphi}_{\sigma_{0,j}}$, for any $G \in \mathscr{G}$, we have 
	$G\ast K_\delta = G\ast \varphi_{\sigma_{0,j}}*g_{\delta;j} = \varphi_{G;\sigma_{0,j}}*g_{\delta;j}$.
	This means, in view of the simple $L_1$-$L_2$ estimate in \cite[Lemma 6-(2)]{nguyen2013convergence}, for some universal constant $c_0>0$,
	\begin{align}\label{ineq:w_hellinger_kernel_3}
	\pnorm{G_1*K_\delta-G_2*K_\delta}{L_1(\R)}^2&\leq c_K L^{c_0}\cdot \pnorm{G_1*K_\delta-G_2*K_\delta}{L_2(\R)} \nonumber\\
	&\leq c_K L^{c_0}\cdot d_{\mathrm{TV}}(\varphi_{G_1;\sigma_{0,j}},\varphi_{G_2;\sigma_{0,j}})\cdot \pnorm{g_{\delta;j}}{L_2(\R)}.
	\end{align}
	By Plancherel's theorem, we may estimate 
	\begin{align}\label{ineq:w_hellinger_kernel_4}
	\pnorm{g_{\delta;j}}{L_2(\R)}^2 &= \frac{1}{2\pi} \int  \frac{\check{K}_\delta^2(x)}{\check{\varphi}_{\sigma_{0,j}}^2(x)}\,\d{x} = \frac{1}{2\pi} \int  \frac{\check{K}^2(\delta x)}{\check{\varphi}_{\sigma_{0,j}}^2(x)}\,\d{x}\nonumber\\
	& \leq c_K \int_{-1/\delta}^{1/\delta} \frac{1}{\check{\varphi}_{\sigma_{0,j}}^2(x)}\,\d{x}\leq c_K \int_{-1/\delta}^{1/\delta} e^{c_0 \sigma_{0,j}^2 x^2}\,\d{x} \leq c_K e^{c_0 \sigma_{0,j}^2 \delta^{-2}}.
	\end{align}
	Combining (\ref{ineq:w_hellinger_kernel_3})-(\ref{ineq:w_hellinger_kernel_4}), averaging over $j \in [n]$ and using $d_{\mathrm{TV}}(\cdot,\cdot)\leq \sqrt{2}\cdot d_H(\cdot,\cdot)$,
	\begin{align}\label{ineq:w_hellinger_kernel_5}
	\pnorm{G_1*K_\delta-G_2*K_\delta}{L_1(\R)}^2&\leq c_K L^{c_0}\cdot e^{c_0 \pnorm{\sigma_0}{\infty}^2 \delta^{-2}}\cdot \E_{\pi_n} d_{\mathrm{TV}}(\varphi_{G_1;\sigma_{0,\pi_n}},\varphi_{G_2;\sigma_{0,\pi_n}})\nonumber\\
	&\leq c_K L^{c_0}\cdot e^{c_0 \pnorm{\sigma_0}{\infty}^2 \delta^{-2}}\cdot \mathfrak{d}_{H;\sigma_{0,[n]}}(G_1,G_2).
	\end{align}
	Now combining (\ref{ineq:w_hellinger_kernel_1}), (\ref{ineq:w_hellinger_kernel_2}) and (\ref{ineq:w_hellinger_kernel_5}),
	\begin{align*}
	\mathsf{W}_p^p\big(G_1*K_\delta,G_2*K_\delta\big) \leq L^{c_1}\cdot e^{c_0 \pnorm{\sigma_0}{\infty}^2 \delta^{-2}}\cdot \mathfrak{d}_{H;\sigma_{0,[n]}}^{1/4}(G_1,G_2).
	\end{align*}
	The claim now follows by combining the above display and (\ref{ineq:w_hellinger_1}). \qed

\section{Auxiliary results}

\begin{lemma}\label{lem:der_log_mix_est}
	For any $x \in \R$ and $\sigma>0$,
	\begin{align*}
	\biggabs{\frac{ \varphi_{G;\sigma}'(x) }{ \varphi_{G;\sigma}(x) }}\leq \frac{1}{\sigma}\cdot \sqrt{- \log\big(2\pi\sigma^2\cdot \varphi_{G;\sigma}^2(x)\big) }.
	\end{align*}
\end{lemma}
\begin{proof}
	The proof is a modification of some of the arguments presented in \cite[Lemma A.1]{jiang2009general}. For $X|\theta\sim \mathcal{N}(\theta,\sigma^2)$ and $\theta\sim G$, Tweedie's formula implies that
	\begin{align*}
	\E[\theta-X|X=x] = \sigma^2\cdot \frac{ \varphi_{G;\sigma}'(x) }{ \varphi_{G;\sigma}(x) },\quad \E[(\theta-X)^2|X=x] = \sigma^2+\sigma^4\cdot \frac{ \varphi_{G;\sigma}''(x) }{ \varphi_{G;\sigma}(x) }.
	\end{align*}
	This means that with $H(x)\equiv e^{x/(2\sigma^2)}$, by Jensen's inequality,
	\begin{align*}
	H\bigg[\sigma^4 \bigg(\frac{ \varphi_{G;\sigma}'(x) }{ \varphi_{G;\sigma}(x) }\bigg)^2\bigg]&= H\big(\big(\E[\theta-X|X=x]\big)^2\big) \\
	&\leq H\big(\E[(\theta-X)^2|X=x]\big)\leq \E\big[H\big((\theta-X)^2\big)|X=x\big].
	\end{align*}
	On the other hand,
	\begin{align*}
	\E\big[H\big((\theta-X)^2\big)|X=x\big]& = \frac{1}{ \varphi_{G;\sigma}(x) }\int e^{ \frac{(u-x)^2}{2\sigma^2} }\varphi_\sigma(x-u)\,G(\d u) = \frac{1}{\sqrt{2\pi\sigma^2}\cdot  \varphi_{G;\sigma}(x) }.
	\end{align*}
	Combining the above two displays, we have
	\begin{align*}
	\exp \bigg[ \frac{\sigma^2}{2} \bigg(\frac{ \varphi_{G;\sigma}'(x) }{ \varphi_{G;\sigma}(x) }\bigg)^2 \bigg]\leq \frac{1}{\sqrt{2\pi\sigma^2}\cdot  \varphi_{G;\sigma}(x) }.
	\end{align*}
	The claim follows. 
\end{proof}

The following version of the Gaussian concentration inequality allowing for `high probability bounded Lipschitz constant', is useful. 
\begin{lemma}\label{lem:gaussian_conc}
	Suppose (A1) holds for some $K\geq 2$. Let $H: \R^{m\times n}\to \R$ be a measurable map. Suppose there exist $q\geq 1$ and $\Lambda\geq 2$ such that the following hold:
	\begin{enumerate}
		\item For all $A \in \R^{m\times n}$, $\abs{H(A)}\leq \Lambda (1+\pnorm{A}{\op})^q$. 
		\item For all $A_1,A_2 \in \R^{m\times n}$,
		\begin{align*}
		\abs{H(A_1)-H(A_2)}\leq \Lambda \big(1+\pnorm{A_1}{\op}+\pnorm{A_2}{\op}\big)^q\cdot \pnorm{A_1-A_2}{\op}.
		\end{align*}
	\end{enumerate}
	Let the entries of $\mathsf{Z}_{m\times n} \in \R^{m\times n}$ be i.i.d. $\mathcal{N}(0,1)$. Then there exist some universal constant $c_0>1$ and another constant $c_q>1$ depending on $q>0$ such that if $x\geq   e^{-n/c_0}$, 
	\begin{align*}
	\Prob\big(\abs{H(n^{-1/2}\mathsf{Z}_{m\times n})-\E H(n^{-1/2}\mathsf{Z}_{m\times n})}\geq  K^{c_q} \Lambda x\big)\leq  c_0 e^{-n(1\wedge x^2)/c_0}
	\end{align*}
\end{lemma}
\begin{proof}
	For notational simplicity, we write $Z_n\equiv n^{-1/2}\mathsf{Z}_{m\times n}$ in the proof. The proof below modifies that of \cite[Lemma A.2]{bao2025leave} tailored for the symmetric setting. Fix $M>0$ to be chosen later. Let
	\begin{align*}
	\mathscr{H}(A)&\equiv \inf_{\substack{A' \in \R^{m\times n}, \pnorm{A'}{\op}\leq M}}\Big\{H(A')+\Lambda \cdot \Gamma(A,A')\big(\pnorm{A-A'}{\op}\wedge (2M)\big)\Big\},
	\end{align*}
	where 
	\begin{align*}
	\Gamma(A,A')\equiv\big( (\pnorm{A}{\op}\wedge M)+\pnorm{A'}{\op}+1\big)^q.
	\end{align*}
	Note that for any $A \in \R^{m\times n}$ with $\pnorm{A}{\op}\leq M$, the definition of $\mathscr{H}$ entails that $\mathscr{H}(A)\leq H(A)$. On the other hand, the condition (2) implies that $H(A)\leq \mathscr{H}(A)$. In summary, 
	\begin{align}\label{ineq:gaussian_conc_1}
	\mathscr{H}(A)=H(A) \hbox{ for all } A \in \R^{m\times n} \hbox{ such that }\pnorm{A}{\op}\leq M.
	\end{align}
	Next, for $A_1,A_2 \in \R^{m\times n}$,
	\begin{align*}
	&\abs{\mathscr{H}(A_1)-\mathscr{H}(A_2)}/\Lambda\\
	&\leq   \sup_{\substack{A' \in \R^{m\times n}, \pnorm{A'}{\op}\leq M}}\bigabs{\Gamma(A_1,A') \big(\pnorm{A_1-A'}{\op}\wedge 2M\big) - \Gamma(A_2,A') \big(\pnorm{A_2-A'}{\op}\wedge 2M\big)  }\\
	&\leq 2M\cdot \sup_{A' \in \R^{m\times n}, \pnorm{A'}{\op}\leq M}\abs{\Gamma(A_1,A')-\Gamma(A_2,A')}+(2M+1)^q\cdot \pnorm{A_1-A_2}{\op}\\
	&\leq 2q(2M+1)^q\cdot \pnorm{A_1-A_2}{\op}\leq 2q(2M+1)^q\cdot \pnorm{A_1-A_2}{F}. 
	\end{align*}
	Using Gaussian concentration inequality, for any $x>0, M>c' K$,
	\begin{align}\label{ineq:gaussian_conc_2}
	2\exp\bigg(-\frac{nx^2}{\Lambda^2 M^{cq}}\bigg)&\geq \Prob\big(\abs{\mathscr{H}(Z_n)-\E \mathscr{H}(Z_n)}\geq x\big)\nonumber\\
	&\geq \Prob\big(\abs{\mathscr{H}(Z_n)-\E \mathscr{H}(Z_n)}\geq x, \pnorm{Z_n}{\op}\leq M\big)\nonumber\\
	&\geq \Prob\big(\abs{H(Z_n)-\E \mathscr{H}(Z_n)}\geq x\big)-c e^{-n M^2/c}. 
	\end{align}
	Here in the last inequality we used (\ref{ineq:gaussian_conc_1}) and the subgaussian tail estimate for $\pnorm{Z_n}{\op}$, cf. \cite[Theorem 4.4.5]{vershynin2018high}. On the other hand, as $\E \mathscr{H}(Z_n)=\E H(Z_n)\bm{1}_{\pnorm{Z_n}{\op}\leq M}+\E \mathscr{H}(Z_n)\bm{1}_{\pnorm{Z_n}{\op}> M}$, using the condition (1), 
	\begin{align*}
	&\bigabs{\E \mathscr{H}(Z_n)-\E H(Z_n)}\leq \E\big(\abs{\mathscr{H}(Z_n)}+\abs{H(Z_n)}\big)\bm{1}_{\pnorm{Z_n}{\op}\geq M}\\
	&\leq \E^{1/2}\big(\abs{H(0)}+\Lambda (1+\pnorm{Z_n}{\op})^{q+1}+\abs{H(Z_n)}\big)^2\cdot \Prob^{1/2}(\pnorm{Z_n}{\op}\geq M)\\
	&\leq c \Lambda \E(1+\pnorm{Z_n}{\op})^{cq }\cdot e^{-nM^2/c}\leq \Lambda K^{c q} e^{-n M^2/c }.
	\end{align*}
	So for any $x>0, M>c' K$ such that $x\geq \Lambda K^{c q} e^{-n M^2/c}$, 
	\begin{align*}
	\Prob\big(\abs{H(A)-\E H(A)}\geq 2x\big)\leq c \exp\bigg(-n\cdot \bigg\{\frac{x^2}{c \Lambda^2 M^{c q}}\wedge \frac{M^2}{c}\bigg\} \bigg),
	\end{align*}
	proving the subgaussian estimate by choosing $M=c_2 K$  for a large universal $c_2>1$, and adjusting the constant. 
\end{proof}

\begin{lemma}
	\label{lem:two_point_prob_metric}
	Let $(\Theta,d)$ be a metric space, and let $\{P_\theta:\theta\in\Theta\}$ be a statistical experiment. 
	For any two points $\theta_0,\theta_1\in\Theta$, write $
	\Delta \equiv d(\theta_0,\theta_1)$. 
	Then if $s>0$ and $d(\theta_0,\theta_1)\ge 2s$, then $
	\inf_{\hat\theta}
	\sup_{i=0,1}
	\Prob_{\theta_i}\big(
	d(\hat\theta,\theta_i)\ge s
	\big)
	\ge
	(1-d_{\TV}(P_{\theta_0},P_{\theta_1}))/2$. 
\end{lemma}

\begin{proof}
	This is the standard two-point testing result. As we did not locate an exact reference, some details are included below. 
	
	Fix an arbitrary estimator $\hat\theta$. Define the induced test $
	\psi_{\hat\theta}
	\in
	\argmin_{i\in\{0,1\}}d(\hat\theta,\theta_i)$, 
	with ties broken arbitrarily. Suppose $
	d(\hat\theta,\theta_i)<\Delta/2$, then by the triangle inequality, $
	d(\hat\theta,\theta_{1-i})
	\ge
	d(\theta_0,\theta_1)-d(\hat\theta,\theta_i)
	>\Delta/2$.
	Hence $\psi_{\hat\theta}=i$. Therefore $
	\{\psi_{\hat\theta}\neq i\}
	\subset
	\big\{
	d(\hat\theta,\theta_i)\ge \Delta/2
	\big\}$. 
	It follows that
	\begin{align*}
	\sup_{i=0,1}
	\Prob_{\theta_i}\big(
	d(\hat\theta,\theta_i)\ge \Delta/2
	\big)
	\ge
	\sup_{i=0,1}\Prob_{\theta_i}(\psi_{\hat\theta}\neq i)\geq \inf_{\psi}
	\sup_{i=0,1}\Prob_{\theta_i}(\psi\neq i).
	\end{align*}
	By the total-variation form of Le Cam's two-point testing bound
	\cite[Theorem 2.2]{tsybakov2008introduction}, the right hand side of the above display is bounded from below by $(1-d_{\TV}(P_{\theta_0},P_{\theta_1}))/2$. The claim follows. 
\end{proof}

\begin{lemma}\label{lem:quantile_inversion}
	Let $F$ and $F_0$ be cumulative distribution functions on $\mathbb R$, and for
	$a\in(0,1)$ define their left quantiles by
	\begin{align*}
	Q_F(a)\equiv \inf\{t\in\mathbb R:F(t)\ge a\},
	\qquad
	Q_{F_0}(a)\equiv \inf\{t\in\mathbb R:F_0(t)\ge a\}.
	\end{align*}
	Suppose the following hold:
	\begin{enumerate}
		\item $F_0$ is supported on an interval $I=[\ell,r]$ and has a density
		$f_0$ on $I$ satisfying $
		\inf_{t \in I} f_0(t)\ge m_0$
		for some $m_0>0$. 
		\item Let $\eta\in(0,1/2)$ and suppose $
		\sup_{t\in\mathbb R}|F(t)-F_0(t)|\le \Delta$ for some $
		0<\Delta\le \eta/2$. 
	\end{enumerate}
	Then, uniformly over $a\in[\eta,1-\eta]$,
	\begin{align*}
	\abs{Q_F(a)-Q_{F_0}(a)}
	\le {2\Delta/m_0}.
	\end{align*}
\end{lemma}

\begin{proof}
	Fix $a\in[\eta,1-\eta]$ and write $
	q_0\equiv Q_{F_0}(a)$ and  $\rho\equiv {2\Delta/m_0}$. 
	Since $F_0$ is supported on $[\ell,r]$ and has density bounded below by
	$m_0>0$ throughout $[\ell,r]$, the function $F_0$ is continuous and strictly
	increasing on $[\ell,r]$. Hence $q_0\in[\ell,r]$ and $
	F_0(q_0)=a$.
	
	We first prove the upper bound $Q_F(a)\le q_0+\rho$. We consider two cases:
	\begin{itemize}
		\item  If $q_0+\rho\le r$,
		then, by the density lower bound, $
		F_0(q_0+\rho)-F_0(q_0)
		=
		\int_{q_0}^{q_0+\rho} f_0(t)\,\d t
		\ge m_0\rho
		=
		2\Delta$, and therefore $
		F(q_0+\rho)
		\ge
		F_0(q_0+\rho)-\Delta
		\ge
		a+2\Delta-\Delta
		=
		a+\Delta
		\ge a$. 
		By the definition of the left quantile, this implies $
		Q_F(a)\le q_0+\rho$.
		\item If $q_0+\rho>r$, then as $\mathrm{supp}(F_0)=[\ell,r]$, we have $
		F(r)\ge F_0(r)-\Delta = 1-\Delta$. 
		Since $a\le 1-\eta$ and $\Delta\le \eta/2$, we have $
		1-\Delta \ge 1-\eta/2 \ge 1-\eta \ge a$. 
		Thus $F(r)\ge a$, and hence $Q_F(a)\le r<q_0+\rho$. 
	\end{itemize}
	The upper bound follows.

	We next prove the lower bound $Q_F(a)\ge q_0-\rho$. The arguments are similar:
	\begin{itemize}
		\item If $q_0-\rho\ge \ell$,
		then again by the density lower bound, $
		F_0(q_0)-F_0(q_0-\rho)
		=
		\int_{q_0-\rho}^{q_0} f_0(t)\,\d t
		\ge m_0\rho
		=
		2\Delta$. 
		Hence $
		F(q_0-\rho)
		\le
		F_0(q_0-\rho)+\Delta
		\le
		a-2\Delta+\Delta
		=
		a-\Delta
		<a$. Since $F$ is nondecreasing, $F(t)<a$ for every $t\le q_0-\rho$, and therefore $
		Q_F(a)>q_0-\rho$. 
		\item If $q_0-\rho<\ell$, then for every $t<\ell$ we have $F_0(t)=0$, so $
		F(t)\le \Delta\le \eta/2<a$ for $t<\ell$. Consequently $Q_F(a)\ge \ell>q_0-\rho$. 
	\end{itemize}
	The lower bound follows.
\end{proof}

\section*{Acknowledgments}
 The research of Q. Han is partially supported by NSF grant DMS-2143468.

\bibliographystyle{alpha}
\bibliography{mybib}

\newcommand{\etalchar}[1]{$^{#1}$}
\begin{thebibliography}{GTM{\etalchar{+}}24}

\bibitem[BHX25]{bao2025leave}
Zhigang Bao, Qiyang Han, and Xiaocong Xu.
\newblock A leave-one-out approach to approximate message passing.
\newblock {\em Ann. Appl. Probab.}, 35(4):2716--2766, 2025.

\bibitem[CB07]{chandler2007inference}
Richard~E. Chandler and Steven Bate.
\newblock Inference for clustered data using the independence loglikelihood.
\newblock {\em Biometrika}, 94(1):167--183, 2007.

\bibitem[CCFM19]{chen2019gradient}
Yuxin Chen, Yuejie Chi, Jianqing Fan, and Cong Ma.
\newblock Gradient descent with random initialization: fast global convergence
  for nonconvex phase retrieval.
\newblock {\em Math. Program.}, 176(1-2):5--37, 2019.

\bibitem[CCM21]{celentano2021high}
Michael Celentano, Chen Cheng, and Andrea Montanari.
\newblock The high-dimensional asymptotics of first order methods with random
  data.
\newblock {\em arXiv preprint arXiv:2112.07572}, 2021.

\bibitem[CDI26]{chen2026normal}
Jiafeng Chen, Nabarun Deb, and Nikolaos Ignatiadis.
\newblock Normal approximations in nonparametric empirical {B}ayes.
\newblock {\em arXiv preprint arXiv:2605.31599}, 2026.

\bibitem[CDP15]{chen2015improved}
Wei-Kuo Chen, Nikos Dafnis, and Grigoris Paouris.
\newblock Improved {H}\"older and reverse {H}\"older inequalities for
  {G}aussian random vectors.
\newblock {\em Adv. Math.}, 280:643--689, 2015.

\bibitem[CGS11]{chen2011normal}
Louis H.~Y. Chen, Larry Goldstein, and Qi-Man Shao.
\newblock {\em Normal approximation by {S}tein's method}.
\newblock Probability and its Applications (New York). Springer, Heidelberg,
  2011.

\bibitem[Che26]{chen2026empirical}
Jiafeng Chen.
\newblock Empirical {B}ayes when estimation precision predicts parameters.
\newblock {\em Econometrica}, 94(2):305--340, 2026.

\bibitem[Cre15]{cressie2015statistics}
Noel A.~C. Cressie.
\newblock {\em Statistics for spatial data}.
\newblock Wiley Classics Library. John Wiley \& Sons, Inc., New York, revised
  edition, 2015.

\bibitem[CW26]{chen2026sharp}
Jiafeng Chen and Yihong Wu.
\newblock Sharp regret-hellinger bounds for {G}aussian empirical {B}ayes via
  polynomial approximation.
\newblock {\em arXiv preprint arXiv:2605.02070}, 2026.

\bibitem[DM13]{dedecker2013minimax}
J\'er\^ome Dedecker and Bertrand Michel.
\newblock Minimax rates of convergence for {W}asserstein deconvolution with
  supersmooth errors in any dimension.
\newblock {\em J. Multivariate Anal.}, 122:278--291, 2013.

\bibitem[DSGS22]{deb2022two}
Nabarun Deb, Sujayam Saha, Adityanand Guntuboyina, and Bodhisattva Sen.
\newblock Two-component mixture model in the presence of covariates.
\newblock {\em J. Amer. Statist. Assoc.}, 117(540):1820--1834, 2022.

\bibitem[DZ16]{dicker2016high}
Lee~H. Dicker and Sihai~D. Zhao.
\newblock High-dimensional classification via nonparametric empirical {B}ayes
  and maximum likelihood inference.
\newblock {\em Biometrika}, 103(1):21--34, 2016.

\bibitem[Efr10]{efron2010large}
Bradley Efron.
\newblock {\em Large-scale inference}, volume~1 of {\em Institute of
  Mathematical Statistics (IMS) Monographs}.
\newblock Cambridge University Press, Cambridge, 2010.
\newblock Empirical Bayes methods for estimation, testing, and prediction.

\bibitem[Efr14]{efron2014two}
Bradley Efron.
\newblock Two modeling strategies for empirical {B}ayes estimation.
\newblock {\em Statist. Sci.}, 29(2):285--301, 2014.

\bibitem[EM72]{efron1972limiting}
Bradley Efron and Carl Morris.
\newblock Limiting the risk of {B}ayes and empirical {B}ayes estimators. {II}.
  {T}he empirical {B}ayes case.
\newblock {\em J. Amer. Statist. Assoc.}, 67:130--139, 1972.

\bibitem[FD18]{feng2018approximate}
Long Feng and Lee~H. Dicker.
\newblock Approximate nonparametric maximum likelihood for mixture models: a
  convex optimization approach to fitting arbitrary multivariate mixing
  distributions.
\newblock {\em Comput. Statist. Data Anal.}, 122:80--91, 2018.

\bibitem[FGSW23]{fan2023gradient}
Zhou Fan, Leying Guan, Yandi Shen, and Yihong Wu.
\newblock Gradient flows for empirical bayes in high-dimensional linear models.
\newblock {\em arXiv preprint arXiv:2312.12708}, 2023.

\bibitem[FKL{\etalchar{+}}25]{fan2025dynamical}
Zhou Fan, Justin Ko, Bruno Loureiro, Yue~M Lu, and Yandi Shen.
\newblock Dynamical mean-field analysis of adaptive {L}angevin diffusions:
  Replica-symmetric fixed point and empirical {B}ayes.
\newblock {\em arXiv preprint arXiv:2504.15558}, 2025.

\bibitem[GIKL25]{ghosh2025stein}
Sulagna Ghosh, Nikolaos Ignatiadis, Frederic Koehler, and Amber Lee.
\newblock Stein's unbiased risk estimate and hyv$\backslash$" arinen's score
  matching.
\newblock {\em arXiv preprint arXiv:2502.20123}, 2025.

\bibitem[GK16]{gu2016problem}
Jiaying Gu and Roger Koenker.
\newblock On a problem of {R}obbins.
\newblock {\em International Statistical Review}, 84(2):224--244, 2016.

\bibitem[GK22]{gu2022ranking}
Jiaying Gu and Roger Koenker.
\newblock Ranking and selection from pairwise comparisons: empirical bayes
  methods for citation analysis.
\newblock In {\em AEA Papers and Proceedings}, volume 112, pages 624--629.
  American Economic Association 2014 Broadway, Suite 305, Nashville, TN 37203,
  2022.

\bibitem[GTM{\etalchar{+}}24]{gerbelot2024rigorous}
C\'{e}dric Gerbelot, Emanuele Troiani, Francesca Mignacco, Florent Krzakala,
  and Lenka Zdeborov\'{a}.
\newblock Rigorous {D}ynamical {M}ean-{F}ield {T}heory for {S}tochastic
  {G}radient {D}escent {M}ethods.
\newblock {\em SIAM J. Math. Data Sci.}, 6(2):400--427, 2024.

\bibitem[GvdV01]{ghosal2001entropies}
Subhashis Ghosal and Aad~W. van~der Vaart.
\newblock Entropies and rates of convergence for maximum likelihood and {B}ayes
  estimation for mixtures of normal densities.
\newblock {\em Ann. Statist.}, 29(5):1233--1263, 2001.

\bibitem[GZ22]{guo2022extreme}
Zijian Guo and Cun-Hui Zhang.
\newblock Extreme eigenvalues of nonlinear correlation matrices with
  applications to additive models.
\newblock {\em Stochastic Process. Appl.}, 150:1037--1058, 2022.

\bibitem[Han25a]{han2025entrywise}
Qiyang Han.
\newblock Entrywise dynamics and universality of general first order methods.
\newblock {\em Ann. Statist.}, 53(4):1783--1807, 2025.

\bibitem[Han25b]{han2025long}
Qiyang Han.
\newblock Long-time dynamics and universality of nonconvex gradient descent.
\newblock {\em arXiv preprint arXiv:2509.11426}, 2025.

\bibitem[HX26]{han2026gradient}
Qiyang Han and Xiaocong Xu.
\newblock Gradient descent inference in empirical risk minimization.
\newblock {\em Ann. Statist., to appear. Available at arXiv:2412.09498}, 2026.

\bibitem[IK26]{ignatiadis2026compound}
Nikolaos Ignatiadis and Sid Kankanala.
\newblock Compound decisions and empirical bayes via bayesian nonparametrics.
\newblock {\em arXiv preprint arXiv:2602.20115}, 2026.

\bibitem[IS25]{ignatiadis2025empirical}
Nikolaos Ignatiadis and Bodhisattva Sen.
\newblock Empirical partially {B}ayes multiple testing and compound {$\chi^2$}
  decisions.
\newblock {\em Ann. Statist.}, 53(1):1--36, 2025.

\bibitem[Jia20]{jiang2020general}
Wenhua Jiang.
\newblock On general maximum likelihood empirical {B}ayes estimation of
  heteroscedastic {IID} normal means.
\newblock {\em Electron. J. Stat.}, 14(1):2272--2297, 2020.

\bibitem[JZ09]{jiang2009general}
Wenhua Jiang and Cun-Hui Zhang.
\newblock General maximum likelihood empirical {B}ayes estimation of normal
  means.
\newblock {\em Ann. Statist.}, 37(4):1647--1684, 2009.

\bibitem[KC22]{kuchibhotla2022moving}
Arun~Kumar Kuchibhotla and Abhishek Chakrabortty.
\newblock Moving beyond sub-{G}aussianity in high-dimensional statistics:
  applications in covariance estimation and linear regression.
\newblock {\em Inf. Inference}, 11(4):1389--1456, 2022.

\bibitem[KCSA20]{kim2020fast}
Youngseok Kim, Peter Carbonetto, Matthew Stephens, and Mihai Anitescu.
\newblock A fast algorithm for maximum likelihood estimation of mixture
  proportions using sequential quadratic programming.
\newblock {\em J. Comput. Graph. Statist.}, 29(2):261--273, 2020.

\bibitem[KG17]{koenker2017rebayes}
Roger Koenker and Jiaying Gu.
\newblock {REB}ayes: an {R} package for empirical bayes mixture methods.
\newblock {\em Journal of Statistical Software}, 82:1--26, 2017.

\bibitem[KG25]{koenker2025empirical}
Roger Koenker and Jiaying Gu.
\newblock {\em Empirical Bayes: Some Tools, Rules, and Duals}.
\newblock Cambridge University Press, 2025.

\bibitem[KM14]{koenker2014convex}
Roger Koenker and Ivan Mizera.
\newblock Convex optimization, shape constraints, compound decisions, and
  empirical {B}ayes rules.
\newblock {\em J. Amer. Statist. Assoc.}, 109(506):674--685, 2014.

\bibitem[KS26]{kim2026empirical}
Taehyun Kim and Bodhisattva Sen.
\newblock Empirical bayes estimation and inference via smooth nonparametric
  maximum likelihood.
\newblock {\em arXiv preprint arXiv:2603.27843}, 2026.

\bibitem[KW56]{kiefer1956consistency}
J.~Kiefer and J.~Wolfowitz.
\newblock Consistency of the maximum likelihood estimator in the presence of
  infinitely many incidental parameters.
\newblock {\em Ann. Math. Statist.}, 27:887--906, 1956.

\bibitem[KWCS24]{kim2024flexible}
Youngseok Kim, Wei Wang, Peter Carbonetto, and Matthew Stephens.
\newblock A flexible empirical {B}ayes approach to multiple linear regression,
  and connections with penalized regression.
\newblock {\em J. Mach. Learn. Res.}, 25:Paper No. [185], 59, 2024.

\bibitem[LD26]{lee2026parametric}
Seunghyun Lee and Nabarun Deb.
\newblock Parametric mean-field empirical {B}ayes in high-dimensional linear
  regression.
\newblock {\em arXiv preprint arXiv:2601.16842}, 2026.

\bibitem[Lin83]{lindsay1983geometry}
Bruce~G. Lindsay.
\newblock The geometry of mixture likelihoods: a general theory.
\newblock {\em Ann. Statist.}, 11(1):86--94, 1983.

\bibitem[MSS23]{mukherjee2023mean}
Sumit Mukherjee, Bodhisattva Sen, and Subhabrata Sen.
\newblock A mean field approach to empirical bayes estimation in
  high-dimensional linear regression.
\newblock {\em arXiv preprint arXiv:2309.16843}, 2023.

\bibitem[Ngu13]{nguyen2013convergence}
XuanLong Nguyen.
\newblock Convergence of latent mixing measures in finite and infinite mixture
  models.
\newblock {\em Ann. Statist.}, 41(1):370--400, 2013.

\bibitem[PW20]{polyanskiy2020self}
Yury Polyanskiy and Yihong Wu.
\newblock Self-regularizing property of nonparametric maximum likelihood
  estimator in mixture models.
\newblock {\em arXiv preprint arXiv:2008.08244}, 2020.

\bibitem[Rob51]{robbins1951asymptotically}
Herbert Robbins.
\newblock Asymptotically subminimax solutions of compound statistical decision
  problems.
\newblock In {\em Proceedings of the {S}econd {B}erkeley {S}ymposium on
  {M}athematical {S}tatistics and {P}robability, 1950}, pages 131--148. Univ.
  California Press, Berkeley-Los Angeles, Calif., 1951.

\bibitem[Rob56]{robbins1956empirical}
Herbert Robbins.
\newblock An empirical {B}ayes approach to statistics.
\newblock In {\em Proceedings of the {T}hird {B}erkeley {S}ymposium on
  {M}athematical {S}tatistics and {P}robability, 1954--1955, vol. {I}}, pages
  157--163. Univ. California Press, Berkeley-Los Angeles, Calif., 1956.

\bibitem[RV13]{rudelson2013hanson}
Mark Rudelson and Roman Vershynin.
\newblock Hanson-{W}right inequality and sub-{G}aussian concentration.
\newblock {\em Electron. Commun. Probab.}, 18:no. 82, 9, 2013.

\bibitem[SG20]{saha2020nonparametric}
Sujayam Saha and Adityanand Guntuboyina.
\newblock On the nonparametric maximum likelihood estimator for {G}aussian
  location mixture densities with application to {G}aussian denoising.
\newblock {\em Ann. Statist.}, 48(2):738--762, 2020.

\bibitem[SGS25]{soloff2025multivariate}
Jake~A. Soloff, Adityanand Guntuboyina, and Bodhisattva Sen.
\newblock Multivariate, heteroscedastic empirical {B}ayes via nonparametric
  maximum likelihood.
\newblock {\em J. R. Stat. Soc. Ser. B. Stat. Methodol.}, 87(1):1--32, 2025.

\bibitem[SL03]{seber2003linear}
George A.~F. Seber and Alan~J. Lee.
\newblock {\em Linear regression analysis}.
\newblock Wiley Series in Probability and Statistics. Wiley-Interscience [John
  Wiley \& Sons], Hoboken, NJ, second edition, 2003.

\bibitem[Ste56]{stein1956inadmissibility}
Charles Stein.
\newblock Inadmissibility of the usual estimator for the mean of a multivariate
  normal distribution.
\newblock In {\em Proceedings of the {T}hird {B}erkeley {S}ymposium on
  {M}athematical {S}tatistics and {P}robability, 1954--1955, vol. {I}}, pages
  197--206. Univ. California Press, Berkeley-Los Angeles, Calif., 1956.

\bibitem[SW26]{shen2026poisson}
Yandi Shen and Yihong Wu.
\newblock Poisson empirical {B}ayes estimation: when does {$g$}-modeling beat
  {$f$}-modeling in theory (and in practice)?
\newblock {\em Ann. Statist.}, 54(1):146--175, 2026.

\bibitem[TM24]{tang2024empirical}
Yiqi Tang and Ryan Martin.
\newblock Empirical {B}ayes inference in sparse high-dimensional generalized
  linear models.
\newblock {\em Electron. J. Stat.}, 18(2):3212--3246, 2024.

\bibitem[Tsy09]{tsybakov2008introduction}
Alexandre~B. Tsybakov.
\newblock {\em Introduction to {N}onparametric {E}stimation}.
\newblock Springer Series in Statistics. Springer, New York, 2009.
\newblock Revised and extended from the 2004 French original, Translated by
  Vladimir Zaiats.

\bibitem[vdG00]{van2000empirical}
Sara van~de Geer.
\newblock {\em Applications of {E}mpirical {P}rocess {T}heory}, volume~6 of
  {\em Cambridge Series in Statistical and Probabilistic Mathematics}.
\newblock Cambridge University Press, Cambridge, 2000.

\bibitem[vdVW96]{van1996weak}
Aad van~der Vaart and Jon~A. Wellner.
\newblock {\em Weak {C}onvergence and {E}mpirical {P}rocesses}.
\newblock Springer Series in Statistics. Springer-Verlag, New York, 1996.

\bibitem[Ver18]{vershynin2018high}
Roman Vershynin.
\newblock {\em High-dimensional probability: An introduction with applications
  in data science}, volume~47 of {\em Cambridge Series in Statistical and
  Probabilistic Mathematics}.
\newblock Cambridge University Press, Cambridge, 2018.

\bibitem[Vil09]{villani2009optimal}
C\'edric Villani.
\newblock {\em Optimal transport}, volume 338 of {\em Grundlehren der
  mathematischen Wissenschaften [Fundamental Principles of Mathematical
  Sciences]}.
\newblock Springer-Verlag, Berlin, 2009.
\newblock Old and new.

\bibitem[VRF11]{varin2011overview}
Cristiano Varin, Nancy Reid, and David Firth.
\newblock An overview of composite likelihood methods.
\newblock {\em Statist. Sinica}, 21(1):5--42, 2011.

\bibitem[WCS21]{willwerscheid2021ebnm}
Jason Willwerscheid, Peter Carbonetto, and Matthew Stephens.
\newblock {EBNM}: An {R} package for solving the empirical {B}ayes normal means
  problem using a variety of prior families.
\newblock {\em arXiv preprint arXiv:2110.00152}, 2021.

\bibitem[WY20]{wu2020optimalestimation}
Yihong Wu and Pengkun Yang.
\newblock Optimal estimation of {G}aussian mixtures via denoised method of
  moments.
\newblock {\em Ann. Statist.}, 48(4):1981--2007, 2020.

\bibitem[Zha09]{zhang2009generalized}
Cun-Hui Zhang.
\newblock Generalized maximum likelihood estimation of normal mixture
  densities.
\newblock {\em Statist. Sinica}, 19(3):1297--1318, 2009.

\bibitem[Zha26]{zhang2026empirical}
Cun-Hui Zhang.
\newblock Empirical {B}ayes for dependent data.
\newblock {\em Preprint}, 2026.

\end{thebibliography}

\end{document}